\documentclass[12pt]{article}
\usepackage[utf8]{inputenc}

\usepackage{geometry} 
\usepackage[font=footnotesize,labelfont=bf]{caption}

\usepackage[backend=biber,giveninits=true,url=false, doi=false, isbn=false]{biblatex}
\usepackage{amsfonts}
\usepackage{amsmath}
\usepackage{amssymb}
\usepackage{amsthm}
\usepackage{mathtools}
\usepackage{xfrac}
\usepackage{spot}
\usepackage{tikz-cd}
\usepackage{subcaption}
\usepackage{hyperref}

\usepackage{pgf,tikz}
\usetikzlibrary{arrows}
\usetikzlibrary[patterns]

\newcommand{\ol}{\overline}
\newcommand{\wt}{\widetilde}

\newcommand{\inv}{^{-1}}

\newcommand{\R}{\mathbb{R}}

\newcommand{\N}{\mathbb{N}}

\newcommand{\CT}{\mathcal{T}}
\newcommand{\CB}{\mathcal{B}}
\newcommand{\CP}{\mathcal{P}}
\newcommand{\CR}{\mathcal{R}}
\newcommand{\CL}{\mathcal{L}}

\newtheorem{theorem}{Theorem}
\newtheorem*{theorem*}{Main Theorem}
\newtheorem{lemma}[theorem]{Lemma}
\newtheorem{corollary}[theorem]{Corollary}
\newtheorem{proposition}[theorem]{Proposition}

\newtheorem{remark}[theorem]{Remark}

\theoremstyle{definition}
\newtheorem{definition}[theorem]{Definition}

\title{Uniqueness of Billiard Coding in Polygons}
\author{Li Yunzhe}

\begin{document}

	\definecolor{zzttqq}{rgb}{0.6,0.2,0.}
	\definecolor{uuuuuu}{rgb}{0.26666666666666666,0.26666666666666666,0.26666666666666666}
	\definecolor{qqwuqq}{rgb}{0.,0.39215686274509803,0.}
	\definecolor{ffffff}{rgb}{1.,1.,1.}
	\definecolor{wrwrwr}{rgb}{0.3803921568627451,0.3803921568627451,0.3803921568627451}
	\definecolor{sqsqsq}{rgb}{0.3803921568627451,0.3803921568627451,0.3803921568627451}
	\definecolor{ududff}{rgb}{0.08235294117647059,0.396078431372549,0.7529411764705882}
	\definecolor{xdxdff}{rgb}{0.08235294117647059,0.396078431372549,0.7529411764705882}

	\maketitle

	\begin{abstract}
		We consider polygonal billiards and we show the uniqueness of coding of non-periodic billiard trajectories in polygons whose holes have non-zero minimal diameters, generalising a theorem of Galperin, Krüger and Troubetzkoy.
	\end{abstract}

	\section{Introduction} 
	\label{sec:introduction}
	
	The study of mathematical billiards is a rich subject in dynamical systems. It describes the frictionless motion of a mass point in a domain with elastic reflection on the boundary. Among many other variations of mathematical billiards, polygonal billiards concerns the billiard problem in a two dimensional polygonal domain. 

	Let $Q \subset \R^2$ be a polygon. 	The set $TQ$ consists of couples $(x, v)$ where $x$ is a point in $Q$ and $v$ a unit vector tangent to $x$ representing the position and direction of the billiard ball respectively. For a time $t \geq 0$ and an initial state $(x, v) \in TQ$ with $v$ pointing strictly into $Q$, the \textit{billiard flow} $\phi_t$ associates $(x, v)$ with a point $\phi_t(x,v)$ in $TQ$ obtained by moving $x$ along a straight line in the direction of $v$ with unit velocity until the moment $t$. Whenever the trajectory enters the boundary $\partial Q$ before time $t$, the direction $v$ is reflected by Descartes' law (the angle of incident is equal to the angle of reflection). This motion is determined for all positive time $t$ unless the flow reaches a vertex of the polygon. 
	
	One particular strategy to study polygonal billiards is to encode the trajectory of a billiard ball by the sequence of sides of the polygonal domain hit by the billiard ball along its trajectory.
	Let the set of edges of $Q$ be labelled by a finite set $\mathcal{A}$. As illustrated in Figure \ref{fig:non_simply_connected}, each infinite billiard trajectory $\{\phi_t(x,v)\}_{t \geq 0}$ can be coded by a sequence $\alpha \in \mathcal{A}^\mathbb{N}$ labelling the edges that the trajectory hits in order. Conversely, if $\alpha \in \mathcal{A}^\mathbb{N}$ is a sequence, we define $X(\alpha)$ to be the set of couples $(x,v) \in TQ$ with $x$ lying on $\partial Q$ such that the trajectory from $(x,v)$ is coded by $\alpha$. This article aims to prove the following theorem.

	\begin{theorem*}\label{thm:main theorem}
		Let $Q$ be a polygon such that all the holes of $Q$ have non-zero minimal diameters. Suppose the edges of $Q$ are indexed by $\mathcal{A}$. Then for any non-periodic sequence $\alpha \in \mathcal{A}^\mathbb{N}$, the set $X(\alpha)$ contains at most one point.
	\end{theorem*}

	This result in the case of simply connected polygons was proven in \cite{Galperin1995}. This classical result has applications for example in \cite{BT11} and \cite{BT12}, where it was shown that the shape of a simply connected polygonal domain satisfying certain conditions could be uniquely determined, up to similarity, by the encoding of certain billiard trajectories. Other applications can be found in articles such as \cite{BT14}, \cite{duchin2018you} and \cite{calderon2018hear}.  
	
	The proof in \cite{Galperin1995} relies on a particular property of simply connected polygon that $X(\alpha)$ consists of so-called `parallel phase points' whose base points form a single interval on an edge. If $Q$ is not simply connected, however, the base points of $X(\alpha)$ may form more than one intervals due to the presence of `holes' in $Q$ and consequently the method in \cite{Galperin1995} is not directly applicable to this case. Thus, we prove the main theorem for a more general class of polygons using a modified method. The main new idea in our approach is a way of coding parallel billiard trajectories (the $\mathcal{B}$-codings introduced in Section \ref{sec:encoding parallel trajectories}) which captures certain information about the positions of holes relative to the trajectories being coded. 

	This article is structured as follows. In Section \ref{sec:basic definitions}, we state the basic definitions of the dynamical system associated with polygonal billiards, and recall the technique of `trajectory unfolding' - a useful visualisation tool for later proofs. From Section \ref{sec:partition and encoding} to Section \ref{sec:encoding parallel trajectories}, we set up various tools for the proof of the main theorem. In particular, Section \ref{sec:partition and encoding} reviews some classical methods to code a billiard trajectory and introduces a partition of the phase space adapted to non-simply connected polygons. In Section \ref{sec:generalised trajectories}, we define `generalised trajectories' as the limits of physical billiard trajectories, and discuss some of their properties. In Section \ref{sec:encoding parallel trajectories}, we introduce a new coding method adapted to parallel trajectories in non-simply connected polygons. Especially, we prove Lemma \ref{lemma:approximate_by_geodesic} - a key result enabling the construction of generalised trajectories using the limit of trajectory codings. Finally, we put everything together in Section \ref{sec:the proof} to prove the main theorem.

	\bigskip

	\noindent \textbf{Acknowledgement.} \quad I would like to thank Prof. Serge Troubetzkoy for his guidance and his insightful suggestions during many discussions, without which this work would not have been possible. I would also like to thank \textit{Institut de Mathématiques de Marseille} for its conducive environment and its support during my stay. Special thanks should also go to the referee whose suggestions had substantially improved the clarity of this article.

	\begin{figure}[htbp]
		\centering
		\begin{tikzpicture}[line cap=round,line join=round,>=triangle 45,x=1.0cm,y=1.0cm]
			\clip(0.057263554898884905,-2.2750066544163365) rectangle (9.2789328768408,3.9452062313706797);
			\fill[line width=2.pt,color=zzttqq,fill=zzttqq,fill opacity=0.10000000149011612] (0.52,1.78) -- (2.88,3.66) -- (7.08,2.22) -- (4.3,1.12) -- (5.92,-1.6) -- (1.6,-0.7) -- cycle;
			\fill[line width=2.pt,color=ffffff,fill=ffffff,fill opacity=1.0] (2.94,2.08) -- (1.82,1.32) -- (3.3,0.48) -- cycle;
			\draw [line width=0.8pt] (1.6,-0.7)-- (0.52,1.78);
			\draw [line width=0.8pt] (0.52,1.78)-- (2.88,3.66);
			\draw [line width=0.8pt] (2.88,3.66)-- (7.08,2.22);
			\draw [line width=0.8pt] (7.08,2.22)-- (4.3,1.12);
			\draw [line width=0.8pt] (4.3,1.12)-- (5.92,-1.6);
			\draw [line width=0.8pt] (5.92,-1.6)-- (1.6,-0.7);
			\draw [line width=0.8pt] (1.82,1.32)-- (3.3,0.48);
			\draw [line width=0.8pt] (3.3,0.48)-- (2.94,2.08);
			\draw [line width=0.8pt] (2.94,2.08)-- (1.82,1.32);
			\draw [line width=0.8pt, dotted] (3.3514808652246257,-1.064891846921797)-- (3.8729875821767714,3.31954711468225);
			\draw [line width=0.8pt, dotted] (3.8729875821767714,3.31954711468225)-- (2.990644613281017,1.8549128298621465);
			\draw [line width=0.8pt, dotted] (2.990644613281017,1.8549128298621465)-- (4.432115479116965,0.8981764795073149);
			\draw [line width=0.8pt, dotted] (4.432115479116965,0.8981764795073149)-- (4.632233426850734,-1.3317152972605708);
			\draw [line width=0.8pt, dotted] (4.632233426850734,-1.3317152972605708)-- (5.172077339215268,-0.3442286189293386);
			\draw [line width=0.8pt, dotted] (5.172077339215268,-0.3442286189293386)-- (1.4512838029530408,-0.3585035475217969);
			\draw [line width=0.8pt, dotted] (1.4512838029530408,-0.3585035475217969)-- (2.6064991820117447,0.8736085723717113);
			\draw [line width=0.8pt, dotted] (2.6064991820117447,0.8736085723717113)-- (2.123956016313511,-0.8091575033986482);
			\draw [->,line width=1.2pt] (3.3514808652246257,-1.064891846921797) -- (3.492067772266861,0.11705786281501096);
			\draw (3.1673699977924024,-1.1749253361876284) node[anchor=north west] {$a$};
			\draw (5.489763891830794,0.07454974130670684) node[anchor=north west] {$b$};
			\draw (5.8700389154160275,1.5277435814359879) node[anchor=north west] {$c$};
			\draw (5.435438888461475,3.3612124451505014) node[anchor=north west] {$d$};
			\draw (1.184507374812256,3.130331180830896) node[anchor=north west] {$e$};
			\draw (0.7227448461730438,0.5498935207882474) node[anchor=north west] {$f$};
			\draw (1.8771511677710742,1.0116560494274582) node[anchor=north west] {$g$};
			\draw (3.2216950011617214,1.378349822170361) node[anchor=north west] {$h$};
			\draw (2.067288679563691,2.111737367656166) node[anchor=north west] {$i$};
			\draw (1.7820824118747658,-1.650269115669169) node[anchor=north west] {$\alpha = (a, d, h, b, a, b, f, g, a, \cdots)$};
			\begin{scriptsize}
			\draw [fill=ududff] (0.52,1.78) circle (1.0pt);
			\draw [fill=ududff] (2.88,3.66) circle (1.0pt);
			\draw [fill=ududff] (7.08,2.22) circle (1.0pt);
			\draw [fill=ududff] (4.3,1.12) circle (1.0pt);
			\draw [fill=ududff] (5.92,-1.6) circle (1.0pt);
			\draw [fill=ududff] (1.6,-0.7) circle (1.0pt);
			\draw [fill=ududff] (2.94,2.08) circle (1.0pt);
			\draw [fill=ududff] (1.82,1.32) circle (1.0pt);
			\draw [fill=ududff] (3.3,0.48) circle (1.0pt);
			\end{scriptsize}
			\end{tikzpicture}
		\caption{Coding a billiard trajectory by a sequence $\alpha$.}
		\label{fig:non_simply_connected}
	\end{figure}

	\section{The Billiard Map and Trajectory Unfolding}\label{sec:basic definitions}
	
	A \textit{polygon} is defined to be a bounded connected region in $\R^2$ bounded by finitely many straight line segments, each of which is called an \textit{edge}. The end points of the edges are called \textit{vertices}. The union of the edges of a polygon $Q$ is called the \textit{boundary} of $Q$, denoted by $\partial Q$. In general, a polygon may not be simply connected as shown in Figure~\ref{fig:non_simply_connected}. In this case, the complement of the interior of $Q$ in $\R^2$ contains finitely many bounded connected components, each of which will be called a \textit{hole} in $Q$. 
	Note that under the assumption of the main theorem that all holes of $Q$ have non-zero minimal diameters, the holes may still have empty interior. For example, the holes may be in the form of a broken line. However, the holes must not be reduced to a single segment, i.e, a `slit' in $Q$.

	Recall that the set $TQ = Q \times S^1$ consists of couples $(x, v)$ where $x$ is a point in $Q$ and $v$ a unit vector tangent to $x$. We will call $x$ the \textit{base point} of $(x,v)$. The billiard flow $\phi_t$ can be defined on a suitable subset of $TQ$. We define the \textit{phase space} $V$ as a subset of $TQ$ consisting of the couples $(x, v)$ such that $x$ lies on $\partial Q$ and $v$ points strictly into $Q$. For $(x,v) \in V$, define $f$ to be the first return map to $V$ under the billiard flow, i.e. $f(x,v) = \phi_{t_0}(x,v)$ where $t_0> 0$ is the smallest value such that $\phi_{t_0}(x,v) \in V$. The map $f$ is called the \textit{billiard map}.

	Following \cite{ZemlyakovKatok}, we introduce a method to visualise billiard trajectories by `unfolding'.
	Let $p = (x, v)$ be a phase point with an infinite orbit $\{f^n(p)\mid n \geq 0\}$ under $f$. 
	Let $Q_0 = Q$. Supposing the base point of $f(p)$ lies in $e$, we obtain a polygon $Q_1$ by reflecting $Q_0$ about the edge $e$. 
	Define $\gamma_1$ to be the image in $Q_1$ of the line segment in $Q$ joining the base points of $p$ and $f(p)$. 
	Continuing this way, we obtain a sequence of polygons $Q_n$ where $Q_n$ is a reflection of $Q_{n-1}$ about the edge $e'$ containing the base point of $f^n(p)$. 
	The line segment $\gamma_n$ is the image in $Q_n$ of the line segment in $Q$ joining the base points of $f^n(p)$ and $f^{n+1}(p)$.

	By identifying the edges of reflection between $Q_n$ and $Q_{n+1}$, an `infinite corridor' $Q^\infty = \bigcup_{i \geq 0}Q_i$ can be constructed with a natural Riemannian metric inherited from $Q$. We will call $Q^\infty$ the \textit{unfolding of the billiard trajectory from $p$}.
	Figure~\ref{fig:unfold_convex} illustrates the unfolding of a polygon $Q$. 
	The edge shared between $Q_n$ and $Q_{n+1}$ will be called a \textit{reflecting edge}. 
	The infinite corridor $Q^\infty$ may or may not be embedded in $\R^2$ depending on the convexity of $Q$. 
	For $0 \leq n \leq m \leq \infty$, put $Q_n^m := \bigcup_{i = n}^m Q_i$. If $m < \infty$, then we will call $Q_n^m$ a \textit{finite corridor}.
	Let $\Gamma_p$ be the piecewise linear curve in $Q^\infty$ joining the line segments $\gamma_0, \gamma_1, \gamma_2, \cdots$ end to end. 
	According to the law of reflection, $\Gamma_p$ is the trajectory in $Q^\infty$ of a straight line flow which begins from the point $x$ in $Q_0$ in the direction $v$.

	\begin{figure}[htbp]
		\centering
		\begin{tikzpicture}[line cap=round,line join=round,>=triangle 45,x=1.0cm,y=1.0cm]
			\clip(-1.4753829709493762,-1.995973668811395) rectangle (10.626882934227394,5.180754644188562);
			\fill[line width=2.pt,color=zzttqq,fill=zzttqq,fill opacity=0.10000000149011612] (0.42,2.02) -- (1.84,2.74) -- (2.84,2.04) -- (3.16,0.56) -- (1.48,0.) -- (0.2,1.1) -- cycle;
			\fill[line width=2.pt,color=ffffff,fill=ffffff,fill opacity=1.0] (1.56,1.68) -- (0.92,1.5) -- (1.56,2.26) -- (2.1,1.72) -- cycle;
			\fill[line width=2.pt,color=zzttqq,fill=zzttqq,fill opacity=0.10000000149011612] (5.052100488485695,3.0215352407536646) -- (3.461493370551291,3.0905931612002795) -- (2.84,2.04) -- (3.16,0.56) -- (4.921284019539429,0.7440614096301474) -- (5.632519190509421,2.2745987438939297) -- cycle;
			\fill[line width=2.pt,color=ffffff,fill=ffffff,fill opacity=1.0] (4.154389392882066,2.240949057920447) -- (4.811584089323099,2.341423586880671) -- (3.9147801814375436,2.769141660851361) -- (3.6460990928122823,2.0542916957431965) -- cycle;
			\fill[line width=2.pt,color=zzttqq,fill=zzttqq,fill opacity=0.10000000149011612] (5.052100488485695,3.021535240753665) -- (5.3780850075539215,4.5799107068490015) -- (6.549639910644462,4.92263378609012) -- (7.904732955018702,4.246969045935247) -- (7.291313190174384,2.585729691688442) -- (5.632519190509421,2.2745987438939297) -- cycle;
			\fill[line width=2.pt,color=ffffff,fill=ffffff,fill opacity=1.0] (6.030245110073837,3.6986121191494568) -- (5.770547709788525,3.0866012856791247) -- (5.577607072077456,4.061267336344503) -- (6.3366732062211995,4.145044445366436) -- cycle;
			\fill[line width=2.pt,color=zzttqq,fill=zzttqq,fill opacity=0.10000000149011612] (9.276394846427571,1.4616964101495409) -- (10.041421194832452,2.8579537955360217) -- (9.37374681286403,3.8798201858547787) -- (7.904732955018703,4.246969045935246) -- (7.291313190174385,2.5857296916884414) -- (8.349832134241694,1.2712192713985804) -- cycle;
			\fill[line width=2.pt,color=ffffff,fill=ffffff,fill opacity=1.0] (8.97301196616818,2.6119828214252347) -- (8.772644542820252,1.9780641392104634) -- (9.552715526882311,2.593441470467027) -- (9.030254159178519,3.150428112368862) -- cycle;
			\fill[line width=2.pt,color=zzttqq,fill=zzttqq,fill opacity=0.10000000149011612] (-0.6764471282123298,0.7441342508074711) -- (-1.1746889481814355,-0.768001685156579) -- (-0.332251088330291,-1.6513467209661559) -- (1.1790422693441929,-1.7451144502176663) -- (1.48,0.) -- (0.2,1.1) -- cycle;
			\fill[line width=2.pt,color=ffffff,fill=ffffff,fill opacity=1.0] (-0.168858306417638,-0.33176239292234233) -- (-0.08716191546131163,0.32802977109956466) -- (-0.7422609184103355,-0.41899452324111763) -- (-0.12718719280999868,-0.8716360061789077) -- cycle;
			\fill[line width=0.4pt,color=zzttqq,fill=zzttqq,pattern=north east lines,pattern color=zzttqq] (-0.6764471282123298,0.7441342508074711) -- (-0.7420410906359197,0.5450602594489449) -- (9.58103153202965,3.5625737953050343) -- (9.41862772029489,3.8111305789008676) -- (4.651644267697687,2.417704646603224) -- (4.811584089323099,2.341423586880671) -- (4.154389392882066,2.240949057920447) -- (3.6460990928122823,2.0542916957431965) -- (3.6754386646287824,2.1323522395523136) -- (0.20423055389158168,1.117691407182978) -- (0.2,1.1) -- cycle;
			\draw [line width=0.8pt] (1.56,2.26)-- (0.92,1.5);
			\draw [line width=0.8pt] (0.92,1.5)-- (1.56,1.68);
			\draw [line width=0.8pt] (1.56,1.68)-- (2.1,1.72);
			\draw [line width=0.8pt] (2.1,1.72)-- (1.56,2.26);
			\draw [line width=0.8pt] (0.42,2.02)-- (1.84,2.74);
			\draw [line width=0.8pt] (1.84,2.74)-- (2.84,2.04);
			\draw [line width=0.8pt] (2.84,2.04)-- (3.16,0.56);
			\draw [line width=0.8pt] (3.16,0.56)-- (1.48,0.);
			\draw [line width=0.8pt] (1.48,0.)-- (0.2,1.1);
			\draw [line width=0.8pt] (0.2,1.1)-- (0.42,2.02);
			\draw [line width=0.8pt] (3.9147801814375436,2.769141660851361)-- (4.811584089323099,2.341423586880671);
			\draw [line width=0.8pt] (4.811584089323099,2.341423586880671)-- (4.154389392882066,2.240949057920447);
			\draw [line width=0.8pt] (4.154389392882066,2.240949057920447)-- (3.6460990928122823,2.0542916957431965);
			\draw [line width=0.8pt] (3.6460990928122823,2.0542916957431965)-- (3.9147801814375436,2.769141660851361);
			\draw [line width=0.8pt] (5.052100488485695,3.0215352407536646)-- (3.461493370551291,3.0905931612002795);
			\draw [line width=0.8pt] (3.461493370551291,3.0905931612002795)-- (2.84,2.04);
			\draw [line width=0.8pt] (3.16,0.56)-- (4.921284019539429,0.7440614096301474);
			\draw [line width=0.8pt] (4.921284019539429,0.7440614096301474)-- (5.632519190509421,2.2745987438939297);
			\draw [line width=0.8pt] (5.632519190509421,2.2745987438939297)-- (5.052100488485695,3.0215352407536646);
			\draw [line width=0.8pt] (6.549639910644462,4.92263378609012)-- (7.904732955018702,4.246969045935248);
			\draw [line width=0.8pt] (5.577607072077456,4.061267336344503)-- (5.770547709788525,3.0866012856791247);
			\draw [line width=0.8pt] (5.770547709788525,3.0866012856791247)-- (6.030245110073837,3.6986121191494568);
			\draw [line width=0.8pt] (6.030245110073837,3.6986121191494568)-- (6.3366732062211995,4.145044445366436);
			\draw [line width=0.8pt] (6.3366732062211995,4.145044445366436)-- (5.577607072077456,4.061267336344503);
			\draw [line width=0.8pt] (5.052100488485695,3.021535240753665)-- (5.3780850075539215,4.5799107068490015);
			\draw [line width=0.8pt] (5.3780850075539215,4.5799107068490015)-- (6.549639910644462,4.92263378609012);
			\draw [line width=0.8pt] (7.904732955018702,4.246969045935247)-- (7.291313190174384,2.585729691688442);
			\draw [line width=0.8pt] (7.291313190174384,2.585729691688442)-- (5.632519190509421,2.2745987438939297);
			\draw [line width=0.8pt] (8.349832134241694,1.2712192713985804)-- (9.27639484642757,1.4616964101495404);
			\draw [line width=0.8pt] (9.37374681286403,3.8798201858547787)-- (7.904732955018704,4.246969045935247);
			\draw [line width=0.8pt] (9.552715526882311,2.593441470467027)-- (8.772644542820252,1.9780641392104634);
			\draw [line width=0.8pt] (8.772644542820252,1.9780641392104634)-- (8.97301196616818,2.6119828214252347);
			\draw [line width=0.8pt] (8.97301196616818,2.6119828214252347)-- (9.030254159178519,3.150428112368862);
			\draw [line width=0.8pt] (9.030254159178519,3.150428112368862)-- (9.552715526882311,2.593441470467027);
			\draw [line width=0.8pt] (9.276394846427571,1.4616964101495409)-- (10.041421194832452,2.8579537955360217);
			\draw [line width=0.8pt] (10.041421194832452,2.8579537955360217)-- (9.37374681286403,3.8798201858547787);
			\draw [line width=0.8pt] (7.291313190174385,2.5857296916884414)-- (8.349832134241694,1.2712192713985804);
			\draw [line width=0.8pt] (-0.7422609184103355,-0.41899452324111763)-- (-0.08716191546131163,0.32802977109956466);
			\draw [line width=0.8pt] (-0.08716191546131163,0.32802977109956466)-- (-0.168858306417638,-0.33176239292234233);
			\draw [line width=0.8pt] (-0.168858306417638,-0.33176239292234233)-- (-0.12718719280999868,-0.8716360061789077);
			\draw [line width=0.8pt] (-0.12718719280999868,-0.8716360061789077)-- (-0.7422609184103355,-0.41899452324111763);
			\draw [line width=0.8pt] (-0.6764471282123298,0.7441342508074711)-- (-1.1746889481814355,-0.768001685156579);
			\draw [line width=0.8pt] (-1.1746889481814355,-0.768001685156579)-- (-0.332251088330291,-1.6513467209661559);
			\draw [line width=0.8pt] (-0.332251088330291,-1.6513467209661559)-- (1.1790422693441929,-1.7451144502176663);
			\draw [line width=0.8pt] (1.1790422693441929,-1.7451144502176663)-- (1.48,0.);
			\draw [line width=0.8pt] (0.2,1.1)-- (-0.6764471282123298,0.7441342508074711);
			\draw [line width=0.8pt] (-0.7420410906359197,0.5450602594489449)-- (0.4427525265036758,0.8913845475359036);
			\draw [line width=0.8pt] (0.4427525265036758,0.8913845475359036)-- (2.931082682561552,1.6187425931528214);
			\draw [line width=0.8pt] (2.931082682561552,1.6187425931528214)-- (5.547796228721604,2.3836280912611443);
			\draw [line width=0.8pt] (5.547796228721604,2.3836280912611443)-- (7.418614729956154,2.930482730083551);
			\draw [line width=0.8pt] (7.418614729956154,2.930482730083551)-- (9.58103153202965,3.5625737953050343);
			\draw [->,line width=0.8pt, -{Latex[width=1.5mm]}] (-0.7420410906359197,0.5450602594489449) -- (0.01147214438639721,0.7653179743016222);
			\draw (0.420356583427964,-0.7434314632406476) node[anchor=north west] {$Q_0$};
			\draw (2.349948629847757,0.9830456309244365) node[anchor=north west] {$Q_1$};
			\draw (3.636343327460952,1.3046443053277366) node[anchor=north west] {$Q_2$};
			\draw (6.784625087409036,4.3175160971060205) node[anchor=north west] {$Q_3$};
			\draw (8.003314800937325,2.286366574558863) node[anchor=north west] {$Q_4$};
			\draw (-0.0027995671026922794,0.9153406468395313) node[anchor=north west] {$\gamma_0$};
			\draw (1.1989639004043717,1.2369393212428312) node[anchor=north west] {$\gamma_1$};
			\draw (4.296466922288776,2.1001778683253733) node[anchor=north west] {$\gamma_2$};
			\draw (6.361468936878379,2.658743987025842) node[anchor=north west] {$\gamma_3$};
			\draw (8.172577261149588,3.268088843789989) node[anchor=north west] {$\gamma_4$};
			\draw [line width=0.8pt] (9.58103153202965,3.5625737953050343)-- (10.305948715218808,3.77447266423725);
			\draw (10.000611831442024,3.7250974863630995) node[anchor=north west] {$\Gamma_p$};
			\draw (-0.7644806380578736,0.6106682184574577) node[anchor=north west] {$p$};
			\begin{scriptsize}
			\draw [fill=ududff] (0.42,2.02) circle (1.0pt);
			\draw [fill=ududff] (1.84,2.74) circle (1.0pt);
			\draw [fill=ududff] (2.84,2.04) circle (1.0pt);
			\draw [fill=ududff] (3.16,0.56) circle (1.0pt);
			\draw [fill=ududff] (1.48,0.) circle (1.0pt);
			\draw [fill=ududff] (0.2,1.1) circle (1.0pt);
			\draw [fill=ududff] (1.56,1.68) circle (1.0pt);
			\draw [fill=ududff] (0.92,1.5) circle (1.0pt);
			\draw [fill=ududff] (1.56,2.26) circle (1.0pt);
			\draw [fill=ududff] (2.1,1.72) circle (1.0pt);
			\draw [fill=ududff] (5.052100488485695,3.0215352407536646) circle (1.0pt);
			\draw [fill=ududff] (3.461493370551291,3.0905931612002795) circle (1.0pt);
			\draw [fill=ududff] (2.84,2.04) circle (1.0pt);
			\draw [fill=ududff] (3.16,0.56) circle (1.0pt);
			\draw [fill=ududff] (4.921284019539429,0.7440614096301474) circle (1.0pt);
			\draw [fill=ududff] (5.632519190509421,2.2745987438939297) circle (1.0pt);
			\draw [fill=ududff] (4.154389392882066,2.240949057920447) circle (1.0pt);
			\draw [fill=ududff] (4.811584089323099,2.341423586880671) circle (1.0pt);
			\draw [fill=ududff] (3.9147801814375436,2.769141660851361) circle (1.0pt);
			\draw [fill=ududff] (3.6460990928122823,2.0542916957431965) circle (1.0pt);
			\draw [fill=ududff] (3.9147801814375436,2.769141660851361) circle (1.0pt);
			\draw [fill=ududff] (4.811584089323099,2.341423586880671) circle (1.0pt);
			\draw [fill=ududff] (4.811584089323099,2.341423586880671) circle (1.0pt);
			\draw [fill=ududff] (4.154389392882066,2.240949057920447) circle (1.0pt);
			\draw [fill=ududff] (4.154389392882066,2.240949057920447) circle (1.0pt);
			\draw [fill=ududff] (3.6460990928122823,2.0542916957431965) circle (1.0pt);
			\draw [fill=ududff] (3.6460990928122823,2.0542916957431965) circle (1.0pt);
			\draw [fill=ududff] (3.9147801814375436,2.769141660851361) circle (1.0pt);
			\draw [fill=ududff] (5.052100488485695,3.0215352407536646) circle (1.0pt);
			\draw [fill=ududff] (3.461493370551291,3.0905931612002795) circle (1.0pt);
			\draw [fill=ududff] (3.461493370551291,3.0905931612002795) circle (1.0pt);
			\draw [fill=ududff] (2.84,2.04) circle (1.0pt);
			\draw [fill=ududff] (3.16,0.56) circle (1.0pt);
			\draw [fill=ududff] (4.921284019539429,0.7440614096301474) circle (1.0pt);
			\draw [fill=ududff] (4.921284019539429,0.7440614096301474) circle (1.0pt);
			\draw [fill=ududff] (5.632519190509421,2.2745987438939297) circle (1.0pt);
			\draw [fill=ududff] (5.632519190509421,2.2745987438939297) circle (1.0pt);
			\draw [fill=ududff] (5.052100488485695,3.0215352407536646) circle (1.0pt);
			\draw [fill=ududff] (5.052100488485695,3.021535240753665) circle (1.0pt);
			\draw [fill=ududff] (5.3780850075539215,4.5799107068490015) circle (1.0pt);
			\draw [fill=ududff] (6.549639910644462,4.92263378609012) circle (1.0pt);
			\draw [fill=ududff] (7.904732955018702,4.246969045935247) circle (1.0pt);
			\draw [fill=ududff] (7.291313190174384,2.585729691688442) circle (1.0pt);
			\draw [fill=ududff] (5.632519190509421,2.2745987438939297) circle (1.0pt);
			\draw [fill=ududff] (6.030245110073837,3.6986121191494568) circle (1.0pt);
			\draw [fill=ududff] (5.770547709788525,3.0866012856791247) circle (1.0pt);
			\draw [fill=ududff] (5.577607072077456,4.061267336344503) circle (1.0pt);
			\draw [fill=ududff] (6.3366732062211995,4.145044445366436) circle (1.0pt);
			\draw [fill=ududff] (6.549639910644462,4.92263378609012) circle (1.0pt);
			\draw [fill=ududff] (7.904732955018702,4.246969045935248) circle (1.0pt);
			\draw [fill=ududff] (5.577607072077456,4.061267336344503) circle (1.0pt);
			\draw [fill=ududff] (5.770547709788525,3.0866012856791247) circle (1.0pt);
			\draw [fill=ududff] (5.770547709788525,3.0866012856791247) circle (1.0pt);
			\draw [fill=ududff] (6.030245110073837,3.6986121191494568) circle (1.0pt);
			\draw [fill=ududff] (6.030245110073837,3.6986121191494568) circle (1.0pt);
			\draw [fill=ududff] (6.3366732062211995,4.145044445366436) circle (1.0pt);
			\draw [fill=ududff] (6.3366732062211995,4.145044445366436) circle (1.0pt);
			\draw [fill=ududff] (5.577607072077456,4.061267336344503) circle (1.0pt);
			\draw [fill=ududff] (5.052100488485695,3.021535240753665) circle (1.0pt);
			\draw [fill=ududff] (5.3780850075539215,4.5799107068490015) circle (1.0pt);
			\draw [fill=ududff] (5.3780850075539215,4.5799107068490015) circle (1.0pt);
			\draw [fill=ududff] (6.549639910644462,4.92263378609012) circle (1.0pt);
			\draw [fill=ududff] (7.904732955018702,4.246969045935247) circle (1.0pt);
			\draw [fill=ududff] (7.291313190174384,2.585729691688442) circle (1.0pt);
			\draw [fill=ududff] (7.291313190174384,2.585729691688442) circle (1.0pt);
			\draw [fill=ududff] (5.632519190509421,2.2745987438939297) circle (1.0pt);
			\draw [fill=ududff] (9.276394846427571,1.4616964101495409) circle (1.0pt);
			\draw [fill=ududff] (10.041421194832452,2.8579537955360217) circle (1.0pt);
			\draw [fill=ududff] (9.37374681286403,3.8798201858547787) circle (1.0pt);
			\draw [fill=ududff] (7.904732955018703,4.246969045935246) circle (1.0pt);
			\draw [fill=ududff] (7.291313190174385,2.5857296916884414) circle (1.0pt);
			\draw [fill=ududff] (8.349832134241694,1.2712192713985804) circle (1.0pt);
			\draw [fill=ududff] (8.97301196616818,2.6119828214252347) circle (1.0pt);
			\draw [fill=ududff] (8.772644542820252,1.9780641392104634) circle (1.0pt);
			\draw [fill=ududff] (9.552715526882311,2.593441470467027) circle (1.0pt);
			\draw [fill=ududff] (9.030254159178519,3.150428112368862) circle (1.0pt);
			\draw [fill=ududff] (8.349832134241694,1.2712192713985804) circle (1.0pt);
			\draw [fill=ududff] (9.27639484642757,1.4616964101495404) circle (1.0pt);
			\draw [fill=ududff] (9.37374681286403,3.8798201858547787) circle (1.0pt);
			\draw [fill=ududff] (7.904732955018704,4.246969045935247) circle (1.0pt);
			\draw [fill=ududff] (9.552715526882311,2.593441470467027) circle (1.0pt);
			\draw [fill=ududff] (8.772644542820252,1.9780641392104634) circle (1.0pt);
			\draw [fill=ududff] (8.772644542820252,1.9780641392104634) circle (1.0pt);
			\draw [fill=ududff] (8.97301196616818,2.6119828214252347) circle (1.0pt);
			\draw [fill=ududff] (8.97301196616818,2.6119828214252347) circle (1.0pt);
			\draw [fill=ududff] (9.030254159178519,3.150428112368862) circle (1.0pt);
			\draw [fill=ududff] (9.030254159178519,3.150428112368862) circle (1.0pt);
			\draw [fill=ududff] (9.552715526882311,2.593441470467027) circle (1.0pt);
			\draw [fill=ududff] (10.041421194832452,2.8579537955360217) circle (1.0pt);
			\draw [fill=ududff] (10.041421194832452,2.8579537955360217) circle (1.0pt);
			\draw [fill=ududff] (9.37374681286403,3.8798201858547787) circle (1.0pt);
			\draw [fill=ududff] (7.291313190174385,2.5857296916884414) circle (1.0pt);
			\draw [fill=ududff] (8.349832134241694,1.2712192713985804) circle (1.0pt);
			\draw [fill=ududff] (9.27639484642757,1.4616964101495404) circle (1.0pt);
			\draw [fill=ududff] (8.349832134241694,1.2712192713985804) circle (1.0pt);
			\draw [fill=ududff] (9.27639484642757,1.4616964101495404) circle (1.0pt);
			\draw [fill=ududff] (8.349832134241694,1.2712192713985804) circle (1.0pt);
			\draw [fill=ududff] (-0.6764471282123298,0.7441342508074711) circle (1.0pt);
			\draw [fill=ududff] (-1.1746889481814355,-0.768001685156579) circle (1.0pt);
			\draw [fill=ududff] (-0.332251088330291,-1.6513467209661559) circle (1.0pt);
			\draw [fill=ududff] (1.1790422693441929,-1.7451144502176663) circle (1.0pt);
			\draw [fill=ududff] (1.48,0.) circle (1.0pt);
			\draw [fill=ududff] (0.2,1.1) circle (1.0pt);
			\draw [fill=ududff] (-0.168858306417638,-0.33176239292234233) circle (1.0pt);
			\draw [fill=ududff] (-0.08716191546131163,0.32802977109956466) circle (1.0pt);
			\draw [fill=ududff] (-0.7422609184103355,-0.41899452324111763) circle (1.0pt);
			\draw [fill=ududff] (-0.12718719280999868,-0.8716360061789077) circle (1.0pt);
			\draw [fill=ududff] (-0.7422609184103355,-0.41899452324111763) circle (1.0pt);
			\draw [fill=ududff] (-0.08716191546131163,0.32802977109956466) circle (1.0pt);
			\draw [fill=ududff] (-0.08716191546131163,0.32802977109956466) circle (1.0pt);
			\draw [fill=ududff] (-0.168858306417638,-0.33176239292234233) circle (1.0pt);
			\draw [fill=ududff] (-0.168858306417638,-0.33176239292234233) circle (1.0pt);
			\draw [fill=ududff] (-0.12718719280999868,-0.8716360061789077) circle (1.0pt);
			\draw [fill=ududff] (-0.12718719280999868,-0.8716360061789077) circle (1.0pt);
			\draw [fill=ududff] (-0.7422609184103355,-0.41899452324111763) circle (1.0pt);
			\draw [fill=ududff] (-0.6764471282123298,0.7441342508074711) circle (1.0pt);
			\draw [fill=ududff] (-1.1746889481814355,-0.768001685156579) circle (1.0pt);
			\draw [fill=ududff] (-1.1746889481814355,-0.768001685156579) circle (1.0pt);
			\draw [fill=ududff] (-0.332251088330291,-1.6513467209661559) circle (1.0pt);
			\draw [fill=ududff] (-0.332251088330291,-1.6513467209661559) circle (1.0pt);
			\draw [fill=ududff] (1.1790422693441929,-1.7451144502176663) circle (1.0pt);
			\draw [fill=ududff] (1.1790422693441929,-1.7451144502176663) circle (1.0pt);
			\draw [fill=ududff] (1.48,0.) circle (1.0pt);
			\draw [fill=ududff] (0.2,1.1) circle (1.0pt);
			\draw [fill=ududff] (-0.6764471282123298,0.7441342508074711) circle (1.0pt);
			\draw [fill=ududff] (-0.3322510883302909,-1.6513467209661568) circle (1.0pt);
			\draw [fill=ududff] (-0.3322510883302909,-1.6513467209661568) circle (1.0pt);
			\draw [fill=zzttqq] (-0.7420410906359197,0.5450602594489449) circle (1.0pt);
			\draw [fill=zzttqq] (0.4427525265036758,0.8913845475359036) circle (1.0pt);
			\draw [fill=zzttqq] (2.931082682561552,1.6187425931528214) circle (1.0pt);
			\draw [fill=zzttqq] (5.547796228721604,2.3836280912611443) circle (1.0pt);
			\draw [fill=zzttqq] (7.418614729956154,2.930482730083551) circle (1.0pt);
			\draw [fill=zzttqq] (9.58103153202965,3.5625737953050343) circle (1.0pt);
			\end{scriptsize}
			\end{tikzpicture}
		\caption{Unfolding a billiard trajectory into $Q^\infty = Q_0 \cup Q_1 \cup Q_2 \cup \cdots$. The shaded region shows an open left $d$-strip of the unfolded trajectory $\Gamma_p$.}
		\label{fig:unfold_convex}
	\end{figure}

	Let $\Gamma$ be any ray contained in $Q^\infty$ pointing in the direction $(a,b) \in \R^2$. For example we can take $\Gamma = \Gamma_p$ and let the direction of $\Gamma_p$ be the direction of the billiard trajectory from $p$. Relative to the direction $(a,b)$ of $\Gamma$, there is a natural way to define the `left side' and the `right side' of $\Gamma$ in $Q^\infty$.
	The following terminology will be frequently used:
	for $d > 0$, the \textit{open left (resp. right)} $d$-\textit{strip} of $\Gamma$ is defined to be the set of points in $Q^\infty \setminus \Gamma$ which can be obtained by translating a point in $\Gamma$ along the direction $(-b, a)$ (resp. in the direction $(b, -a)$) by a distance $< d$.
	The \textit{closed left/right $d$-strip} of $\Gamma$ is defined as the closure of the left/right $d$-strip of $\Gamma$ in $Q^\infty$.

	\section{Partitions of the Phase Space} \label{sec:partition and encoding}

	We follow the methods in \cite{katok1987} to parametrise the phase space $V \subset TQ$.
	Fix a vertex $v_0$ of $Q$ and a counter-clockwise orientation on the boundary $\partial Q$. Let $L$ be the total length of the boundary of $Q$. Each point on $\partial Q$ can be uniquely assigned a spatial coordinate $x \in [0, L)$ according to the orientation on $\partial Q$ such that $v_0$ corresponds to $x = 0$. Each phase point in $V$ can be assigned a pair $(x, \theta)$ with  $ 0 < \theta < \pi$ and $ 0 \leq x < L$. 
	Here the angular coordinate $\theta$ of the phase point is the angle between its direction and the positive orientation on $\partial Q$. This pair $(x, \theta)$ is uniquely defined unless $x$ is a vertex of $Q$, in which case there are exactly two ways to assign $\theta$: it could be the angle measured from either of the two edges having $x$ as an endpoint. Let $x \in e$ for some edge $e \subset \partial Q$. Since the phase points always point into the polygon, it is natural to define the \textit{left side} (resp. \textit{right side}) of $x$ to be the set of points in $e$ whose spatial coordinates are smaller than (resp. greater than) that of $x$.

	Let $\{e_a \mid a \in \mathcal{A}\}$ be the set of edges of $Q$ indexed by the finite set $\mathcal{A}$. We define a family $\{E_a\}_{a \in \mathcal{A}}$ of disjoint open subsets in $V$ where $E_a$ consists of phase points $(x,v)$ with $x$ lying in the interior of the edge $e_a$. Then the set \[\mathring{V} = \bigcup_{a \in \mathcal{A}} E_a\] is the set of phase points that are not based on a vertex and therefore have unique coordinates in the form $(x, \theta) \in [0, L) \times (0, \pi)$. 

	\begin{definition}\label{def:edge coding}
		Let $p$ be a phase point such that $f^n(p)$ is defined up to $N \leq \infty$. The \textit{edge coding} of $p$ is the sequence $(\alpha_n)_{0 \leq n \leq N} \in \mathcal{A}^{N + 1}$ satisfying 
		\[f^n(p) \in E_{\alpha_n}\]
		for all $n = 0, 1, \ldots, N$.
	\end{definition}

	The partition of $\mathring{V}$ into the disjoint union $\bigcup_{a \in \mathcal{A}} E_a$ does not provide geometric information about the presence of holes in a non-simply connected domain. To remedy this, we refine this partition by partitioning each $E_a$ further: for $a, b \in \mathcal{A}$, put 
	\[V_{a,b} = E_a \cap f^{-1}(E_b).\] 
	In other words, the subset $V_{a,b}$ consists of the phase points which will be sent from edge $e_a$ to edge $e_b$ by the billiard map $f$. If $Q$ is not simply connected, then $V_{a,b}$ may not be connected in $V$ as illustrated in Figure~\ref{fig:V_not_connected}. Let \[\{V_{a,b}^i\}_{i\in I_{a,b}}\] be the collection of connected components of $V_{a,b}$ for a set $I_{a,b}$ depending only on $a,b$. The set $I_{a,b}$ is finite as there are finitely many holes in $Q$ that separate $V_{a,b}$. It follows that 
	\[\mathring{V} = \bigcup_{a, b \in \mathcal{A}} \bigcup_{i\in I_{a,b}} V_{a,b}^i\]
	is a partition of $\mathring{V}$ into finitely many open subsets $V_{a,b}^i$ indexed by $\bigsqcup_{a, b \in \mathcal{A}} I_{a,b}$.

	\begin{definition}\label{def:V coding}
		Let $p$ be a phase point such that $f^n(p)$ is defined up to $N \leq \infty$. Then $p$ can be uniquely associated with a sequence $\xi = (\xi_{n})_{0 \leq n \leq N - 1} \in \left(\bigsqcup_{a,b \in \mathcal{A}} I_{a,b}\right)^{N}$ such that for $0 \leq n \leq N-1$
		\[f^n(p)  \in V^i_{a,b} \quad \text{ if and only if } \quad \xi_{n} = i \in I_{a,b}.\]
		We will call $\xi$ the \textit{$V$-coding} of $p$.
	\end{definition}

	We can verify the following result using the intermediate value theorem.

	\begin{remark}\label{rem:V partition detects holes}
		Let $a, b \in \mathcal{A}$ and suppose $p \in V^i_{a,b}$ and $q \in V^j_{a,b}$ are two phase points. Then $i = j$ if and only if there are no holes between the billiard trajectory from $p$ to $f(p)$ and the billiard trajectory from $q$ to $f(q)$. Hence, if two phase points have the same $V$-coding, then there are no holes between the trajectories from the two phase points.
	\end{remark}

	\begin{figure}[htbp]
		\centering
		\begin{tikzpicture}[line cap=round,line join=round,>=triangle 45,x=.6cm,y=.6cm]
			\clip(-1.9987094724900802,-4.154987180941571) rectangle (6.398896426545703,4.3223428247305735);
			\fill[line width=2.pt,color=ududff,fill=ududff,fill opacity=0.10000000149011612] (0.74,3.08) -- (-1.14,0.68) -- (0.3,-2.86) -- (4.37921905842317,-3.038849688031759) -- (5.335908338060158,0.) -- (3.342805672149767,3.2062053318207613) -- cycle;
			\fill[line width=2.pt,color=ffffff,fill=ffffff,fill opacity=1.0] (2.42,1.36) -- (1.68,0.34) -- (1.94,-0.46) -- (2.64,0.06) -- cycle;
			\draw [line width=1.2pt,color=ududff] (0.74,3.08)-- (-1.14,0.68);
			\draw [line width=1.2pt,color=ududff] (-1.14,0.68)-- (0.3,-2.86);
			\draw [line width=1.2pt,color=ududff] (0.3,-2.86)-- (4.37921905842317,-3.038849688031759);
			\draw [line width=1.2pt,color=ududff] (4.37921905842317,-3.038849688031759)-- (5.335908338060158,0.);
			\draw [line width=1.2pt,color=ududff] (5.335908338060158,0.)-- (3.342805672149767,3.2062053318207613);
			\draw [line width=1.2pt,color=ududff] (3.342805672149767,3.2062053318207613)-- (0.74,3.08);
			\draw [line width=1.2pt,color=ududff] (2.42,1.36)-- (1.68,0.34);
			\draw [line width=1.2pt,color=ududff] (1.68,0.34)-- (1.94,-0.46);
			\draw [line width=1.2pt,color=ududff] (1.94,-0.46)-- (2.64,0.06);
			\draw [line width=1.2pt,color=ududff] (2.64,0.06)-- (2.42,1.36);
			\draw [line width=1.2pt,dotted,color=sqsqsq] (1.1327159237635374,-2.8965096801743604)-- (1.2097266816847787,3.1027761958418236);
			\draw [line width=1.2pt,dotted,color=sqsqsq] (3.5036336341977123,-3.0004604331952187)-- (2.9375039543166164,3.1865529857643575);
			\draw [->,line width=1.2pt,color=sqsqsq] (1.1327159237635374,-2.8965096801743604) -- (1.163592505429342,-0.49116469607794855);
			\draw [->,line width=1.2pt,color=sqsqsq] (3.5036336341977123,-3.0004604331952187) -- (3.2880794082262836,-0.644751245012789);
			\begin{scriptsize}
				\draw [fill=ududff] (0.74,3.08) circle (1.5pt);
				\draw [fill=ududff] (-1.14,0.68) circle (1.5pt);
				\draw [fill=ududff] (0.3,-2.86) circle (1.5pt);
				\draw [fill=ududff] (4.37921905842317,-3.038849688031759) circle (1.5pt);
				\draw [fill=ududff] (5.335908338060158,0.) circle (1.5pt);
				\draw [fill=ududff] (3.342805672149767,3.2062053318207613) circle (1.5pt);
				\draw[color=black] (2.3329669880885042,-3.2647346568349347) node {$e_a$};
				\draw[color=black] (1.9609211571185643,3.8307108338060134) node {$e_b$};
				\draw [fill=ududff] (2.42,1.36) circle (1.5pt);
				\draw [fill=ududff] (1.68,0.34) circle (1.5pt);
				\draw [fill=ududff] (1.94,-0.46) circle (1.5pt);
				\draw [fill=ududff] (2.64,0.06) circle (1.5pt);
				\draw[color=black] (1.3965536018151005,-2.1702525241066313) node {$u$};
				\draw[color=black] (3.077058650028384,-2.3828501418037382) node {$v$};
			\end{scriptsize}
		\end{tikzpicture}
		\caption{Both phase points $u$ and $v$ belong to $V_{a,b}$ because they are mapped from edge $e_a$ to edge $e_b$. They belong to two different connected components of $V_{a,b}$ due to the hole between their trajectories.}
		\label{fig:V_not_connected}
	\end{figure}

	\section{Generalised Trajectories and Their Codings}\label{sec:generalised trajectories}

	A billiard trajectory does not contain any vertex of the polygon. When the trajectory is unfolded into an infinite corridor (see Section \ref{sec:basic definitions}), it intersects the reflecting edges transversally and never intersects non-reflecting edges. These properties are not necessarily preserved when we take the `limit' of a sequence of billiard trajectories. In this section, we define the \textit{generalised trajectories} which may not be physically realisable but may be thought of as the limit of physical trajectories.
	
	Let $f$ denote the billiard map on the polygon $Q$. We have defined the partition 
	\[\mathring{V} = \bigcup_{a, b \in \mathcal{A}} \bigcup_{i\in I_{a,b}} V_{a,b}^i\] 
	in Section \ref{sec:partition and encoding}. 
	For $a , b \in \mathcal{A}$ and $i \in I_{a,b}$, let $\ol{V^i_{a,b}}$ be the closure of $V^i_{a,b}$ in $TQ$. 
	We call each element of the disjoint union \[\bigsqcup_{\substack{a, b \in \mathcal{A}\\ i\in I_{a,b}}} \ol{V_{a,b}^i}\] a \textit{generalised phase point}.

	\begin{definition}\label{def:generalised trajectories}
		A sequence of generalised phase points $\{z_n\}_{n \geq 0}$ is said to be a \textit{generalised trajectory} if for all $n \geq 0$ and $a, b \in \mathcal{A}$ such that $z_n \in \ol{V_{a,b}^i}$, there exists a sequence of phase points $\{p_m\}_{m \geq 0} \subset V_{a,b}^i$ satisfying
		\[p_m \rightarrow z_n \quad \text{ and } \quad f(p_m) \rightarrow z_{n+1} \quad \text{ as } \quad m \rightarrow \infty.\]
		We call $\{z_n\}_{n \geq 0}$ a \textit{physical trajectory} if it is the orbit under $f$ of some point in $\mathring{V}$.
	\end{definition}

	We can extend the definitions of edge codings and $V$-codings to generalised trajectories in an obvious way. Using the `unfolding' method in Section \ref{sec:basic definitions}, we can similarly construct an infinite corridor $Q^\infty$ according to the edge coding of a generalised trajectory $\{z_n\}_{n \geq 0}$ and obtain an unfolded `trajectory' $\Gamma$ in $Q^\infty$. The following properties of $\Gamma$ can be easily deduced from Definition \ref{def:generalised trajectories}. 

	\begin{remark}\label{rem:properties of generalised trajectories}
		Let $\{z_n\}_{n \geq 0}$ and $\Gamma$ be defined as above. Then 
		\begin{enumerate}
			\item For all $n \geq 0$, the interior of the segment in $Q$ joining the base points of $z_n$ and $z_{n+1}$ does not transversally intersect the interior of any edge of $Q$.
			\item $\Gamma$ is a straight line and $\Gamma$ is contained in $Q^\infty$.
			\item If $\Gamma$ does not contains any vertices, then $\Gamma$ is a physical trajectory.
		\end{enumerate}
	\end{remark}

	For convenience, we will also call the straight line $\Gamma$ in $Q^\infty$ associated with $\{z_n\}_{n \geq 0}$ a \textit{generalised trajectory}. Despite the nice properties of $\Gamma$ in Remark \ref{rem:properties of generalised trajectories}, the generalised trajectory $\Gamma$ may contain vertices in $Q^\infty$ or even overlap one or more edges entirely as illustrated in Figure \ref{fig:generalised trajectories overlapping vertices}. Two particular classes of vertices lying on $\Gamma$ are distinguished in the following definition.

	\begin{definition}\label{def:blocking vertices}
		Let $\{z_n\}_{n \geq 0}$ be a generalised trajectory in $Q$ whose unfolded trajectory in $Q^\infty$ is $\Gamma$. Let $v$ be a vertex in $Q^\infty$ lying on $\Gamma$. We say $v$ \textit{blocks $\gamma$ from the left (resp. from the right)} if $v$ is an end vertex of an edge $e$ such that one of the followings holds
		\begin{enumerate}
			\item The edge $e$ is a reflecting edge, and the interior of $e$ lies on the right (resp. left) side of $\Gamma$ relative to the forward direction of $\Gamma$.
			\item The edge $e$ is not a reflecting edge and the interior of $e$ lies on the left (resp. right) side of $\Gamma$ relative to the forward direction of $\Gamma$.
		\end{enumerate}
	\end{definition}

	For example, the vertices $v_1$ $v_2$ and $v_3$ in Figure \ref{fig:generalised trajectories overlapping vertices} block $\Gamma$ from the right whereas  $v_4$ blocks $\Gamma$ from the left.
	The vertices that block $\Gamma$ in the sense of Definition \ref{def:blocking vertices} ensure that trajectories near $\Gamma$ do not have the same $V$-coding as $\Gamma$.
	More precisely, we have the following lemma.

	\begin{figure}[]
		\centering
		\begin{tikzpicture}[line cap=round,line join=round,>=triangle 45,x=2.0cm,y=1.0cm]
			\clip(-0.6454946728781936,1.658736164980056) rectangle (5.421672843560385,5.995641480626333);
			\fill[line width=2.pt,color=zzttqq,fill=zzttqq,fill opacity=0.10000000149011612] (-0.06064225813758718,3.9182177174564456) -- (1.78,5.2) -- (3.7,3.98) -- (3.96,2.06) -- (0.5383767338770005,2.455229025420818) -- cycle;
			\fill[line width=2.pt,color=ffffff,fill=ffffff,fill opacity=1.0] (2.3155318727118956,3.2850746504288963) -- (0.88,3.82) -- (0.78,2.98) -- cycle;
			\fill[line width=2.pt,color=ffffff,fill=ffffff,fill opacity=1.0] (1.963759348485022,4.54292345131072) -- (2.1431537600660575,3.901454803388197) -- (2.88,3.54) -- (2.94,3.98) -- cycle;
			\draw [line width=0.8pt,color=ffffff] (2.3155318727118956,3.2850746504288963)-- (0.88,3.82);
			\draw [line width=0.8pt,color=ffffff] (0.88,3.82)-- (0.78,2.98);
			\draw [line width=0.8pt,color=ffffff] (0.78,2.98)-- (2.3155318727118956,3.2850746504288963);
			\draw [line width=0.8pt,color=ffffff] (1.963759348485022,4.54292345131072)-- (2.1431537600660575,3.901454803388197);
			\draw [line width=0.8pt,color=ffffff] (2.1431537600660575,3.901454803388197)-- (2.88,3.54);
			\draw [line width=0.8pt,color=ffffff] (2.88,3.54)-- (2.94,3.98);
			\draw [line width=0.8pt,color=ffffff] (2.94,3.98)-- (1.963759348485022,4.54292345131072);
			\draw [line width=0.8pt] (0.88,3.82)-- (0.78,2.98);
			\draw [line width=0.8pt] (0.78,2.98)-- (2.3155318727118956,3.2850746504288963);
			\draw [line width=0.8pt] (2.3155318727118956,3.2850746504288963)-- (0.88,3.82);
			\draw [line width=0.8pt] (1.963759348485022,4.54292345131072)-- (2.1431537600660575,3.901454803388197);
			\draw [line width=0.8pt] (2.1431537600660575,3.901454803388197)-- (2.88,3.54);
			\draw [line width=0.8pt] (2.94,3.98)-- (2.88,3.54);
			\draw [line width=0.8pt] (2.94,3.98)-- (1.963759348485022,4.54292345131072);
			\draw [line width=0.8pt] (1.78,5.2)-- (-0.06064225813758718,3.9182177174564456);
			\draw [line width=0.8pt] (-0.06064225813758718,3.9182177174564456)-- (0.5383767338770005,2.455229025420818);
			\draw [line width=0.8pt] (0.5383767338770005,2.455229025420818)-- (3.96,2.06);
			\draw [line width=0.8pt] (3.96,2.06)-- (3.7,3.98);
			\draw [line width=0.8pt] (3.7,3.98)-- (1.78,5.2);
			\draw [line width=0.8pt,dotted] (1.4522106732596802,4.971734811821746)-- (3.043348269995034,2.1658817238844277);
			\draw [line width=0.8pt,dotted] (1.6086333625232443,5.080664089523283)-- (2.805161949135644,2.193394439382325);
			\draw [line width=0.8pt,dotted] (0.9086338948254077,4.5932000472748635)-- (3.7440084533793025,2.0849490144970755);
			\draw [line width=0.8pt,dash pattern=on 2pt off 2pt] (2.7476850515893867,1.7398051355825919)-- (1.6003580852964723,5.842353655944471);
			\draw (1.194879473774842,5.823738400993907) node[anchor=north west] {$z_{n+1}$};
			\draw (2.307193518455248,2.2441095662951427) node[anchor=north west] {$z_n$};
			\draw (3.7835376141219688,2.092430378384178) node[anchor=north west] {$p_0$};
			\draw (3.37905977969273,2.1328781618271018) node[anchor=north west] {$p_1$};
			\draw (2.843126649073989,2.183437891130757) node[anchor=north west] {$\cdots$};
			\draw [line width=0.8pt,dotted] (3.409307151521811,2.1236100999093526)-- (1.2069548197822173,4.800944096631064);
			\draw (1.8015962254187,5.490044187589785) node[anchor=north west] {$v_1$};
			\draw (1.9532754133296644,4.883327435945927) node[anchor=north west] {$v_2$};
			\draw (2.145402384683553,4.195715117416221) node[anchor=north west] {$v_3$};
			\draw (2.024059034354781,3.265416098228971) node[anchor=north west] {$v_4$};
			\draw [->,line width=0.8pt, -{Latex[width=1.5mm]}] (2.614680117619929,2.215396846653664) -- (2.525737716274771,2.5334321824873482);
			\draw [->,line width=0.8pt, -{Latex[width=1.5mm]}] (1.78,5.2) -- (1.6832642984121484,5.5459021892405245);
			\draw [->,line width=0.8pt, -{Latex[width=1.5mm]}] (3.7440084533793025,2.0849490144970755) -- (3.4762169173767266,2.321844840031388);
			\draw [->,line width=0.8pt, -{Latex[width=1.5mm]}] (3.409307151521811,2.1236100999093526) -- (3.1778658642462254,2.40496636508901);
			\draw [->,line width=0.8pt, -{Latex[width=1.5mm]}] (0.9086338948254077,4.5932000472748635) -- (0.6219635737335659,4.8467965700962);
			\draw [->,line width=0.8pt, -{Latex[width=1.5mm]}] (1.2069548197822173,4.800944096631064) -- (0.9866644285688551,5.06874455028545);
			\draw (0.021893753930049997,4.984446894553236) node[anchor=north west] {$f(p_0)$};
			\draw (0.4870432635236744,5.469820295868323) node[anchor=north west] {$f(p_1)$};
			\begin{scriptsize}
			\draw [fill=ududff] (-0.06064225813758718,3.9182177174564456) circle (1.0pt);
			\draw [fill=ududff] (1.78,5.2) circle (1.5pt);
			\draw [fill=ududff] (3.7,3.98) circle (1.0pt);
			\draw [fill=ududff] (3.96,2.06) circle (1.0pt);
			\draw [fill=ududff] (0.5383767338770005,2.455229025420818) circle (1.0pt);
			\draw [fill=xdxdff] (2.614680117619929,2.215396846653664) circle (1.0pt);
			\draw [fill=xdxdff] (2.3155318727118956,3.2850746504288963) circle (1.5pt);
			\draw [fill=ududff] (0.88,3.82) circle (1.0pt);
			\draw [fill=ududff] (0.78,2.98) circle (1.0pt);
			\draw [fill=xdxdff] (1.963759348485022,4.54292345131072) circle (1.5pt);
			\draw [fill=xdxdff] (2.1431537600660575,3.901454803388197) circle (1.5pt);
			\draw [fill=ududff] (2.88,3.54) circle (1.0pt);
			\draw [fill=ududff] (2.94,3.98) circle (1.0pt);
			\draw [fill=xdxdff] (1.4522106732596802,4.971734811821746) circle (1.0pt);
			\draw [fill=xdxdff] (3.043348269995034,2.1658817238844277) circle (1.0pt);
			\draw [fill=xdxdff] (1.6086333625232443,5.080664089523283) circle (1.0pt);
			\draw [fill=xdxdff] (2.805161949135644,2.193394439382325) circle (1.0pt);
			\draw [fill=xdxdff] (0.9086338948254077,4.5932000472748635) circle (1.0pt);
			\draw [fill=xdxdff] (3.7440084533793025,2.0849490144970755) circle (1.0pt);
			\draw [fill=xdxdff] (3.409307151521811,2.1236100999093526) circle (1.0pt);
			\draw [fill=xdxdff] (1.2069548197822173,4.800944096631064) circle (1.0pt);
			\end{scriptsize}
			\end{tikzpicture}
		\caption{
			The generalised trajectory $\Gamma$ contains $v_1$, $v_2$, $v_3$ and $v_4$. It overlaps the edge $v_2v_3$.
			}
		\label{fig:generalised trajectories overlapping vertices}
	\end{figure}

	\begin{lemma}\label{lem:blocking vertices make V coding different}
		Suppose a vertex $v \in Q_n$ in the unfolding $Q^\infty$ blocks $\Gamma$ from the left (resp. right). Then if $\Gamma'$ is another generalised trajectory whose $V$-coding coincides with the $V$-coding of $\Gamma$ up to at least the $(n + 1)$-th term, then $v$ does not lie on the right side (resp. left side) of $\Gamma'$. 
		
		If $\Gamma$ does not contain any vertices blocking it from the left (resp. right) or overlap any reflecting edge, and there are no vertices in the open left (resp. right) $\epsilon$-strip of $\Gamma$, then any billiard trajectory in the open left (resp. right) $\epsilon$-strip of $\Gamma$ will have the same edge coding as $\Gamma$.

	\end{lemma}
	\begin{proof}
		Let us assume that $v$ blocks $\Gamma$ from the left, the case where $v$ blocks $\Gamma$ from the right being similar. Let $\Gamma'$ be another generalised trajectory satisfying the condition in the first statement of the Lemma. 
		Suppose on the contrary that $v$ lies on the right side of $\Gamma'$. If $v$ is the end vertex of a non-reflecting edge $e$ in $Q_n$, then Definition \ref{def:blocking vertices} implies that the interior of $e$ lies on the left of $\Gamma$.
		By Remark \ref{rem:V partition detects holes}, no holes of $Q_n$ lie between $\Gamma$ and $\Gamma'$. Thus, the edge $e$ must belong to the outer boundary of $Q$ (i.e. the unique connected component of $\partial Q$ not bounding a hole of $Q$) as illustrated in Figure \ref{fig:e is not reflecting}. However, this would force $\Gamma'$ to intersect some non-reflecting edge on the outer boundary, contradicting the condition on the $V$-coding of $\Gamma'$.

		On the other hand, if $v$ is an end vertex of a reflecting edge $e$, then $\Gamma$ and $\Gamma'$ exit a copy of $Q$ both through $e$. By Definition \ref{def:blocking vertices}, the other end vertex $u$ of $e$ must be on the right side of $\Gamma$.
		The vertex $u$ lies either on the trajectory $\Gamma'$ or on the left side of $\Gamma'$. 
		Geometrically, this means $\Gamma'$ and $\Gamma$ hit $e$ from different sides as illustrated in Figure \ref{fig:e is a reflecting edge}. This contradicts the requirement that $\Gamma$ and $\Gamma'$ must exit a copy of $Q$ both through $e$. 
		In conclusion, we deduce that $v$ does not lie on the right side of $\Gamma'$.

		\begin{figure}
			\begin{subfigure}[t]{0.49\textwidth}
				\begin{tikzpicture}[line cap=round,line join=round,>=triangle 45,x=1.0cm,y=1.0cm]
					\clip(0.72,0.32) rectangle (5.98,4.34);
					\fill[line width=2.pt,color=zzttqq,fill=zzttqq,fill opacity=0.10000000149011612] (3.8077697309298517,2.3625437437225774) -- (2.6,2.78) -- (1.62,2.2) -- (1.26,1.26) -- (3.92,0.66) -- (5.7,3.04) -- (4.02,3.88) -- cycle;
					\draw [line width=0.8pt,dotted] (2.72,0.62)-- (4.78,3.92);
					\draw [line width=0.8pt,color=zzttqq] (3.8077697309298517,2.3625437437225774)-- (2.6,2.78);
					\draw [line width=0.8pt,color=zzttqq] (2.6,2.78)-- (1.62,2.2);
					\draw [line width=0.8pt,color=zzttqq] (1.62,2.2)-- (1.26,1.26);
					\draw [line width=0.8pt,color=zzttqq] (1.26,1.26)-- (3.92,0.66);
					\draw [line width=0.8pt,color=zzttqq] (3.92,0.66)-- (5.7,3.04);
					\draw [line width=0.8pt,color=zzttqq] (5.7,3.04)-- (4.02,3.88);
					\draw [line width=0.8pt,color=zzttqq] (4.02,3.88)-- (3.8077697309298517,2.3625437437225774);
					\draw [line width=0.8pt,dotted] (1.64,0.84)-- (4.48,4.06);
					\draw [line width=1.6pt] (2.6,2.78)-- (3.8077697309298517,2.3625437437225774);
					\begin{scriptsize}
					\draw[color=black] (3.6,1.47) node {$\Gamma$};
					\draw [fill=xdxdff] (3.8077697309298517,2.3625437437225774) circle (1.0pt);
					\draw[color=xdxdff] (3.96,2.05) node {$v$};
					\draw[color=black] (3.5,3.17) node {$\Gamma'$};
					\end{scriptsize}
					\end{tikzpicture}
				\caption{}
				\label{fig:e is not reflecting}
			\end{subfigure}
			\begin{subfigure}[t]{0.49\textwidth}
				\begin{tikzpicture}[line cap=round,line join=round,>=triangle 45,x=1.0cm,y=1.0cm]
					\clip(-1.84,1.18) rectangle (3.14,5.66);
					\fill[line width=2.pt,color=zzttqq,fill=zzttqq,fill opacity=0.10000000149011612] (0.6432993185513162,2.4602947179079178) -- (1.14,4.28) -- (2.84,5.46) -- (1.36,1.5) -- (-1.58,2.24) -- (-0.4,3.38) -- cycle;
					\draw [line width=0.8pt,dotted] (0.62,1.14)-- (0.68,4.54);
					\draw [line width=0.8pt,color=zzttqq] (0.6432993185513162,2.4602947179079178)-- (1.14,4.28);
					\draw [line width=0.8pt,color=zzttqq] (1.14,4.28)-- (2.84,5.46);
					\draw [line width=0.8pt,color=zzttqq] (2.84,5.46)-- (1.36,1.5);
					\draw [line width=0.8pt,color=zzttqq] (1.36,1.5)-- (-1.58,2.24);
					\draw [line width=0.8pt,color=zzttqq] (-1.58,2.24)-- (-0.4,3.38);
					\draw [line width=0.8pt,color=zzttqq] (-0.4,3.38)-- (0.6432993185513162,2.4602947179079178);
					\draw [line width=0.8pt,dotted] (-1.22,1.52)-- (2.04,4.38);
					\draw [line width=1.6pt] (0.6432993185513162,2.4602947179079178)-- (1.14,4.28);
					\begin{scriptsize}
					\draw[color=black] (0.4,4.11) node {$\Gamma$};
					\draw [fill=xdxdff] (0.6432993185513162,2.4602947179079178) circle (1.0pt);
					\draw[color=xdxdff] (0.82,2.29) node {$v$};
					\draw [fill=ududff] (1.14,4.28) circle (1.0pt);
					\draw[color=ududff] (1.04,4.47) node {$u$};
					\draw[color=black] (1.72,3.65) node {$\Gamma'$};
					\end{scriptsize}
					\end{tikzpicture}
				\caption{}
				\label{fig:e is a reflecting edge}
			\end{subfigure}
			
			\caption{The vertex $v$ lies on $\Gamma$. The edge $e$ is indicated by the thicker edge.}\label{fig:blocking vertices make V coding different}
			\end{figure}
		
		For the second statement, let $l$ be the billiard flow in the open left $\epsilon$-strip of $\Gamma$. 
		We show that $e$ is a reflecting edge in $Q^\infty$ if and only if $l$ intersects the interior of $e$, thereby proving that $l$ and $\Gamma$ have the same edge coding.
		If $e$ is a reflecting edge in $Q^\infty$, then its end vertex $w$ on the left side of $\Gamma$ is not on $\Gamma$ for otherwise it would block $\Gamma$ from the left or $e$ would be contained in $\Gamma$, contradicting the assumption. Also, the vertex $w$ is not in the left $\epsilon$-strip of $\Gamma$ by assumption. Thus $l$ also intersects $e$.

		Conversely, if $l$ intersects the interior of an edge $e$, then the assumption implies that a vertex $v$ of $e$ lies either on $\Gamma$ or on the right side of $\Gamma$.
		In the latter case, $\Gamma$ clearly intersect the interior of $e$ and thus $e$ is a reflecting edge. If $v$ is on $\Gamma$, then $v$ must be a reflecting edge, for, if not, the edge $e$ would block $\Gamma$ from the left by Definition \ref{def:blocking vertices}, contradicting the assumption.
		We conclude that an edge $e$ in $Q^\infty$ is a reflecting edge if and only if $l$ intersects the interior of $e$, as desired. The proof of the analoguous statement with `left' replaced by `right' is similar.
	\end{proof}

	\section{Coding Parallel Trajectories}\label{sec:encoding parallel trajectories}

	In this section, we introduce the main new tools in our proof: the coding of two parallel trajectories which detects the positions of holes (Definition \ref{def:B coding}) and the construction of parallel generalised trajectories from a convergent sequence of codings (Lemma \ref{lemma:approximate_by_geodesic}).

	\paragraph{Definition of $\mathcal{B}$-codings.}

	Before discussing parallel trajectories, let us consider a slightly more general situation. Suppose $0 \leq N \leq \infty$ and $p_1 = (x_1, \theta_1)$ and $p_2 = (x_2, \theta_2)$ are two phase points in $\mathring{V}$ such that the first $N + 1$ terms in their edge codings coincide. Then we can define a coding $(\beta_n)_{0 \leq n \leq N - 1}$ associated to the pair $(p_1, p_2)$ as follows. 
	Let $(\xi_n)_{0 \leq n \leq N-1}$ and  $(\xi'_n)_{0 \leq n \leq N-1}$ be the first $N$ terms of the $V$-codings associated with $p_1$ and $p_2$ respectively. If $f^n(p_1)$ is on the left side of $f^n(p_2)$ relative to the orientation on $\partial Q$, then we put $\beta_n = (\xi_n, \xi'_n)$. Otherwise, let $\beta_n = (\xi'_n, \xi_n)$.
	Thus defined, the sequence $\beta = (\beta_n)_{0 \leq n \leq N-1}$ is a sequence of elements in the finite set
	\[\mathcal{B} = \bigsqcup_{a, b \in \mathcal{A}} I_{a,b} \times I_{a,b}.\]

	\begin{definition}\label{def:B coding}
		The sequence $\beta$ constructed above is called the $\mathcal{B}$-\textit{coding} associated with the pair $(p_1, p_2)$.
	\end{definition}

	Clearly, we can recover the edge coding and the $V$-codings of $p_1$ and $p_2$ from their $\mathcal{B}$-coding $\beta$. Moreover, the coding $\beta$ also contains information about the holes between the two billiard trajectories from $p_1$ and $p_2$. The following lemma is a direct consequence of Remark \ref{rem:V partition detects holes} and the definition of $\beta$.
	\begin{lemma}\label{lem:B coding detects holes}
		Let $(\beta_n)_{0 \leq n \leq N-1}$ be the $\mathcal{B}$-coding associated with a pair of phase points $(p_1, p_2)$ as above. Suppose $\beta_n$ is given by $(i,j) \in I_{a,b}^2$ for some $a,b \in \mathcal{A}$. Then $i \neq j$ if and only if there is a hole in $Q$ separating the trajectory from $f^n(p_1)$ to $f^{n+1}p_1)$ and the trajectory from $f^n(p_2)$ to $f^{n+1}(p_2)$.
	\end{lemma}

	\paragraph{Parallel trajectories and alternating orbits.}

	Next, we shall apply $\mathcal{B}$-codings to parallel trajectories - a special situation where the $\mathcal{B}$-coding can be interpreted as the coding of a single dynamical system (Corollary \ref{cor:single dynamical system on the alternating orbit}). The notion of `parallel trajectories' is made precise in the following definition.
	
	\begin{definition}\label{def:parallel phase points}
		Two phase points $(x_1, \theta_1)$ and $(x_2, \theta_2)$ are said to be \textit{parallel} if the following properties are satisfied.
		\begin{enumerate}
			\item The base points $x_1$ and $x_2$ are two disinct points on the same edge of $Q$ and $\theta_1 = \theta_2$.
			\item The parallel phase points $(x_1, \theta_1)$ and $(x_2, \theta_2)$ both have infinite forward orbit under $f$ and the same edge coding.
		\end{enumerate}
		In this case, the two unfolded trajectories $\Gamma_1$ and $\Gamma_2$ in $Q^\infty$ associated with $p_1$ and $p_2$ will be called \textit{parallel trajectories} from $p_1$ and $p_2$. The \textit{parallel separation} between $p_1$ and $p_2$ is defined as $|x_1 - x_2|\sin\theta_1 $. Geometrically, this is nothing but the perpendicular distance between $\Gamma_1$ and $\Gamma_2$.
	\end{definition}

	From now on until the end of this section, let $p_1$ and $p_2$ be parallel phase points with a parallel separation $L > 0$. Assume that $p_1$ is on the left side of $p_2$. Let $\beta = (\beta_n)_{n \geq 0}$ be the $\mathcal{B}$-coding associated with $(p_1, p_2)$. 
	
	The \textit{alternating orbit} $(\mathcal{P}_n)_{n \geq 0}$ of the parallel phase points $p_1$ and $p_2$ is defined as the sequence 
	\[p_1, f(p_2), f^{2}(p_1), f^{3}(p_2), f^{4}(p_1), \ldots.\]
	In other words, the sequence $(\mathcal{P}_n)_{n \geq 0}$ is given by
	\[\mathcal{P}_n= \begin{cases}
		f^n(p_1) & \quad \text{ if $n$ is even} \\
		f^n(p_2) & \quad \text{ if $n$ is odd} \\
	\end{cases} \qquad \text{for all $n \geq 0$}.\]
	We would like to interpret the alternating orbit associated with $(p_1, p_2)$ as the orbit of a single point in a new dynamical system. We will also interpret the $\mathcal{B}$-coding associated with parallel phase points as a natural coding of this dynamical system, which will be useful for the proof of lemma \ref{lemma:approximate_by_geodesic}. The main idea is to compose $f$ with a map $\tau$, whose job is to translate each phase point to the left by a distance calibrated by $L$, so that the alternating orbit will be the orbit of $p_1$ under the iterations of $\tau \circ f$.

	Since $p_1$ and $p_2$ are parallel phase points, by the law of reflection their images under the billiard map $f$ are still parallel with the same parallel separation $L$. Suppose for the moment that $f^n(p_2)$ is the phase point on the right. If $(x_n, \theta_n)$ is the coordinates of $f^n(p_2)$ as defined in Section \ref{sec:partition and encoding}, then the coordinates of $f^n(p_1)$ will be given by
	\[f^n(p_1) = \Big(x_n - \frac{L}{\sin(\theta_n)}, \theta_n\Big).\]
	This motivates us to define a parallel translation map $\tau$ associated with $(p_1, p_2)$ as follows.
	\begin{equation}\label{def:tau}
	\begin{split}
		\tau: \qquad F &\rightarrow V\\ 
		(x, \theta) &\mapsto \Big(x - \frac{L}{\sin \theta}, \theta\Big)
	\end{split}
	\end{equation}
	where 
	\begin{equation}\label{def:F}
		F = \bigcup_{a \in \mathcal{A}} \{(x, \theta) \in {E_a}\mid (x - L / \sin\theta , \theta) \in {E_a}\}
	\end{equation}
	is a finite union of some disjoint open subsets of $\mathring{V}$. The set $F$ is exactly the subset of $\mathring{V}$ on which $\tau(p)$ and $p$ always lie on the same edge. 
	Notice that the map $\tau$ is also a homeomorphism onto its image.

	This map $\tau$ helps define another family of open subsets $U^{i,j}_{a,b}$ in $V$ which partitions $\tau(F)$: within each $V_{a,b}^i$, we define subsets $U^{i,j}_{a,b}$ by
	\begin{equation}\label{def:Uabij}
		U^{i,j}_{a,b} = V_{a,b}^i \cap \tau(V_{a,b}^j \cap F).
	\end{equation}

	\begin{lemma}\label{lem:commutation lemma}
		For all $p \in U^{i,j}_{a,b}$, we have
		\begin{equation}\label{eq:commutation}
			\tau\circ f(p) = f \circ \tau^{-1}(p).
		\end{equation}
	\end{lemma}
	\begin{proof}
		Let $p \in U^{i,j}_{a,b}$ for some $i, j \in I_{a,b}$ and $a,b \in \mathcal{A}$. Note that $\tau^{-1}(p)$ and $p$ are parallel phase points with parallel separation $L$. Since the billiard map $f$ preserves the parallel separation between two parallel trajectories, $f(p)$ and $f(\tau^{-1}(p))$ have the same parallel separation $L$. On the other hand, each reflection inverts the orientation in the sense that what initially lies on the right side of $p$ on $e_a$ will be mapped by $f$ to the left side of $f(p)$ on $e_b$. Thus, the base point of $f(\tau^{-1}(p))$ is on the left side of the base point of $f(p)$ and we may obtain $f(\tau^{-1}(p))$ by translating $f(p)$ to the left using the map $\tau$. Therefore we have \eqref{eq:commutation} as desired.
		The geometrical meaning of this relation is illustrated in Figure~\ref{fig:commute}	
	\end{proof}
	
	\begin{figure}[]
		\centering
		\begin{tikzpicture}[line cap=round,line join=round,>=triangle 45,x=1.0cm,y=1.0cm]
			\clip(-2.434481563290093,-0.14379676887468248) rectangle (6.4493397064923155,4.564504601442356);
			\fill[line width=2.pt,color=zzttqq,fill=zzttqq,fill opacity=0.10000000149011612] (0.1,2.14) -- (3.44,2.28) -- (4.56,0.92) -- (3.94,0.3) -- (-1.04,0.34) -- (-1.24,1.26) -- cycle;
			\fill[line width=2.pt,color=ffffff,fill=ffffff,fill opacity=1.0] (1.24,1.56) -- (1.1,1.04) -- (2.3,1.04) -- cycle;
			\fill[line width=2.pt,color=zzttqq,fill=zzttqq,fill opacity=0.10000000149011612] (0.1,2.14) -- (3.44,2.28) -- (4.442259288424368,3.7289569761615002) -- (3.772549216121412,4.294897272532035) -- (-1.1866346911017251,3.8382847734268744) -- (-1.3089426587443624,2.904774858615507) -- cycle;
			\fill[line width=2.pt,color=ffffff,fill=ffffff,fill opacity=1.0] (1.1874636695540126,2.8133667406399887) -- (1.0044383993127641,3.3198267592526314) -- (2.2002290786742074,3.4202491230582) -- cycle;
			\draw [line width=0.8pt,color=zzttqq] (0.1,2.14)-- (3.44,2.28);
			\draw [line width=0.8pt,color=zzttqq] (3.44,2.28)-- (4.56,0.92);
			\draw [line width=0.8pt,color=zzttqq] (4.56,0.92)-- (3.94,0.3);
			\draw [line width=0.8pt,color=zzttqq] (3.94,0.3)-- (-1.04,0.34);
			\draw [line width=0.8pt,color=zzttqq] (-1.04,0.34)-- (-1.24,1.26);
			\draw [line width=0.8pt,color=zzttqq] (-1.24,1.26)-- (0.1,2.14);
			\draw [line width=0.8pt] (1.1,1.04)-- (2.3,1.04);
			\draw [line width=0.8pt] (2.3,1.04)-- (1.24,1.56);
			\draw [line width=0.8pt] (1.24,1.56)-- (1.1,1.04);
			\draw [line width=0.8pt,color=zzttqq] (0.1,2.14)-- (3.44,2.28);
			\draw [line width=0.8pt,color=zzttqq] (3.44,2.28)-- (4.442259288424368,3.7289569761615002);
			\draw [line width=0.8pt,color=zzttqq] (4.442259288424368,3.7289569761615002)-- (3.772549216121412,4.294897272532035);
			\draw [line width=0.8pt,color=zzttqq] (3.772549216121412,4.294897272532035)-- (-1.1866346911017251,3.8382847734268744);
			\draw [line width=0.8pt,color=zzttqq] (-1.1866346911017251,3.8382847734268744)-- (-1.3089426587443624,2.904774858615507);
			\draw [line width=0.8pt,color=zzttqq] (-1.3089426587443624,2.904774858615507)-- (0.1,2.14);
			\draw [line width=0.8pt] (1.0044383993127641,3.3198267592526314)-- (2.2002290786742074,3.4202491230582);
			\draw [line width=0.8pt] (2.2002290786742074,3.4202491230582)-- (1.1874636695540126,2.8133667406399887);
			\draw [line width=0.8pt] (1.1874636695540126,2.8133667406399887)-- (1.0044383993127641,3.3198267592526314);
			\draw [->,line width=0.8pt, -{Latex[width=1.5mm]}] (0.5930569547869818,2.160667057985083) -- (0.5435085564368582,2.837219128375695);
			\draw [->,line width=0.8pt, -{Latex[width=1.5mm]}] (2.888564244656307,2.2568859264227195) -- (2.84,2.92);
			\draw [->,line width=0.8pt, -{Latex[width=1.5mm]}] (0.7274361745020563,0.3258037255060076) -- (0.676647106879365,1.019296350848455);
			\draw [->,line width=0.8pt, -{Latex[width=1.5mm]}] (3.031345476321013,0.30729842990906814) -- (2.973242408725599,1.1006591032440256);
			\draw [line width=0.8pt,dotted] (0.7274361745020563,0.3258037255060076)-- (0.35562800961506674,5.402609275605069);
			\draw [line width=0.8pt,dotted] (3.031345476321013,0.30729842990906814)-- (2.652554146053284,5.479454690697797);
			\draw (0.15508419038427698,0.9713272398846161) node[anchor=north west] {{\scriptsize $p$}};
			\draw (1.753428602939271,0.9837175066486084) node[anchor=north west] {{\scriptsize $\tau^{-1}(p)$}};
			\draw (-0.26618487959145787,3.01572125594333) node[anchor=north west] {{\scriptsize $f(p)$}};
			\draw (1.4932330008954346,2.8670380547754237) node[anchor=north west] {{\scriptsize $\tau \circ f(p)$}};
			\draw (3.103967680214421,2.2475247165758137) node[anchor=north west] {{$e_b$}};
			\draw (3.4261146160782183,0.3270333681570214) node[anchor=north west] {{$e_a$}};
			\begin{scriptsize}
			\draw [fill=ududff] (0.1,2.14) circle (1.0pt);
			\draw [fill=ududff] (3.44,2.28) circle (1.0pt);
			\draw [fill=ududff] (4.56,0.92) circle (1.0pt);
			\draw [fill=ududff] (3.94,0.3) circle (1.0pt);
			\draw [fill=ududff] (-1.04,0.34) circle (1.0pt);
			\draw [fill=ududff] (-1.24,1.26) circle (1.0pt);
			\draw [fill=ududff] (1.24,1.56) circle (1.0pt);
			\draw [fill=ududff] (1.1,1.04) circle (1.0pt);
			\draw [fill=ududff] (2.3,1.04) circle (1.0pt);
			\draw [fill=ududff] (0.1,2.14) circle (1.0pt);
			\draw [fill=ududff] (3.44,2.28) circle (1.0pt);
			\draw [fill=ududff] (4.442259288424368,3.7289569761615002) circle (1.0pt);
			\draw [fill=ududff] (3.772549216121412,4.294897272532035) circle (1.0pt);
			\draw [fill=ududff] (-1.1866346911017251,3.8382847734268744) circle (1.0pt);
			\draw [fill=ududff] (-1.3089426587443624,2.904774858615507) circle (1.0pt);
			\draw [fill=ududff] (1.1874636695540126,2.8133667406399887) circle (1.0pt);
			\draw [fill=ududff] (1.0044383993127641,3.3198267592526314) circle (1.0pt);
			\draw [fill=ududff] (2.2002290786742074,3.4202491230582) circle (1.0pt);
			\draw [fill=ududff] (1.0044383993127641,3.3198267592526314) circle (1.0pt);
			\draw [fill=ududff] (2.2002290786742074,3.4202491230582) circle (1.0pt);
			\draw [fill=ududff] (2.2002290786742074,3.4202491230582) circle (1.0pt);
			\draw [fill=ududff] (1.1874636695540126,2.8133667406399887) circle (1.0pt);
			\draw [fill=ududff] (1.1874636695540126,2.8133667406399887) circle (1.0pt);
			\draw [fill=ududff] (1.0044383993127641,3.3198267592526314) circle (1.0pt);
			\draw [fill=uuuuuu] (0.7274361745020563,0.3258037255060076) circle (1.0pt);
			\draw [fill=uuuuuu] (3.031345476321013,0.30729842990906814) circle (1.0pt);
			\draw [fill=uuuuuu] (2.888564244656307,2.2568859264227195) circle (1.0pt);
			\draw [fill=uuuuuu] (0.5930569547869818,2.160667057985083) circle (1.0pt);
			\end{scriptsize}
			\end{tikzpicture}
		\caption{}
		\label{fig:commute}
	\end{figure}

	Using Lemma \ref{lem:commutation lemma}, we can interpret the parallel trajectories from $p_1$ and $p_2$ as an orbit in a single dynamical system with a natural coding given by the $\CB$-coding of $p_1$ and $p_2$.
	\begin{corollary}\label{cor:single dynamical system on the alternating orbit}
		The alternating orbit $(\mathcal{P}_n)_{n \geq 0}$ is equal to the orbit of $p_1$ under $\tau \circ f$. 
		Moreover, if $\beta$ is the $\CB$-coding of $(p_1, p_2)$, then for all $n \geq 0$ and $(i,j) \in \bigsqcup_{a,b \in \mathcal{A}} I_{a,b} \times I_{a,b}$, we have 
		\[(\tau \circ f)^n(p_1) \in U^{i,j}_{a,b} \quad \text{ if and only if } \quad \beta_n = (i,j) \in I_{a,b}.\]
	\end{corollary}
	\begin{proof}
		Let $\alpha = (\alpha_n)_{n \geq 0}$ be the edge coding of $p_1$ and $p_2$. 
		The phase point $p_1$ is initially on the left hand side of $p_2$. Since each reflection inverts the orientation, the definition of the alternating orbit ensures that $\CP_n$ is always on the left side of the parallel pair $f^n(p_1)$ and $f^n(p_2)$. Thus, by definition of $\CB$-coding, we have $\CP_n \in V_{\alpha_{n},\alpha_{n+1}}^i$ and $\tau\inv(\CP_n) \in V_{\alpha_{n},\alpha_{n+1}}^j$ where $\beta_n = (i,j) \in I_{\alpha_{n}, \alpha_{n+1}}$. By Lemma \ref{lem:commutation lemma}, we also have $\CP_{n+1} = \tau \circ f (\CP_n)$ for $n \geq 0$ and $\CP_0 = p_1$. This shows that $(\mathcal{P}_n)_{n \geq 0}$ is equal to the orbit of $p_1$ under $\tau \circ f$.
	\end{proof}

	\paragraph{Construction of generalised trajectories.}

	For the rest of this section, we apply the notion of generalised trajectories defined in Section \ref{sec:generalised trajectories} to study parallel trajectories. Let $\beta$ be the $\mathcal{B}$-coding of parallel phase points $p_1$ and $p_2$. Let $S: \mathcal{B}^{\mathbb{N}} \rightarrow \mathcal{B}^{\mathbb{N}}$ be the left shift map by one index, i.e.
	\[S((\nu_1, \nu_2, \nu_3, \ldots)) = (\nu_2, \nu_3, \nu_4, \ldots) \qquad \text{for all $\nu = (\nu_1, \nu_2, \nu_3, \ldots) \in \mathcal{B}^\N$}.\] 
	Define $\Omega$ to be the $\omega$-limit of the action of $S$ on $\beta$, i.e. 
	\begin{equation}\label{eq:definition of omega limit}
		\Omega = \bigcap_{n \geq 0} \ol{\{S^k(\beta) \mid k \geq n\}}.
	\end{equation}
	Every $\omega \in \Omega$ is the limit of a sequence $(S^{i_n}(\beta))_{n \geq 1}$ with $i_1 < i_2 < \cdots$. The following lemma shows that this approximation on the level of symbolic encodings can be converted into a geometrical result, where a unique parallel pair of generalised trajectories can be associated with $\omega$ and this pair can be approximated arbitrarily well by parts of the physical trajectories from $p_1$ and $p_2$.

	\begin{lemma}\label{lemma:approximate_by_geodesic}
		Let $\beta$ be the $\mathcal{B}$-coding of parallel phase points $p_1$ and $p_2$. Let $\Omega$ be the $\omega$-limit of $\beta$ under $S$ as defined in \eqref{eq:definition of omega limit}. Then for every $\omega \in \Omega$, there exists a sequence $0 < j_1 < j_2 < \cdots$ such that 
		the following properties hold for $q_m := (\tau \circ f)^{j_m} (p_1)$.
		\begin{enumerate}
			\item for all $m \geq 1$, the first $m$ terms of the $\mathcal{B}$-coding associated with the parallel phase points $q_m$ and $\tau\inv(q_m)$ coincide with the first $m$ terms of $\omega$.
			\item There exists a sequence $(z_k)_{k \geq 0}$ of generalised phase points such that for all $k \geq 0$, \[(\tau \circ f)^k(q_m) \rightarrow z_k\] as $m \rightarrow \infty$. This is illustrated by Figure \ref{fig:approx by geodesics}.
			\item The sequences 
			\[z_0, \tau^{-1}(z_1),z_2,\tau^{-1}(z_3),z_4, \tau^{-1}(z_5),\cdots\] 
			and 
			\[\tau^{-1}(z_0),z_1,\tau^{-1}(z_2),z_3 ,\tau^{-1}(z_4), z_5,\cdots\]
			are two generalised trajectories.
		\end{enumerate}
	\end{lemma}

	\begin{figure}[]
		\centering
		\begin{tikzpicture}[line cap=round,line join=round,>=triangle 45,x=1.2cm,y=1.2cm]
			\clip(-3.543991487611119,-0.43655350245580354) rectangle (9.207249053604746,1.9307654165884738);
			\draw [line width=0.8pt] (-2.3,1.74)-- (-2.24,-0.22);
			\draw [line width=0.8pt] (-0.4,1.76)-- (-0.68,-0.22);
			\draw [line width=0.8pt] (1.26,1.78)-- (1.48,-0.1);
			\draw [line width=0.8pt] (4.145301268658998,1.7365275830150164)-- (3.843267163447756,-0.005202423703129025);
			\draw [line width=0.8pt] (5.635336187701127,1.7667309935361404)-- (5.806488847320831,0.09547561136728401);
			\draw [line width=0.8pt] (7.26,1.82)-- (7.32,-0.12);
			\draw [line width=0.8pt] (8.707064380256549,1.2076128548229907)-- (8.44,-0.22);
			\draw [line width=0.8pt,dotted] (-2.2791844377405597,1.0600249661916157)-- (9.179836098981433,1.3633188216036503);
			\draw [line width=0.8pt] (8.707064380256549,1.2076128548229907)-- (8.282870049869437,2.076201245615645);
			\draw [line width=0.8pt,dotted] (-2.2600990158372247,0.43656785068267073)-- (9.180082817580361,0.7393630890499198);
			\draw [->,line width=0.8pt, -{Latex[width=1.5mm]}] (-2.2760312614052163,0.957021205903741) -- (-1.7986459835373858,0.961464959504541);
			\draw [->,line width=0.8pt, -{Latex[width=1.5mm]}] (-0.5970778294382274,0.3663782061153929) -- (-0.26978349882067176,0.36942483412767513);
			\draw [->,line width=0.8pt, -{Latex[width=1.5mm]}] (1.3523536451047546,0.9907961236502758) -- (1.6910107477553402,0.9939485223293082);
			\draw [->,line width=0.8pt, -{Latex[width=1.5mm]}] (3.914986464151645,0.40837887702262005) -- (4.352083728318066,0.4124476085725636);
			\draw [->,line width=0.8pt, -{Latex[width=1.5mm]}] (5.710644713120435,1.031365392382924) -- (6.217383485308945,1.0360823837690747);
			\draw [->,line width=0.8pt, -{Latex[width=1.5mm]}] (7.302683092530681,0.4399133415080243) -- (7.746999375970695,0.4440492714915571);
			\draw [->,line width=0.8pt, -{Latex[width=1.5mm]}] (-2.2791844377405597,1.0600249661916157) -- (-1.8016462448576172,1.072664301216963);
			\draw [->,line width=0.8pt, -{Latex[width=1.5mm]}] (-2.2600990158372247,0.4365678506826707) -- (-1.7746523140364887,0.44941650570640274);
			\draw (-1.962192028067885,0.004628659729720897) node[anchor=north west] {$Q_0$};
			\draw (-0.046815811750227615,0.03691028135305195) node[anchor=north west] {$Q_1$};
			\draw (2.2882214856707366,0.5856978489496799) node[anchor=north west] {$\cdots$};
			\draw (5.989847431812726,0.069191902976383) node[anchor=north west] {$Q_k$};
			\draw (4.752385269585027,0.05843136243527265) node[anchor=north west] {$Q_{k-1}$};
			\draw [->,line width=0.8pt, -{Latex[width=1.5mm]}] (8.679263071614768,1.0589988556075962) -- (9.126004768298968,1.0631573626150488);
			\draw (7.926744729212604,0.09071298405860372) node[anchor=north west] {{\scriptsize $\cdots$}};
			\draw (-2.9844433794733765,1.5218648760262805) node[anchor=north west] {{\scriptsize $q_m$}};
			\draw (-2.995203920014487,1.166767038169639) node[anchor=north west] {{\scriptsize $z_0$}};
			\draw (-3.5547520281522296,0.714824335443004) node[anchor=north west] {{\scriptsize $\tau^{-1}(q_m)$}};
			\draw (-0.5848428388057494,0.3920081192096936) node[anchor=north west] {{\scriptsize $z_1$}};
			\draw (1.427378242381902,0.9838378489707629) node[anchor=north west] {{\scriptsize $z_2$}};
			\draw (3.945344729001744,0.37048703812747286) node[anchor=north west] {{\scriptsize $z_{k-1}$}};
			\draw (5.699312837202744,1.8738075787529152) node[anchor=north west] {{\scriptsize $(\tau \circ f)^{k}(q_m)$}};
			\draw (5.796157702072739,1.016119470594094) node[anchor=north west] {{\scriptsize $z_k$}};
			\draw (7.356436080533752,0.37048703812747286) node[anchor=north west] {{\scriptsize $z_{k+1}$}};
			\draw [->,line width=0.8pt, -{Latex[width=1.5mm]}] (5.686119365939423,1.2708481942680818) -- (6.164725214772242,1.2835157876849947);
			\draw [->,line width=0.8pt, -{Latex[width=1.5mm]}] (-0.5808668123218005,0.48101325572441184) -- (-0.2628899675416194,0.4894293697261347);
			\draw [->,line width=0.8pt, -{Latex[width=1.5mm]}] (1.3330642396906054,1.1556328608257347) -- (1.68541929735553,1.1649588868652774);
			\draw (-0.477237433394645,1.080682713840756) node[anchor=north west] {{\scriptsize $(\tau \circ f) (q_m)$}};
			\draw (1.405857161299681,1.7478362273895978) node[anchor=north west] {{\scriptsize $(\tau \circ f)^2(q_m)$}};
			\draw [->,line width=0.8pt, -{Latex[width=1.5mm]}] (3.9483701936874147,0.6008917173455617) -- (4.37534989480781,0.612192885654012);
			\draw (3.592950134412848,1.0882881192518595) node[anchor=north west] {{\scriptsize $(\tau \circ f)^{k-1}(q_m)$}};
			\begin{scriptsize}
			\draw [fill=uuuuuu] (-2.2791844377405597,1.0600249661916157) circle (1.0pt);
			\draw [fill=uuuuuu] (-2.2760312614052163,0.957021205903741) circle (1.0pt);
			\draw [fill=uuuuuu] (-2.2600990158372247,0.4365678506826707) circle (1.0pt);
			\draw [fill=uuuuuu] (5.710644713120435,1.031365392382924) circle (1.0pt);
			\draw [fill=uuuuuu] (5.686119365939423,1.2708481942680818) circle (1.0pt);
			\draw [fill=uuuuuu] (-0.5970778294382274,0.3663782061153929) circle (1.0pt);
			\draw [fill=uuuuuu] (1.3523536451047546,0.9907961236502758) circle (1.0pt);
			\draw [fill=uuuuuu] (3.914986464151645,0.40837887702262005) circle (1.0pt);
			\draw [fill=uuuuuu] (7.302683092530681,0.4399133415080243) circle (1.0pt);
			\draw [fill=uuuuuu] (8.679263071614768,1.0589988556075962) circle (1.0pt);
			\draw [fill=uuuuuu] (-0.5808668123218005,0.48101325572441184) circle (1.0pt);
			\draw [fill=uuuuuu] (1.3330642396906054,1.1556328608257347) circle (1.0pt);
			\draw [fill=uuuuuu] (3.9483701936874147,0.6008917173455617) circle (1.0pt);
			\end{scriptsize}
			\end{tikzpicture}
		\caption{}
		\label{fig:approx by geodesics}
	\end{figure}

	The proof of Lemma \ref{lemma:approximate_by_geodesic} requires the following result, whose proof uses elementary geometry and can be found in the appendix.

	\begin{lemma}
		\label{lemma:extend}
		For each $a,b \in \mathcal{A}$ and $i,j \in I_{a,b}$, the billiard map $f$ restricted onto $U^{i,j}_{a,b}$ is uniformly continuous. Thus the map $f|_{U^{i,j}_{a,b}}$ can be continuously extended to $\ol{U^{i,j}_{a,b}}$.
	\end{lemma}

	In addition, we can also extend the map $\tau : F \rightarrow V$ defined by (\ref{def:tau}) continuously to $\ol{F}$, where $F$ is the open subset of $V$ defined in (\ref{def:F}). By Lemma~\ref{lemma:extend}, we may compose the extended map $f$ with $\tau$ to get a continuous map from $\ol{U^{i,j}_{a,b}}$ to $\ol{V}$.

	\begin{proof}[Proof of Lemma \ref{lemma:approximate_by_geodesic}]
		It follows from the definition of $\Omega$ that there exists a sequence $0 \leq i_1 < i_2 < \cdots$ such that $(S^{i_n}(\beta))_{n \geq 1}$ converges to $\omega \in \Omega$.
		Hence, for all $m \geq 1$ there exists $n(m) \geq 0$ such that
		\[S^{i_{n(m)}}(\beta)_k = \omega_k \quad \text{ for } \quad 0 \leq k \leq m.\]
		By definition of $\mathcal{B}$-coding, for $0 \leq k \leq m$, the phase point $(\tau \circ f)^{i_{n(m)} + k}(p_1)$ lies in $U^{i,j}_{a,b}$ where $a,b,i,j$ are such that $\omega_k = (i,j) \in I_{a,b} \times I_{a,b}$. 
		Let us set \[j_m := i_{n(m)}.\]
		Then $q_m := (\tau \circ f)^{j_m}(p_1)$ satisfies (1) by construction.
		To simplify the notation, we also define
		\begin{equation*}
			U_k = U^{i_k,j_k}_{a_k,b_k} \quad
			\text{where} \quad \omega_k = (i_k, j_k) \in I_{a_k, b_k} \times I_{a_k, b_k},
		\end{equation*}
		We can therefore write
		\[q_{m'} \in \bigcap_{k = 0}^{m} (\tau \circ f)^{-k}(U_k) \quad \text{ for all } m' \geq m.\]
		Since $q_m$ is in the relatively compact set $U_0$ for all $m \geq 1$, by extracting a subsequence, we may assume that $q_m$ converges to a generalised phase point $z_0 \in \ol{U_0}$ as $m \rightarrow \infty$. The continuity of $\tau \circ f$, which is guaranteed by Lemma~\ref{lemma:extend} and the remarks after it, allows us to define the generalised phase point 
		\begin{equation}\label{eq:convergence to z_k}
			z_k := \lim_{m \rightarrow \infty} (\tau\circ f)^k(q_m) \in \ol{U_k}
		\end{equation}
		for all $k \geq 1$.
		By construction, the sequence $(z_k)$ satisfies the required properties. In particular, we deduce from \eqref{eq:convergence to z_k} and Lemma \ref{lem:commutation lemma} that the sequences \[z_0, \tau^{-1}(z_1),z_2,\tau^{-1}(z_3), z_4,\ldots \qquad \text{and} \qquad \tau^{-1}(z_0),z_1,\tau^{-1}(z_2),z_3,\tau^{-1}(z_4), \ldots\] are generalised trajectories.
	\end{proof}

	As an application of Lemma \ref{lemma:approximate_by_geodesic}, we prove the following result which will be used in the proof of the main theorem in Section \ref{sec:the proof}. 

	\begin{corollary}\label{cor:rule out periodic omega}
		Let $\beta$ be the $\mathcal{B}$-coding of parallel phase points $p_1$ and $p_2$. Suppose the edge coding of $p_1$ and $p_2$ is non-periodic. Then the $\omega$-limit set $\Omega$ of $\beta$ under the left shift map $S$ does not contain periodic elements.
	\end{corollary}

	In preparation for the proof of Corollary \ref{cor:rule out periodic omega}, we recall the following result on periodic edge codings in \cite[Theorem 1]{Galperin1995}. This result can be readily extended to generalised trajectories.

	\begin{theorem}\label{thm:periodic coding implies periodic trajectory}
		If a phase point $p$ has infinite forward orbit under the billiard flow and the associated edge coding is periodic, then the orbit of $p$ under the billiard map $f$ is periodic.
	\end{theorem}

	\begin{proof}[Proof of Corollary \ref{cor:rule out periodic omega}]
		Suppose there exists a periodic $\omega \in \Omega$. We show that the edge coding of $p_1$ and $p_2$ must be periodic. Apply Lemma \ref{lemma:approximate_by_geodesic} to $\omega$ to find a subsequence $(q_m)_{m \geq 1}$ from the sequence of phase points $((\tau \circ f)^n p_1)_{n \geq 1}$ and a sequence of generalised phase points $(z_n)_{n \geq 1}$ satisfying the conclusion of Lemma \ref{lemma:approximate_by_geodesic}. Construct the infinite corridor $Q^\infty$ by unfolding the polygon $Q$ according to $\omega$. 
		
		By periodicity of $\omega$, we have $S^T(\omega) = \omega$ for some $T > 1$. We may assume $T$ is even since if $T$ is odd it suffices to consider $2T$. Hence, the polygon $Q_T$ in the unfolding $Q^\infty$ is obtained from $Q$ by an even number of reflections and thus has the same orientation as $Q_0$.
		Theorem \ref{thm:periodic coding implies periodic trajectory} extended to generalised trajectories implies that the unfolding $Q^\infty$ is periodic and $z_n = z_{n + T}$ for all $n \geq 0$. In particular, we have that $Q_T$ is a parallel translation of $Q_0$ (See Figure \ref{fig:omega periodic}). By periodicity, the unfolding $Q^\infty$ is obtained by joining infinitely many copies of $Q_0^{T-1}$ end to end. 

		Choose any $m \geq T$. The phase points $f^T(q_m)$ and $q_m$ lie on the same edge of $Q$ by Lemma \ref{lemma:approximate_by_geodesic} point (1) and they are parallel. By point (2) of Lemma \ref{lemma:approximate_by_geodesic} and the periodicity $z_n = z_{n + T}$, both $f^T(q_m)$ and $q_m$ converge to $z_0$ and both $f^T(\tau\inv(q_m))$ and $\tau\inv(q_m)$ converge to $\tau\inv(z_0)$. With a sufficiently large $m$, we may assume that the base point of either $f^T(q_m)$ or $f^T(\tau^{-1}(q_m))$ lies between the base points of $q_m$ and $\tau\inv(q_m)$. 

		Suppose for example that the base point of $f^T(q_m)$ lies strictly between the base points of $q_m$ and $\tau\inv(q_m)$ as shown in Figure \ref{fig:omega periodic}. Consider the four parallel physical trajectories $S_m$, $l$, $S'_m$ and $l'$ from $q_m$, $f^T(q_m)$, $\tau\inv(q_m)$ and $f^T(\tau^{-1}(q_m))$ respectively. 
		If $l$ intersects an edge $e$ in $Q^T_0$, then so does $l'$ by edge coding. This forces $S_m'$ to intersect $e$ since $S_m'$ lies between $l$ and $l'$.
		Similarly, if $S_m$ intersects an edge in $Q^T_0$, then $l$ must intersect this edge, too. It follows that the first $T$ terms of the edge coding of $f^T(q_m)$ coincide with the first $T$ terms of the edge coding of $q_m$. This further implies that the base point of $f^{2T}(q_m)$ lies between the base points of $f^T(q_m)$ and $f^T(\tau^{-1}(q_m))$.

		By induction on $n \geq 1$, we can deduce that the first $T$ terms of the edge codings of $f^{nT}(q_m)$ and $f^{(n+1)T}(q_m)$ coincide, and in particular, the edge coding of $q_m$ is periodic with period $\leq T$. By Theorem \ref{thm:periodic coding implies periodic trajectory}, the physical trajectory from $q_m$ is periodic, and thus the edge coding associated to $p_1$ and $p_2$ is also periodic.
	\begin{figure}[]
		\centering
		\begin{tikzpicture}[line cap=round,line join=round,>=triangle 45,x=1.0cm,y=1.0cm]
			\clip(-5.764933592518881,0.5195969900826466) rectangle (7.560969957905934,4.162010627198767);
			\fill[line width=2.pt,color=uuuuuu,fill=uuuuuu,fill opacity=0.10000000149011612] (-2.56,2.86) -- (-1.6255453986239186,3.328346911422567) -- (-0.8698529835023866,3.063854566130031) -- (-0.8787689573631154,0.8083249003418513) -- (-1.3666195877186724,0.7028096462182266) -- (-3.059745471711411,0.7287248383201561) -- cycle;
			\fill[line width=2.pt,color=uuuuuu,fill=uuuuuu,fill opacity=0.10000000149011612] (5.088617914576904,3.430173652900992) -- (6.023072515952985,3.898520564323559) -- (6.778764931074517,3.634028219031023) -- (6.769848957213789,1.3784985532428433) -- (6.281998326858231,1.2729832991192187) -- (4.588872442865492,1.298898491221148) -- cycle;
			\draw [line width=0.8pt] (-2.56,2.86)-- (-3.059745471711411,0.7287248383201562);
			\draw [line width=0.8pt] (5.088617914576902,3.430173652900992)-- (4.588872442865492,1.298898491221148);
			\draw [line width=0.8pt,dotted] (-2.623707292055822,2.588306154107852)-- (10.61577286767017,2.8098874120112165);
			\draw [line width=0.8pt,dotted] (-2.810042256668076,1.7936394601394463)-- (10.569513253324526,2.0175650753694483);
			\draw [line width=0.8pt,dotted] (-2.914130034992727,1.3497340940376554)-- (10.506162459832268,1.5743414998087855);
			\draw [line width=0.8pt,dotted] (-2.7277950703804743,2.144400788006061)-- (10.453057548747484,2.365000831840839);
			\draw [line width=0.8pt,color=uuuuuu] (-2.56,2.86)-- (-1.6255453986239186,3.328346911422567);
			\draw [line width=0.8pt,color=uuuuuu] (-1.6255453986239186,3.328346911422567)-- (-0.8698529835023866,3.063854566130031);
			\draw [line width=0.8pt,color=uuuuuu] (-0.8698529835023866,3.063854566130031)-- (-0.8787689573631154,0.8083249003418513);
			\draw [line width=0.8pt,color=uuuuuu] (-0.8787689573631154,0.8083249003418513)-- (-1.3666195877186724,0.7028096462182266);
			\draw [line width=0.8pt,color=uuuuuu] (-1.3666195877186724,0.7028096462182266)-- (-3.059745471711411,0.7287248383201561);
			\draw [line width=0.8pt,color=uuuuuu] (-3.059745471711411,0.7287248383201561)-- (-2.56,2.86);
			\draw [line width=0.8pt,color=uuuuuu] (5.088617914576904,3.430173652900992)-- (6.023072515952985,3.898520564323559);
			\draw [line width=0.8pt,color=uuuuuu] (6.023072515952985,3.898520564323559)-- (6.778764931074517,3.634028219031023);
			\draw [line width=0.8pt,color=uuuuuu] (6.778764931074517,3.634028219031023)-- (6.769848957213789,1.3784985532428433);
			\draw [line width=0.8pt,color=uuuuuu] (6.769848957213789,1.3784985532428433)-- (6.281998326858231,1.2729832991192187);
			\draw [line width=0.8pt,color=uuuuuu] (6.281998326858231,1.2729832991192187)-- (4.588872442865492,1.298898491221148);
			\draw [line width=0.8pt,color=uuuuuu] (4.588872442865492,1.298898491221148)-- (5.088617914576904,3.430173652900992);
			\draw [->,line width=0.8pt, -{Latex[width=1.5mm]}] (-2.623707292055822,2.588306154107852) -- (-1.9070927665688975,2.6002997026515664);
			\draw [->,line width=0.8pt, -{Latex[width=1.5mm]}] (-2.7277950703804743,2.144400788006061) -- (-1.922776364499923,2.157873904422472);
			\draw [->,line width=0.8pt, -{Latex[width=1.5mm]}] (-2.810042256668076,1.7936394601394463) -- (-1.8922038657466536,1.8090007721213948);
			\draw [->,line width=0.8pt, -{Latex[width=1.5mm]}] (-2.914130034992727,1.3497340940376554) -- (-1.8840184152400934,1.3669744558745196);
			\draw (-3.3662709534424144,3.118148182415489) node[anchor=north west] {$q_m$};
			\draw (-4.454553076727108,2.1631250946350424) node[anchor=north west] {$\tau^{-1}(q_m)$};
			\draw (-4.099195648715779,2.673951397401328) node[anchor=north west] {$f^T(q_m)$};
			\draw (-5.342946646755428,1.6745086311194655) node[anchor=north west] {$f^T(\tau^{-1}(q_m))$};
			\draw (-1.9670510806478088,1.3857807208602608) node[anchor=north west] {$Q_0$};
			\draw (5.58429426459292,3.7178138421846065) node[anchor=north west] {$Q_T$};
			\draw [->,line width=0.8pt, -{Latex[width=1.5mm]}] (4.9208228441964295,2.714574440907053) -- (5.627760906561677,2.7264060402771833);
			\draw [->,line width=0.8pt, -{Latex[width=1.5mm]}] (4.734487879584176,1.9199077469386476) -- (5.4966834559992295,1.9326641582594015);
			\draw (0.9868575396963587,3.4512957711761096) node[anchor=north west] {$S_m$};
			\draw (2.252818376986716,2.718371075902744) node[anchor=north west] {$l$};
			\draw (1.0756968966991909,2.2297546123871665) node[anchor=north west] {$S'_m$};
			\draw (2.2306085377360083,1.519039756364509) node[anchor=north west] {$l'$};
			\begin{scriptsize}
			\draw [fill=uuuuuu] (4.9208228441964295,2.714574440907053) circle (1.0pt);
			\draw [fill=uuuuuu] (-2.7277950703804743,2.144400788006061) circle (1.0pt);
			\draw [fill=uuuuuu] (-2.623707292055822,2.588306154107852) circle (1.0pt);
			\draw [fill=uuuuuu] (-2.810042256668076,1.7936394601394463) circle (1.0pt);
			\draw [fill=uuuuuu] (-2.914130034992727,1.3497340940376554) circle (1.0pt);
			\draw [fill=uuuuuu] (4.734487879584176,1.9199077469386476) circle (1.0pt);
			\end{scriptsize}
			\end{tikzpicture}
		\caption{}
		\label{fig:omega periodic}
	\end{figure}
	\end{proof}

	\section{Proof of the Main Theorem} \label{sec:the proof}

	In this section, we take on the proof of the main theorem.

	\begin{theorem*}\label{thm:main theorem}
		Let $Q$ be a polygon such that all the holes of $Q$ have non-zero minimal diameters. Suppose the edges of $Q$ are indexed by $\mathcal{A}$. Then for any non-periodic sequence $\alpha$ in $\mathcal{A}$, the set $X(\alpha)$ contains at most one point, i.e. there is at most one point $p \in V$ whose edge coding is equal to $\alpha$.
	\end{theorem*}

	The starting point of the proof is the observation that $X(\alpha)$ consists of parallel phase points. This well-known result is proven for example in \cite[Lemma 1]{Galperin1995}. 

	\begin{lemma}\label{lem:rule out non-parallel phase points}
		If two phase points $(x, \theta)$ and $(x', \theta')$ with infinite forward orbit under $f$ do not have the same direction, that is $\theta \neq \theta'$, then they do not have the same edge coding.
	\end{lemma}

	Let us assume $Q$ is a polygon satisfying the condition of the main theorem. Let $\delta > 0$ be smaller than the minimal diameter of every hole in $Q$. Our proof of the main theorem proceeds by contradiction. Suppose towards a contradiction that there exists at least one non-periodic $\alpha \in \mathcal{A}^{\N}$ with at least two distinct phase points in $X(\alpha)$. In view of Lemma \ref{lem:rule out non-parallel phase points}, all phase points in $X(\alpha)$ are necessarily parallel. 

	Under the above assumption, we now choose $\epsilon > 0$ and parallel phase points $p$ and $q$ with a parallel separation $L > 0$ satisfying the following properties.
	\begin{enumerate}
		\item The two edge codings of $p$ and $q$ are identical and non-periodic.
		\item The edge coding of any two parallel phase points\footnote{Recall that parallel phase points have the same edge coding by definition.} with parallel separation $\geq L + \frac{\epsilon}{4}$ is periodic.
		\item $\epsilon$ is smaller than both $L /2$ and $\delta / 2$.
	\end{enumerate}
	This choice is possible as the parallel separation is bounded by the maximal length of the edges of $Q$. 
	Assume without loss of generality that $p$ is on the left side of $q$. Let $\tau$ be the left translation map associated to $(p,q)$ as defined in Section \ref{sec:encoding parallel trajectories}. Then $q = \tau\inv(p)$. Let $\beta \in \mathcal{B}^{\N}$ be the $\mathcal{B}$-coding of $p$ and $q$. As defined in \eqref{eq:definition of omega limit}, we let $\Omega$ denote the $\omega$-limit set of $\beta$ under the left shift map $S$.
	Recall the following classical theorem due to Birkhoff (see for example \cite[Theorem 1.16]{furstenberg2014recurrence}) 

	\begin{theorem}\label{thm:uniform_recurrence}
		If $Z$ is compact and $T: Z \rightarrow Z$ is continuous, then the topological dynamical system $(Z, T)$ contains a uniformly recurrent point $z \in Z$, i.e. for all open neighbourhood $V$ of $z$, the set
		\[\{n \in \mathbb{N} \mid T^n(z) \in V\}\]
		can be arranged as an increasing sequence $s_1 < s_2 < s_3 < \cdots$ with bounded gaps $s_{n+1} - s_n$.
	\end{theorem}

	Applying Theorem \ref{thm:uniform_recurrence} with $Z = \Omega$ and $T = S$, we obtain a uniformly recurrent point $\omega \in \Omega$ under $S$. Since $p$ and $q$ are assumed to have non-periodic edge coding, by Corollary \ref{cor:rule out periodic omega}, the point $\omega$ is a non-periodic sequence.
	By Lemma~\ref{lemma:approximate_by_geodesic} applied on $\omega \in \Omega$, there exist a sequence $(q_m = (\tau \circ f)^{j_m}(p))_{m \geq 1}$ of phase points and a sequence $(z_n)_n$ of generalised phase points such that the sequences 
	\[z_0, \tau^{-1}(z_1),z_2,\tau^{-1}(z_3),z_4, \tau^{-1}(z_5),\cdots \tag{L} \label{eq:left Gamma}\] 
	and 
	\[\tau^{-1}(z_0),z_1,\tau^{-1}(z_2),z_3 ,\tau^{-1}(z_4), z_5,\cdots \tag{R} \label{eq:right Gamma}\]
	are two generalised trajectories which can be approximated arbitrarily well in the sense of Lemma \ref{lemma:approximate_by_geodesic} point (2) by parallel trajectories from $q_m$ and $\tau\inv(q_m)$ as $m \rightarrow \infty$. Let $\beta(m)$ denote the $\mathcal{B}$-coding of $q_m$ and $\tau\inv(q_m)$. Note that $\beta(m) = S^{j_m}(\beta)$.

	\paragraph{Locating vertices in \texorpdfstring{$Q^\infty$}{Q infty}}

	When $Q$ is unfolded according to $\omega$, the two sequences \eqref{eq:left Gamma} and \eqref{eq:right Gamma} define two generalised trajectories $\Gamma_L$ and $\Gamma_R$ respectively in $Q^\infty$. These two trajectories satisfy properties in Remark \ref{rem:properties of generalised trajectories}.
	Note that the parallel separation between $\Gamma_L$ and $\Gamma_R$ is the same as the parallel separation $L$ between $p$ and $q$.

	Let us also define $S_m$ and $S'_m$ to be the parallel billiard trajectories from $q_m$ and $\tau\inv(q_m)$. 
	For $m \geq 1$, let $N \geq 1$ be the largest integer depending on $m$ such that $\beta(m)_{n} = \omega_{n}$ for $n = 0, 1,2, \ldots, N-1$. The index $N$ is the smallest index at which $\beta(m)$ and $\omega$ differ, and the four trajectories $S_m$, $S'_m$, $\Gamma_L$ and $\Gamma_R$ intersect the same sequence of edges in the finite corridor $Q_0^{N-1}$. It follows from point (1) of Lemma \ref{lemma:approximate_by_geodesic} that $N \geq m$. We have the following lemma describing the location of some vertices in $Q_0^N$.

	\begin{lemma}\label{lem:splitting vertices}
		Suppose that $\Gamma_L$ and $\Gamma_R$ enter $Q_N$ via the edge $e$ and exit $Q_N$ via the edge $e'$. In other words, the edge $e$ is the boundary between $Q_{N-1}$ and $Q_{N}$ and $e'$ is the boundary between $Q_{N}$ and $Q_{N+1}$. 
		Let $\beta(m)_{N} = (i,j) \in I_{a,b} \times I_{a,b}$ and  $\omega_{N} = (i',j') \in I_{a,b'} \times I_{a,b'}$ where $a, b, b' \in \mathcal{A}$. Note that $e = e_{a}$ and $e' = e_{b'}$.
		Within the finite corridor $Q_0^{N-1}$, let $\CT_1$ be the interior of the region bounded by the straight lines $S_m$, $\Gamma_L$, $e$ and $\CT_2$ be the interior of the region bounded by the straight lines $S'_m$, $\Gamma_R$, $e$. The regions $\CT_1$ and $\CT_2$ are illustrated in Figure \ref{fig:splitting vertex case one}.
		The following statements hold.
		\begin{enumerate}
			\item The regions $\CT_1$ and $\CT_2$ do not contain any vertices. 
			\item Suppose $S_m$ bounds $\CT_1$ from the left. No vertex in $\ol{\CT_1} \cap \Gamma_L$ blocks\footnote{See Definition \ref{def:blocking vertices}.} $\Gamma_L$ from the left and no vertex in $\ol{\CT_2} \cap \Gamma_R$ blocks $\Gamma_R$ from the left. Analoguous statement holds if $S_m$ bounds $\CT_1$ from the right.
			\item If $b = b'$, then there exists a vertex $v$ in $Q_N$ lying strictly between $S_m$ and $\Gamma_L$ or between $S'_M$ and $\Gamma_R$. Moreover, the vertex $v$ is attached to a hole of $Q$.
			\item If $b \neq b'$, then the base points of $f^{N+1}(q_m)$ and $f^{N+1}(\tau\inv(q_m))$ either both lie on the right side of $\Gamma_R$ or both lie on the left side of $\Gamma_L$.
		\end{enumerate}
	\end{lemma}
	\begin{proof}
		It follows from the choice of $N$ and the definition of $\mathcal{B}$-codings that the first $N$ terms of the $V$-codings associated to $S_m$ and $\Gamma_L$ coincide.
		By Lemma \ref{rem:V partition detects holes}, there are no holes between the trajectories $S_m$ and $\Gamma_L$ before they meet the edge $e$. Also, there must not be any vertices of the outer boundary of $Q_n$ with $n = 0, 1, \ldots, N-1$ which lies trictly between $S_m$ and $\Gamma_L$, for, if not, the interior of a non-reflecting edge would transversally intersect either $S_m$ or $\Gamma_L$, contradicting Remark \ref{rem:properties of generalised trajectories}. Similar argument applies to $S'_m$ and $\Gamma_R$. This implies (1).

		The statement (2) follows immediately from the first statement of Lemma \ref{lem:blocking vertices make V coding different} and the assumption that the $\mathcal{B}$-coding of $(S_m, S'_m)$ coincide with the $\mathcal{B}$-coding of $(\Gamma_L, \Gamma_R)$ up to the $N$-th term.
		
		If $b = b'$, then the four trajectories $S_m, S'_m, \Gamma_L$ and $\Gamma_R$ exit $Q_N$ by the same edge $e'$ as shown in Figure \ref{fig:splitting vertex case one}. We must have $i \neq i'$ or $j \neq j'$. It follows from Remark \ref{rem:V partition detects holes} that there is a hole in $Q_N$ lying either between $S_m$ and $\Gamma_L$ or between $S'_m$ and $\Gamma_R$. The vertex $v$ can be taken to be any vertex on this hole. This proves (3).

		If $b \neq b'$, then $S_m$ and $S'_m$ exit $Q_N$ by a different edge from $e'$ as shown in Figure \ref{fig:splitting vertex case two}. Thus neither of these two phase points lies on the segment of edge $e'$ between $z_{N+1}$ and $\tau\inv(z_{N+1})$. Hence, $f^{N+1}(q_m)$ and $f^{N+1}(\tau\inv(q_m))$ either both lie on the right side of $\Gamma_R$ or both lie on the left side of $\Gamma_L$. This shows (4).
	\end{proof}

	\begin{figure}[]
		\centering
		\begin{subfigure}[b]{\textwidth}
			\centering
			\begin{tikzpicture}[line cap=round,line join=round,>=triangle 45,x=.9cm,y=.7cm]
				\clip(-4.303799249037241,-0.5818127122884735) rectangle (11.788104781207128,4.212475279510988);
				\fill[line width=2.pt,color=uuuuuu,fill=ffffff,fill opacity=1.0] (7.4,3.38) -- (8.56,3.94) -- (10.22,3.4) -- (10.04,-0.14) -- (8.88,-0.52) -- (8.,0.24) -- cycle;
				\fill[line width=2.pt,color=ffffff,fill=ffffff,fill opacity=1.0] (9.08,2.46) -- (8.38,2.38) -- (9.06,2.08) -- (9.6,2.26) -- cycle;
				\fill[line width=0.4pt,color=ududff,fill=ududff,pattern=north east lines,pattern color=ududff] (0.7534658680654296,1.9099735715602402) -- (7.552699543132439,2.5808723909402347) -- (7.685001297657424,1.88849320892615) -- cycle;
				\fill[line width=0.4pt,color=ududff,fill=ududff,pattern=north east lines,pattern color=ududff] (0.7455503645827102,1.007606174529642) -- (7.721789065040014,1.695970559623924) -- (7.857535042331567,0.9855666117981376) -- cycle;
				\draw [line width=0.8pt] (-2.38,2.62)-- (-2.4,0.34);
				\draw [line width=0.8pt] (-0.82,2.62)-- (-0.38,0.28);
				\draw [line width=0.8pt] (1.38,2.74)-- (0.82,0.16);
				\draw [line width=0.8pt] (5.56,2.92)-- (5.86,0.14);
				\draw [line width=0.8pt] (7.4,3.38)-- (8.,0.24);
				\draw [line width=0.8pt,color=uuuuuu] (7.4,3.38)-- (8.56,3.94);
				\draw [line width=0.8pt,color=uuuuuu] (8.56,3.94)-- (10.22,3.4);
				\draw [line width=0.8pt,color=uuuuuu] (10.22,3.4)-- (10.04,-0.14);
				\draw [line width=0.8pt,color=uuuuuu] (10.04,-0.14)-- (8.88,-0.52);
				\draw [line width=0.8pt,color=uuuuuu] (8.88,-0.52)-- (8.,0.24);
				\draw [line width=0.8pt,color=uuuuuu] (8.,0.24)-- (7.4,3.38);
				\draw [line width=0.8pt] (8.38,2.38)-- (9.08,2.46);
				\draw [line width=0.8pt] (9.08,2.46)-- (9.6,2.26);
				\draw [line width=0.8pt] (9.6,2.26)-- (9.06,2.08);
				\draw [line width=0.8pt] (9.06,2.08)-- (8.38,2.38);
				\draw [line width=0.8pt,dotted] (-2.3861429560667844,1.9197030083865507)-- (11.779986994161224,1.8758031233548078);
				\draw [line width=0.8pt,dotted] (-2.3889482188197277,1.5999030545510504)-- (11.883140873091772,3.0081687767546876);
				\draw [line width=0.8pt,dotted] (-2.3940584595495094,1.0173356113559526)-- (11.859978813549034,0.973163307436269);
				\draw [line width=0.8pt,dotted] (-2.396863722302452,0.6975356575204523)-- (11.974251051758941,2.11557251223955);
				\draw [->,line width=0.8pt, -{Latex[width=1.5mm]}] (-2.386142956066785,1.9197030083865507) -- (-1.6560803693239552,1.9174405934760146);
				\draw [->,line width=0.8pt, -{Latex[width=1.5mm]}] (-2.3889482188197273,1.5999030545510504) -- (-1.6250597512964962,1.6752780005363095);
				\draw [->,line width=0.8pt, -{Latex[width=1.5mm]}] (-2.3940584595495094,1.0173356113559526) -- (-1.7993187454004675,1.0154925529365617);
				\draw [->,line width=0.8pt, -{Latex[width=1.5mm]}] (-2.396863722302452,0.6975356575204523) -- (-1.819898483229745,0.7544663818073126);
				\draw (-1.8818539166004122,2.8223099882085414) node[anchor=north west] {{\scriptsize $Q_0$}};
				\draw (-0.06394545874334676,2.911423147907416) node[anchor=north west] {{\scriptsize $Q_1$}};
				\draw (3.1441282904161807,2.964891043726741) node[anchor=north west] {{\scriptsize $\cdots$}};
				\draw (6.01357203271998,3.339166314462015) node[anchor=north west] {{\scriptsize $Q_{N-1}$}};
				\draw (8.544385768168052,3.5708605296790896) node[anchor=north west] {{\scriptsize $Q_N$}};
				\draw (6.227443615997282,1.62819364824362) node[anchor=north west] {{\scriptsize $\mathcal{T}_2$}};
				\draw (6.191798352117732,2.5549705091119175) node[anchor=north west] {{\scriptsize $\mathcal{T}_1$}};
				\draw (11.021731607796797,3.6421510574381895) node[anchor=north west] {{\scriptsize $S_m$}};
				\draw (11.021731607796797,2.6975515646301167) node[anchor=north west] {{\scriptsize $S'_m$}};
				\draw (11.003908975857021,1.895533127340244) node[anchor=north west] {{\scriptsize $\Gamma_L$}};
				\draw (10.932618448097921,0.9687562664719467) node[anchor=north west] {{\scriptsize $\Gamma_R$}};
				\draw (-3.254196575963099,1.9202915509186687) node[anchor=north west] {{\scriptsize $q_m$}};
				\draw (-4.1245207980359037,1.0222241622912716) node[anchor=north west] {{\scriptsize $\tau^{-1}(q_m)$}};
				\draw (-3.236373944023324,2.3658573494130427) node[anchor=north west] {{\scriptsize $z_0$}};
				\draw (-4.053230270276803,1.5034352246651952) node[anchor=north west] {{\scriptsize $\tau^{-1}(z_0)$}};
				\draw (7.427443615997282,0.82819364824362) node[anchor=north west] {{\scriptsize $e$}};
				\draw (9.427443615997282,0.62819364824362) node[anchor=north west] {{\scriptsize $e'$}};
				\begin{scriptsize}
				\draw [fill=uuuuuu] (-2.3940584595495094,1.0173356113559526) circle (1.0pt);
				\draw [fill=uuuuuu] (-2.396863722302452,0.6975356575204523) circle (1.0pt);
				\draw [fill=uuuuuu] (-2.386142956066785,1.9197030083865507) circle (1.0pt);
				\draw [fill=uuuuuu] (-2.3889482188197273,1.5999030545510504) circle (1.0pt);
				\end{scriptsize}
				\end{tikzpicture}
			\caption{If $b = b'$, then the four trajectories $S_m, S'_m, \Gamma_L$ and $\Gamma_R$ exit $Q_N$ by the same edge $e'$. There is a hole in $Q_N$ lying either between $S_m$ and $\Gamma_L$ or between $S'_m$ and $\Gamma_R$.}
			\label{fig:splitting vertex case one}
		\end{subfigure}

		\begin{subfigure}[b]{\textwidth}
			\centering
			\begin{tikzpicture}[line cap=round,line join=round,>=triangle 45,x=0.9cm,y=0.7cm]
				\clip(-4.898667335160345,-0.728277996571486) rectangle (12.13094502584317,4.314902306374234);
				\fill[line width=2.pt,color=zzttqq,fill=zzttqq,fill opacity=0] (7.4,3.38) -- (8.56,3.94) -- (10.22,3.4) -- (10.04,-0.14) -- (8.88,-0.52) -- (8.,0.24) -- cycle;
				\fill[line width=0.4pt,color=ududff,fill=ududff,pattern=north east lines,pattern color=ududff] (0.7534658680654296,1.9099735715602402) -- (7.552699543132439,2.5808723909402347) -- (7.685001297657424,1.88849320892615) -- cycle;
				\fill[line width=0.4pt,color=ududff,fill=ududff,pattern=north east lines,pattern color=ududff] (0.7455503645827102,1.007606174529642) -- (7.721789065040014,1.695970559623924) -- (7.857535042331567,0.9855666117981376) -- cycle;
				\fill[line width=0.4pt,color=zzttqq,fill=zzttqq,fill opacity=0] (10.22,3.4) -- (10.885579083238898,1.9340960721973857) -- (10.04,-0.14) -- cycle;
				\draw [line width=0.8pt] (-2.38,2.62)-- (-2.4,0.34);
				\draw [line width=0.8pt] (-0.82,2.62)-- (-0.38,0.28);
				\draw [line width=0.8pt] (1.38,2.74)-- (0.82,0.16);
				\draw [line width=0.8pt] (5.56,2.92)-- (5.86,0.14);
				\draw [line width=0.8pt] (7.4,3.38)-- (8.,0.24);
				\draw [line width=0.8pt] (7.4,3.38)-- (8.56,3.94);
				\draw [line width=0.8pt] (8.56,3.94)-- (10.22,3.4);
				\draw [line width=0.8pt] (10.04,-0.14)-- (8.88,-0.52);
				\draw [line width=0.8pt] (8.88,-0.52)-- (8.,0.24);
				\draw [line width=0.8pt] (8.,0.24)-- (7.4,3.38);
				\draw [line width=0.8pt,dotted] (-2.3861429560667844,1.9197030083865507)-- (11.779986994161224,1.8758031233548078);
				\draw [line width=0.8pt,dotted] (-2.3889482188197277,1.5999030545510504)-- (11.883140873091772,3.0081687767546876);
				\draw [line width=0.8pt,dotted] (-2.3940584595495094,1.0173356113559526)-- (11.859978813549034,0.973163307436269);
				\draw [line width=0.8pt,dotted] (-2.396863722302452,0.6975356575204523)-- (11.974251051758941,2.11557251223955);
				\draw [->,line width=0.8pt, -{Latex[width=1.5mm]}] (-2.386142956066785,1.9197030083865507) -- (-1.6560803693239552,1.9174405934760146);
				\draw [->,line width=0.8pt, -{Latex[width=1.5mm]}] (-2.3889482188197273,1.5999030545510504) -- (-1.6250597512964962,1.6752780005363095);
				\draw [->,line width=0.8pt, -{Latex[width=1.5mm]}] (-2.3940584595495094,1.0173356113559526) -- (-1.7993187454004675,1.0154925529365617);
				\draw [->,line width=0.8pt, -{Latex[width=1.5mm]}] (-2.396863722302452,0.6975356575204523) -- (-1.819898483229745,0.7544663818073126);
				\draw (-1.8798622242420997,2.8232574280381764) node[anchor=north west] {{\scriptsize $Q_0$}};
				\draw (-0.050821480568103726,2.9120458136534175) node[anchor=north west] {{\scriptsize $Q_1$}};
				\draw (3.145560401580627,2.9653188450225625) node[anchor=north west] {{\scriptsize $\cdots$}};
				\draw (6.0045464183914365,3.3382300646065772) node[anchor=north west] {{\scriptsize $Q_{N-1}$}};
				\draw (8.543894246987373,3.5690798672062054) node[anchor=north west] {{\scriptsize $Q_N$}};
				\draw (6.235396220991067,1.6334930607939393) node[anchor=north west] {{\scriptsize $\mathcal{T}_2$}};
				\draw (6.19988086674497,2.5568922711924515) node[anchor=north west] {{\scriptsize $\mathcal{T}_1$}};
				\draw (11.029969044214162,3.6401105756983987) node[anchor=north west] {{\scriptsize $S_m$}};
				\draw (11.029969044214162,2.698953688176838) node[anchor=north west] {{\scriptsize $S'_m$}};
				\draw (11.012211367091115,1.8998582176396641) node[anchor=north west] {{\scriptsize $\Gamma_L$}};
				\draw (10.94118065859892,0.9764590072411519) node[anchor=north west] {{\scriptsize $\Gamma_R$}};
				\draw (-3.2472033627168346,1.9353735718857608) node[anchor=north west] {{\scriptsize $q_m$}};
				\draw (-4.223875604484502,1.1007627471024901) node[anchor=north west] {{\scriptsize $\tau^{-1}(q_m)$}};
				\draw (-3.2472033627168346,2.3438001457158717) node[anchor=north west] {{\scriptsize $z_0$}};
				\draw (-4.188360250238405,1.4559162895634563) node[anchor=north west] {{\scriptsize $\tau^{-1}(z_0)$}};
				\draw [line width=0.8pt] (10.22,3.4)-- (10.885579083238898,1.9340960721973857);
				\draw [line width=0.8pt] (10.885579083238898,1.9340960721973857)-- (10.04,-0.14);
				\draw (7.567222005219705,0.7811245588876204) node[anchor=north west] {{\scriptsize $e$}};
				\draw (9.680385582862478,0.639063141903234) node[anchor=north west] {{\scriptsize $e'$}};
				\begin{scriptsize}
				\draw [fill=zzttqq] (-2.3940584595495094,1.0173356113559526) circle (1.0pt);
				\draw [fill=zzttqq] (-2.396863722302452,0.6975356575204523) circle (1.0pt);
				\draw [fill=zzttqq] (-2.386142956066785,1.9197030083865507) circle (1.0pt);
				\draw [fill=zzttqq] (-2.3889482188197273,1.5999030545510504) circle (1.0pt);
				\end{scriptsize}
				\end{tikzpicture}
			\caption{If $b \neq b'$, then $S_m$ and $S'_m$ exit $Q_N$ by a different edge from $e'$.}
			\label{fig:splitting vertex case two}
		\end{subfigure}
		\caption{}
	\end{figure}

	The following lemma is another observation regarding the vertices on $\Gamma_L$ and $\Gamma_R$. Recall that we have assumed that the holes of $Q$ have non-zero minimal diameters.

	\begin{lemma}\label{lem:blocking vertices on parallel trajectories}
		Let $v \in Q^\infty$ be a vertex lying on $\Gamma_L$. Then the following statements hold.
		\begin{enumerate}
			\item If $v$ is an end vertex of a reflecting edge, then $v$ blocks $\Gamma_L$ from the left.
			\item If $v$ is not an end vertex of any reflecting edge, then there exists a vertex $u \in Q^\infty$ blocking $\Gamma_L$ from left or from right and $v$ is either euqal to $u$ or connected to $u$ via a sequence of edges overlapping $\Gamma_L$.
		\end{enumerate}
		The analoguous statements for $v$ lying on $\Gamma_R$ also hold.
	\end{lemma}
	\begin{proof}
		Suppose $v$ lies on $\Gamma_L$ and $v$ is an end vertex of some reflecting edge $e$. Since $\Gamma_R$ intersects $e$ on the right side of $\Gamma_L$ and $\Gamma_R$ has a nonzero parallel separation from $\Gamma_L$, the edge $e$ cannot overlap $\Gamma_L$ or lie on the left side of $\Gamma_L$. Consequently, the interior of $e$ lies on the right side of $\Gamma_L$ and thus $v$ blocks $\Gamma_L$ from the left by Definition \ref{def:blocking vertices}.

		Suppose $v$ is not an end vertex of any reflecting edge. If $v$ does not block $\Gamma_L$ either from left or from right, then all edges containing $v$ must overlap $\Gamma_L$. Since $Q$ has nonzero diameter and all the holes of $Q$ have nonzero diameters, the connected component of $\partial Q$ containing $v$ is not contained in $\Gamma_L$. Thus, there exists a vertex $u$ connected to $v$ via a sequence of edges lying in $\Gamma_L$ such that $u$ is the end vertex of an edge $e'$ not overlapping $\Gamma_L$. Since $e'$ lies on either the left side or the right side of $\Gamma_L$, we see that $u$ blocks $\Gamma_L$ from left or from right.
	\end{proof}

	\paragraph{Existence of uniformly recurrent vertices}

	Recall that Corollary \ref{cor:rule out periodic omega} and the non-periodicity of the edge coding of $p$ and $q$ imply that the sequence $\omega$ is non-periodic. 
	The main objective of this subsection is to establish a useful consequence of Theorem \ref{thm:uniform_recurrence}: the existence of certain \textit{uniformly recurrent vertices} near the generalised trajectories (Corollary \ref{cor:existence of uniform recurrent vertices}).

	For convenience, we define the following subsets of $Q^\infty$. Let $\eta > 0$.
	\begin{itemize}
		\item The set $\CL_L(\eta)$ is the set of vertices in $Q^\infty$ which either lies in the open left $\eta$-strip of $\Gamma_L$ or blocks $\Gamma_L$ from the left.
		\item The set $\CR_R(\eta)$ is the set of vertices in $Q^\infty$ which either lies in the open right $\eta$-strip of $\Gamma_R$ or blocks $\Gamma_R$ from the right.
		\item The set $\CR_L(\eta)$ is the set of vertices in $Q^\infty$ which either lies in the open right $\eta$-strip of $\Gamma_L$ or blocks $\Gamma_L$ from the right.
		\item The set $\CL_R(\eta)$ is the set of vertices in $Q^\infty$ which either lies in the open left $\eta$-strip of $\Gamma_R$ or blocks $\Gamma_R$ from the left.
	\end{itemize}
	These definitions are motivated by our results in Lemma \ref{lem:blocking vertices make V coding different}: for example, if $\CL_L(\eta)$ is empty, then by Lemma \ref{lem:blocking vertices make V coding different} there exists a physical trajectory in the open left $\eta$-strip of $\Gamma_L$ having the same edge coding as $\Gamma_L$.

	In Lemma \ref{lem:infinitely many blocking vertices on either side} and Proposition \ref{prop:infinite neighbouring vertices} below, we exploit the fact that $L$ is $\epsilon/4$ away from the supremum of parallel separations between non-periodic parallel trajectories. The basic idea is as follows. The choice of $\epsilon$ ensures that no pairs of trajectories with parallel separations $\geq L + \epsilon /4$ have the same edge coding as $\Gamma_L$. Thus, there must be vertices in the sets $\CL_L(\eta)$, $\CR_R(\eta)$, $\CL_R(\eta)$ and $\CR_L(\eta)$ for some small $\eta > 0$ to prevent the existence of such parallel trajectories.

	Let us put $\epsilon' = \epsilon / 2$ and let $\CL_L = \CL_L(\epsilon')$, $\CR_R = \CR_R(\epsilon')$, $\CL_R = \CL_R(\epsilon')$ and $\CR_L = \CR_L(\epsilon')$.

	\begin{lemma}\label{lem:infinitely many blocking vertices on either side}
		$\CL_L \cup \CR_R$ contains infinitely many vertices.
	\end{lemma}
	\begin{proof}
		Suppose not, then $\CL_L \cup \CR_R$ is contained in $Q_0^M$ for a sufficiently large $M$ as illustrated in Figure \ref{fig:extending Gamma outwards}. By only considering the parts of $\Gamma_L$ and $\Gamma_R$ after the $M$-th reflection, we may assume $\CL_L \cup \CR_R$ is empty. Then Lemma \ref{lem:blocking vertices make V coding different} implies that there are physical trajectories $l_1$ and $l_2$ parallel to $\Gamma_L$ and $\Gamma_R$ with the same edge coding and the parallel separtion between $l_1$ and $l_2$ can be made arbitrarily close to $L + 2 \epsilon'$. By the choice of $\epsilon$ the edge coding of $l_1$ and $l_2$ would be periodic, which, together with Theorem \ref{thm:periodic coding implies periodic trajectory} extended to generalised trajectories, contradicts the nonperiodicity of $\omega$.
	\end{proof}

	\begin{figure}[]
		\centering
		\begin{tikzpicture}[line cap=round,line join=round,>=triangle 45,x=1.0cm,y=1.0cm]
			\clip(-4.594991454913543,-0.5634597914255917) rectangle (9.210366382341675,2.813410029887159);
			\fill[line width=2.pt,color=zzttqq,fill=zzttqq,fill opacity=0.10000000149011612] (-2.404459585003525,1.5445559593445757) -- (-2.98,1.82) -- (-3.229129578852706,-0.21006104884517038) -- (-1.3867817202534716,-0.4206150898279399) -- (-1.26,2.1) -- cycle;
			\fill[line width=2.pt,color=zzttqq,fill=zzttqq,fill opacity=0.10000000149011612] (3.252771311896054,1.6667628093996147) -- (3.8158801651101726,1.966803820361086) -- (4.1524427255278855,-0.050605174370213724) -- (2.320905981949846,-0.3405220472543383) -- (2.085393249764743,2.1722667995526215) -- cycle;
			\fill[line width=2.pt,color=zzttqq,fill=zzttqq,fill opacity=0.10000000149011612] (4.4458932566727585,1.8658102894462143) -- (3.8158801651101717,1.9668038203610858) -- (4.152442725527885,-0.05060517437021386) -- (5.978903404165862,0.26973843122954655) -- (5.385951675078212,2.722896045943837) -- cycle;
			\fill[line width=2.pt,color=zzttqq,fill=zzttqq,pattern=north east lines,pattern color=zzttqq] (-3.062869610979888,1.144727436388889) -- (-3.025361121653415,1.4503696790055975) -- (5.885096460606248,1.7939523130883641) -- (5.90576869273214,1.4905534889880334) -- cycle;
			\fill[line width=2.pt,color=zzttqq,fill=zzttqq,pattern=north east lines,pattern color=zzttqq] (-3.1475730911548214,0.45451137855809504) -- (-3.1781176125571533,0.2056158301449067) -- (6.010215064559455,0.5599131841122088) -- (6.002147178250808,0.8073198580622013) -- cycle;
			\draw [line width=0.8pt,color=zzttqq] (-2.404459585003525,1.5445559593445757)-- (-2.98,1.82);
			\draw [line width=0.8pt,color=zzttqq] (-2.98,1.82)-- (-3.229129578852706,-0.21006104884517038);
			\draw [line width=0.8pt,color=zzttqq] (-3.229129578852706,-0.21006104884517038)-- (-1.3867817202534716,-0.4206150898279399);
			\draw [line width=0.8pt,color=zzttqq] (-1.3867817202534716,-0.4206150898279399)-- (-1.26,2.1);
			\draw [line width=0.8pt,color=zzttqq] (-1.26,2.1)-- (-2.404459585003525,1.5445559593445757);
			\draw [line width=0.8pt,color=zzttqq] (3.252771311896054,1.6667628093996147)-- (3.8158801651101726,1.966803820361086);
			\draw [line width=0.8pt,color=zzttqq] (3.8158801651101726,1.966803820361086)-- (4.1524427255278855,-0.050605174370213724);
			\draw [line width=0.8pt,color=zzttqq] (4.1524427255278855,-0.050605174370213724)-- (2.320905981949846,-0.3405220472543383);
			\draw [line width=0.8pt,color=zzttqq] (2.320905981949846,-0.3405220472543383)-- (2.085393249764743,2.1722667995526215);
			\draw [line width=0.8pt,color=zzttqq] (2.085393249764743,2.1722667995526215)-- (3.252771311896054,1.6667628093996147);
			\draw [line width=0.8pt,color=zzttqq] (4.4458932566727585,1.8658102894462143)-- (3.8158801651101717,1.9668038203610858);
			\draw [line width=0.8pt,color=zzttqq] (3.8158801651101717,1.9668038203610858)-- (4.152442725527885,-0.05060517437021386);
			\draw [line width=0.8pt,color=zzttqq] (4.152442725527885,-0.05060517437021386)-- (5.978903404165862,0.26973843122954655);
			\draw [line width=0.8pt,color=zzttqq] (5.978903404165862,0.26973843122954655)-- (5.385951675078212,2.722896045943837);
			\draw [line width=0.8pt,color=zzttqq] (5.385951675078212,2.722896045943837)-- (4.4458932566727585,1.8658102894462143);
			\draw [line width=0.8pt,dotted] (3.8632091153372787,1.6231066385329863)-- (8.506228268658116,1.8021390547773333);
			\draw [line width=0.8pt,dotted] (4.056133124900351,0.5466897402626639)-- (8.567004815814132,0.7206266187299649);
			\draw [line width=0.8pt] (-3.147573091154821,0.45451137855809504)-- (8.744871002210084,0.9130778696230968);
			\draw [line width=0.8pt] (-3.062869610979888,1.144727436388889)-- (8.530615283525652,1.591766215427981);
			\draw [line width=0.8pt,dotted] (6.657770884041835,0.6270075461624068)-- (6.612959827447407,1.7891356032652521);
			\draw [->,line width=0.8pt, -{Latex[width=1.5mm]}] (6.596658457073745,2.2118946569496396) -- (6.612959827447407,1.7891356032652521);
			\draw [->,line width=0.8pt, -{Latex[width=1.5mm]}] (6.675382149768395,0.17027772148617082) -- (6.657770884041835,0.6270075461624068);
			\draw (6.739254777017316,2.36204624188991) node[anchor=north west] {$l_1$};
			\draw (6.822840663683473,0.7237628632332291) node[anchor=north west] {$l_2$};
			\draw (-3.809284120251669,1.3432076165609584) node[anchor=north west] {$\Gamma_L$};
			\draw (-3.959738716250751,0.7903285085667663) node[anchor=north west] {$\Gamma_R$};
			\draw (6.672386067684391,1.559621729894801) node[anchor=north west] {$\approx L + 2 \epsilon'$};
			\draw (-2.1375663869285333,0.33926778456890605) node[anchor=north west] {$Q_0$};
			\draw (2.493091734376552,0.38941931656860035) node[anchor=north west] {$Q_M$};
			\draw (0.1526869077241624,1.2754297152298666) node[anchor=north west] {$\cdots$};
			\begin{scriptsize}
			\draw [fill=ududff] (-2.0609290915037843,1.3142589027853508) circle (1.0pt);
			\draw [fill=ududff] (-0.509473216158598,1.2462710679523188) circle (1.0pt);
			\draw [fill=ududff] (-0.2813730050939308,0.3865087339393433) circle (1.0pt);
			\draw [fill=ududff] (-2.2398017533421677,0.3416388040391424) circle (1.0pt);
			\draw [fill=ududff] (1.4803584465441246,1.4342371832894487) circle (1.0pt);
			\draw [fill=ududff] (2.9582221501667334,0.5105723685253205) circle (1.0pt);
			\draw [fill=ududff] (2.4596573797100847,1.476035843228645) circle (1.0pt);
			\end{scriptsize}
			\end{tikzpicture}
		\caption{If $\CL_L \cup \CR_R$ contains only finitely many vertices, it will be contained in $Q_0^M$ for sufficiently large $M$. In $Q_{M+1}^\infty$, there are physical trajectories $l_1$ and $l_2$ parallel to $\Gamma_L$ and $\Gamma_R$ with parallel separtion arbitrarily close to $L + 2 \epsilon'$.}
		\label{fig:extending Gamma outwards}
	\end{figure}

	\begin{proposition}\label{prop:infinite neighbouring vertices}
		At least one of the following statements holds:
		\begin{enumerate}
			\item Both $\CL_L$ and $\CR_R$ contain infinitely many vertices.
			\item Both $\CL_L$ and $\CR_L$ contain infinitely many vertices.
			\item Both $\CL_R$ and $\CR_R$ contain infinitely many vertices.
		\end{enumerate}
		The sets involved in these three cases are illustrated in Figure \ref{fig:infinite neighbouring vertices}.
	\end{proposition}

	\begin{figure}[]
		\centering
		\begin{subfigure}[t]{0.3\textwidth}
			\begin{tikzpicture}[line cap=round,line join=round,>=triangle 45,x=0.9cm,y=0.9cm]
				\clip(-0.32074249182768655,1.5090809764086726) rectangle (5.559270454516181,4.734273262674883);
				\fill[line width=2.pt,color=zzttqq,fill=zzttqq,fill opacity=0.10000000149011612] (-0.18,2.44) -- (0.54,3.86) -- (1.64,3.94) -- (2.2,2.32) -- (1.3,1.66) -- cycle;
				\fill[line width=2.pt,color=zzttqq,fill=zzttqq,fill opacity=0.10000000149011612] (3.9978134785568424,3.8841824370319955) -- (2.5545786249149085,4.556397549353302) -- (1.64,3.94) -- (2.2,2.32) -- (3.315460857726346,2.356702518720218) -- cycle;
				\fill[line width=2.pt,color=zzttqq,fill=zzttqq,pattern=north east lines,pattern color=zzttqq] (0.4606707916641302,3.7035451724487016) -- (0.3093545841881016,3.405115985482089) -- (4.525838409877995,3.476455176591337) -- (4.552043141847935,3.772767587630408) -- cycle;
				\fill[line width=2.pt,color=zzttqq,fill=zzttqq,pattern=north west lines,pattern color=zzttqq] (0.010918344705208584,2.816533402057495) -- (-0.13510309522006095,2.5285466733159914) -- (4.485628216840791,2.606725375704743) -- (4.477542824205154,2.8921047526256793) -- cycle;
				\draw [line width=0.8pt,color=zzttqq] (-0.18,2.44)-- (0.54,3.86);
				\draw [line width=0.8pt,color=zzttqq] (0.54,3.86)-- (1.64,3.94);
				\draw [line width=0.8pt,color=zzttqq] (1.64,3.94)-- (2.2,2.32);
				\draw [line width=0.8pt,color=zzttqq] (2.2,2.32)-- (1.3,1.66);
				\draw [line width=0.8pt,color=zzttqq] (1.3,1.66)-- (-0.18,2.44);
				\draw [line width=0.8pt,color=zzttqq] (3.9978134785568424,3.8841824370319955)-- (2.5545786249149085,4.556397549353302);
				\draw [line width=0.8pt,color=zzttqq] (2.5545786249149085,4.556397549353302)-- (1.64,3.94);
				\draw [line width=0.8pt,color=zzttqq] (1.64,3.94)-- (2.2,2.32);
				\draw [line width=0.8pt,color=zzttqq] (2.2,2.32)-- (3.315460857726346,2.356702518720218);
				\draw [line width=0.8pt,color=zzttqq] (3.315460857726346,2.356702518720218)-- (3.9978134785568424,3.8841824370319955);
				\draw [line width=0.8pt,color=zzttqq] (0.3093545841881016,3.405115985482089)-- (4.525838409877995,3.476455176591337);
				\draw [line width=0.8pt,color=zzttqq] (0.010918344705208584,2.816533402057495)-- (-0.13510309522006095,2.5285466733159914);
				\draw [->,line width=0.8pt, -{Latex[width=1.5mm]}] (0.3093545841881016,3.405115985482089) -- (1.1153535162686017,3.4187527772324127);
				\draw [->,line width=0.8pt, -{Latex[width=1.5mm]}] (0.010918344705208584,2.816533402057495) -- (0.9077804573809618,2.831707518762271);
				\draw [line width=0.8pt,dotted] (0.3093545841881016,3.405115985482089)-- (5.2043415082175235,3.4879348511736867);
				\draw [line width=0.8pt,dotted] (0.010918344705208584,2.816533402057495)-- (5.2044857155926145,2.904403985825208);
				\draw (4.667416639021344,3.8942714131971883) node[anchor=north west] {{\scriptsize $\Gamma_L$}};
				\draw (4.636305459411059,3.3342701802120582) node[anchor=north west] {{\scriptsize $\Gamma_R$}};
				\end{tikzpicture}
			\caption{$\CL_L$ and $\CR_R$.}
			\label{fig:infinite neighbouring vertices 1}
		\end{subfigure}
		\begin{subfigure}[t]{0.3\textwidth}
			\begin{tikzpicture}[line cap=round,line join=round,>=triangle 45,x=0.9cm,y=0.9cm]
				\clip(-0.32074249182768655,1.5090809764086726) rectangle (5.559270454516181,4.734273262674883);
				\fill[line width=2.pt,color=zzttqq,fill=zzttqq,fill opacity=0.10000000149011612] (-0.18,2.44) -- (0.54,3.86) -- (1.64,3.94) -- (2.2,2.32) -- (1.3,1.66) -- cycle;
				\fill[line width=2.pt,color=zzttqq,fill=zzttqq,fill opacity=0.10000000149011612] (3.9978134785568424,3.8841824370319955) -- (2.5545786249149085,4.556397549353302) -- (1.64,3.94) -- (2.2,2.32) -- (3.315460857726346,2.356702518720218) -- cycle;
				\fill[line width=2.pt,color=zzttqq,fill=zzttqq,pattern=north east lines,pattern color=zzttqq] (0.4606707916641302,3.7035451724487016) -- (0.3093545841881016,3.405115985482089) -- (4.525838409877995,3.476455176591337) -- (4.552043141847935,3.772767587630408) -- cycle;
				\fill[line width=2.pt,color=zzttqq,fill=zzttqq,pattern=north west lines,pattern color=zzttqq] (0.3093545841881013,3.4051159854820887) -- (0.11208016314263462,3.0160469884201966) -- (4.511774365579895,3.090485937673978) -- (4.480286992965737,3.475684486753372) -- cycle;
				\draw [line width=0.8pt,color=zzttqq] (-0.18,2.44)-- (0.54,3.86);
				\draw [line width=0.8pt,color=zzttqq] (0.54,3.86)-- (1.64,3.94);
				\draw [line width=0.8pt,color=zzttqq] (1.64,3.94)-- (2.2,2.32);
				\draw [line width=0.8pt,color=zzttqq] (2.2,2.32)-- (1.3,1.66);
				\draw [line width=0.8pt,color=zzttqq] (1.3,1.66)-- (-0.18,2.44);
				\draw [line width=0.8pt,color=zzttqq] (3.9978134785568424,3.8841824370319955)-- (2.5545786249149085,4.556397549353302);
				\draw [line width=0.8pt,color=zzttqq] (2.5545786249149085,4.556397549353302)-- (1.64,3.94);
				\draw [line width=0.8pt,color=zzttqq] (2.2,2.32)-- (3.315460857726346,2.356702518720218);
				\draw [line width=0.8pt,color=zzttqq] (3.315460857726346,2.356702518720218)-- (3.9978134785568424,3.8841824370319955);
				\draw [->,line width=0.8pt, -{Latex[width=1.5mm]}] (0.3093545841881016,3.405115985482089) -- (1.1153535162686017,3.4187527772324127);
				\draw [->,line width=0.8pt, -{Latex[width=1.5mm]}] (0.010918344705208584,2.816533402057495) -- (0.9077804573809618,2.831707518762271);
				\draw [line width=0.8pt,dotted] (0.3093545841881016,3.405115985482089)-- (5.2043415082175235,3.4879348511736867);
				\draw [line width=0.8pt,dotted] (0.010918344705208584,2.816533402057495)-- (5.2044857155926145,2.904403985825208);
				\draw (4.729638998241914,3.88390101999376) node[anchor=north west] {{\scriptsize $\Gamma_L$}};
				\draw (4.646675852614487,3.32389978700863) node[anchor=north west] {{\scriptsize $\Gamma_R$}};
				\end{tikzpicture}
			\caption{$\CL_L$ and $\CR_L$.}
			\label{fig:infinite neighbouring vertices 2}
		\end{subfigure}
		\begin{subfigure}[t]{0.3\textwidth}
			\begin{tikzpicture}[line cap=round,line join=round,>=triangle 45,x=0.9cm,y=0.9cm]
				\clip(-0.31979972880919333,1.499841898827437) rectangle (5.560213217534675,4.725034185093646);
				\fill[line width=2.pt,color=zzttqq,fill=zzttqq,fill opacity=0.10000000149011612] (-0.18,2.44) -- (0.54,3.86) -- (1.64,3.94) -- (2.2,2.32) -- (1.3,1.66) -- cycle;
				\fill[line width=2.pt,color=zzttqq,fill=zzttqq,fill opacity=0.10000000149011612] (3.9978134785568424,3.8841824370319955) -- (2.5545786249149085,4.556397549353302) -- (1.64,3.94) -- (2.2,2.32) -- (3.315460857726346,2.356702518720218) -- cycle;
				\fill[line width=2.pt,color=zzttqq,fill=zzttqq,pattern=north west lines,pattern color=zzttqq] (0.010918344705208584,2.816533402057495) -- (-0.13510309522006095,2.5285466733159914) -- (4.485628216840791,2.606725375704743) -- (4.477542824205154,2.8921047526256793) -- cycle;
				\fill[line width=0.4pt,color=zzttqq,fill=zzttqq,pattern=north east lines,pattern color=zzttqq] (0.005046086458227239,2.804952003848171) -- (0.17997442344975037,3.14994955735923) -- (4.398729178750701,3.2213271705965387) -- (4.388690635835577,2.8906014518877634) -- cycle;
				\draw [line width=0.8pt,color=zzttqq] (-0.18,2.44)-- (0.54,3.86);
				\draw [line width=0.8pt,color=zzttqq] (0.54,3.86)-- (1.64,3.94);
				\draw [line width=0.8pt,color=zzttqq] (1.64,3.94)-- (2.2,2.32);
				\draw [line width=0.8pt,color=zzttqq] (2.2,2.32)-- (1.3,1.66);
				\draw [line width=0.8pt,color=zzttqq] (1.3,1.66)-- (-0.18,2.44);
				\draw [line width=0.8pt,color=zzttqq] (3.9978134785568424,3.8841824370319955)-- (2.5545786249149085,4.556397549353302);
				\draw [line width=0.8pt,color=zzttqq] (2.5545786249149085,4.556397549353302)-- (1.64,3.94);
				\draw [line width=0.8pt,color=zzttqq] (2.2,2.32)-- (3.315460857726346,2.356702518720218);
				\draw [line width=0.8pt,color=zzttqq] (3.315460857726346,2.356702518720218)-- (3.9978134785568424,3.8841824370319955);
				\draw [->,line width=0.8pt, -{Latex[width=1.5mm]}] (0.3093545841881016,3.405115985482089) -- (1.1153535162686017,3.4187527772324127);
				\draw [->,line width=0.8pt, -{Latex[width=1.5mm]}] (0.010918344705208584,2.816533402057495) -- (0.9077804573809618,2.831707518762271);
				\draw [line width=0.8pt,dotted] (0.3093545841881016,3.405115985482089)-- (5.2043415082175235,3.4879348511736867);
				\draw [line width=0.8pt,dotted] (0.010918344705208584,2.816533402057495)-- (5.2044857155926145,2.904403985825208);
				\draw (4.7305817612604075,3.8850323356159513) node[anchor=north west] {{\scriptsize $\Gamma_L$}};
				\draw (4.647618615632981,3.3250311026308217) node[anchor=north west] {{\scriptsize $\Gamma_R$}};
				\end{tikzpicture}
			\caption{$\CL_R$ and $\CR_R$.}
			\label{fig:infinite neighbouring vertices 3}
		\end{subfigure}
		\caption{}
		\label{fig:infinite neighbouring vertices}
	\end{figure}

	\begin{proof}
		Suppose (1) does not hold. By Lemma \ref{lem:infinitely many blocking vertices on either side}, either $\CL_L$ or $\CR_R$ is finite but not both.
		We first assume $\CR_R$ is finite and show that $\CR_L$ is infinite, thereby deducing (2). Assume for a contradiction that $\CR_L$ is finite. Then by starting $\Gamma_L$ and $\Gamma_R$ at a later time, we can assume that $\CR_R \cup \CR_L$ is empty. By Lemma \ref{lem:blocking vertices make V coding different}, there exists a physical trajectory $l$ in the open right $\epsilon'$-strip of $\Gamma_R$ and a physical trajectory $l'$ in the open right $\epsilon'$-strip of $\Gamma_L$ such that $l$ and $l'$ have the same edge coding as $\Gamma_L$ and $\Gamma_R$. Moreover, the parallel separation between $l$ and $l'$ can be chosen to be arbitrarily close to $L + \frac{\epsilon}{2}$. This contradicts the choice of $\epsilon$ in the beginning of the proof and the nonperiodicity of $\omega$. Therefore, $\CR_L$ is infinite and we have (2). 
		A similar argument shows that, if $\CL_L$ is finite, then we will have (3).
	\end{proof}

	For convenience, we make the following definition.

	\begin{definition}\label{def:uniformly recurrent vertices}
		A subset $E$ of the unfolding $Q^\infty$ is said to be \textit{uniformly recurrent} if there exists an increasing sequence $i_1 < i_2 < i_3 < \cdots$ with uniformly bounded gaps $i_{n+1} - i_n$ such that $Q^{i_n} \cap E \neq \varnothing$ for all $n \geq 1$.
	\end{definition}

	\begin{corollary}\label{cor:existence of uniform recurrent vertices}
		At least one of the following statements holds:
		\begin{enumerate}
			\item Both $\CL_L(\epsilon)$ and $\CR_R(\epsilon)$ are uniformly recurrent.
			\item Both $\CL_L(\epsilon)$ and $\CR_L(\epsilon)$ are uniformly recurrent.
			\item Both $\CL_R(\epsilon)$ and $\CR_R(\epsilon)$ are uniformly recurrent.
		\end{enumerate}
		In particular, both $\CL_L(\epsilon) \cup \CL_R(\epsilon)$ and $\CR_L(\epsilon) \cup \CR_R(\epsilon)$ are uniformly recurrent.
	\end{corollary}
	\begin{proof}
		If (1) in Proposition \ref{prop:infinite neighbouring vertices} holds, then we can pick $v_1, v_2$ and $v_3$ in $Q^\infty$, appearing in that order, such that $v_1, v_3 \in \CL_L(\epsilon/2)$ and $v_2 \in \CR_R(\epsilon / 2)$. 
		Let $u$ be the closest point on the segment $v_1v_3$ to $v_2$. By choosing $v_1$ and $v_3$ sufficiently far apart, we can assume the line segment $uv_2$ is approximately perpendicular to $\Gamma_L$ and $\Gamma_R$ so that the length of $uv_2$ is less than $\epsilon + L$. 
		Let $Q_0^M$ be the finite corridor containing these three vertices as shown in Figure \ref{fig:recurring finite corridors}. 
		By the uniform recurrence property of $\omega$, the trajectories $\Gamma_L$ and $\Gamma_R$ will pass through copies of $Q_0^M$ infinitely many times with bounded time interval. We will deduce case (1) by showing that each copy of $Q_0^M$ contains a vertex in $\CR_R(\epsilon)$ and a vertex in $\CL_L(\epsilon)$.
		Let us denote one of these copies of $Q_0^M$ by $\wt{Q_0^M}$ and denote the images of $v_i$ in $\wt{Q_0^M}$ by $\wt{v}_i$ for $i = 1,2,3$.

		By the $\mathcal{B}$-coding of $\Gamma_L$ and $\Gamma_R$, the line segment $\wt{v}_1\wt{v}_3$ and the vertex $\wt{v}_2$ will lie on two different sides of the parallel trajectories.\footnote{Note that the vertex appearing on the left of $\Gamma_L$ need not always be $\wt{v}_1$ or $\wt{v}_3$. It could also be $\wt{v}_2$ since the orientation of $Q$ is inverted after each reflection.} 
		Due to the length of $uv_2$, the vertex $\wt{v}_2$ will lie at most $\epsilon$ away from one of the two trajectories and $\wt{v}_1$ or $\wt{v}_3$ will lie at most $\epsilon$ away from the other trajectory.
		If $\wt{v}_i$ does not lie on $\Gamma_L$ or $\Gamma_R$, then we have $\wt{v}_i \in \CR_R(\epsilon)$ or $\wt{v}_i \in \CL_L(\epsilon)$.
		If $\wt{v}_i$ lies on $\Gamma_L$, then by Lemma \ref{lem:blocking vertices on parallel trajectories} there exists a vertex $\wt{w}$ in $\Gamma_L \cap \wt{Q_0^M}$ which blocks $\Gamma_L$ either from left or from right, and $\wt{w}$ is either equal to $\wt{v}_i$ or connected to $\wt{v}_i$ via a sequence of edges $e_1, e_2, \ldots$ overlapping $\Gamma_L$. 
		Clearly, these edges $e_1, e_2, \ldots$ are non-reflecting edges. We can show $\wt{w} \in \CL_L(\epsilon)$ by proving that $\wt{w}$ does not block $\Gamma_L$ from the right as follows.

		We first consider the case $v_i \in \CL_L(\epsilon)$. Choose a physical trajectory $S$ in $Q_0^M$ with the same $V$-coding as $\Gamma_L \cap Q_0^M$. 
		It follows from Lemma \ref{lem:blocking vertices make V coding different} that $v_i$ must lie on the left side of $S$.
		Let $w \in Q_0^M$ be a vertex whose image in $\wt{Q_0^M}$ is $\wt{w}$. Since $S$ does not interesect the edges connecting $w$ and $v_i$, the vertex $w$ is also on the left side of $S$. 
		Consider the image $\wt{S}$ of $S$ in $\wt{Q_0^M}$. 
		Since $\wt{w}$ lies on the left side of $\wt{S}$ and $\wt{S}$ has the same $V$-coding as $\Gamma_L \cap \wt{Q_0^M}$, Lemma \ref{lem:blocking vertices make V coding different} implies that $\wt{w}$ does not block $\Gamma_L$ from the right.
		The argument for case $v_i \in \CR_R(\epsilon)$ is similar except that $S$ is taken to have the same $V$-coding as $\Gamma_R \cap Q_0^M$. In this case, both $v_i$ and $w$ lie on the right side of $S$, and $\wt{w}$ lies on the left side of $\wt{S}$. The same conclusion follows.

		A similar argument applies if $\wt{v}_i$ lies on $\Gamma_R$ and we can show that a vertex $\wt{w} \in \wt{Q_0^M}$ lies in $\CR_R(\epsilon)$. In conclusion, each copy $\wt{Q_0^M}$ contains at least one vertex in $\CR_R(\epsilon)$ and one vertex in $\CL_L(\epsilon)$. By the uniform recurrence property, we have case (1).

		\begin{figure}[]
			\centering
			\begin{tikzpicture}[line cap=round,line join=round,>=triangle 45,x=2.0cm,y=2.0cm]
				\clip(-2.9038338645920083,-0.6156239478301437) rectangle (4.350533562567108,0.8333677302275667);
				\fill[line width=2.pt,color=uuuuuu,fill=uuuuuu,fill opacity=0.10000000149011612] (-2.7425921831777074,0.4771300359574781) -- (-1.0757989208734393,0.5910103209596329) -- (-1.1896792058755943,-0.24756268678350726) -- (-2.680475664085623,-0.2682681931475354) -- cycle;
				\fill[line width=2.pt,color=uuuuuu,fill=uuuuuu,fill opacity=0.10000000149011612] (-0.027651306492101346,0.46931937719128514) -- (1.622873983406778,0.7280369181105291) -- (1.5825135006084454,-0.11727038698971248) -- (0.09919457679178412,-0.26782857525173487) -- cycle;
				\fill[line width=2.pt,color=uuuuuu,fill=uuuuuu,fill opacity=0.10000000149011612] (2.2026651086700286,-0.02340314196801785) -- (3.8724789389615184,-0.07716438025673583) -- (3.728463532404968,0.7567618582691302) -- (2.2378888289081287,0.7237489653231244) -- cycle;
				\draw [line width=1.2pt] (-2.7425921831777074,0.4771300359574781)-- (-2.680475664085623,-0.2682681931475354);
				\draw [line width=1.2pt] (-1.0757989208734393,0.5910103209596329)-- (-1.1896792058755943,-0.24756268678350726);
				\draw [line width=0.8pt,dotted,domain=-2.72154816717285:4.350533562567108] plot(\x,{(--0.20680173159777854--0.021605076986145527*\x)/0.6589548480774363});
				\draw [line width=0.8pt,dotted,domain=-2.701491470953744:4.350533562567108] plot(\x,{(--0.0657909391972154--0.02975488744300598*\x)/0.9075240670116831});
				\draw (-2.188747062433652,-0.15458114117541769) node[anchor=north west] {$Q_0^M$};
				\draw (3.861263645300864,0.852185803968576) node[anchor=north west] {$\Gamma_L$};
				\draw (3.97417208774692,0.23118937051527147) node[anchor=north west] {$\Gamma_R$};
				\draw (-2.602744684735858,0.560505660982933) node[anchor=north west] {$v_1$};
				\draw (-1.3795698915702486,0.6640050665584837) node[anchor=north west] {$v_3$};
				\draw (-1.7088861820379126,0.02419055936416997) node[anchor=north west] {$v_2$};
				\draw (0.15410311832201598,0.6640050665584837) node[anchor=north west] {$\widetilde{v}_1$};
				\draw (0.8785989573508771,0.08064478058719765) node[anchor=north west] {$\widetilde{v}_2$};
				\draw (0.8974170310918865,0.729868324652016) node[anchor=north west] {$\widetilde{v}_3$};
				\draw (-0.5986198313183593,0.654596029687979) node[anchor=north west] {$\cdots$};
				\draw (1.8289116812718509,0.6922321771699975) node[anchor=north west] {$\cdots$};
				\draw [->,line width=0.8pt,dotted] (-1.6031467745731205,-0.3663952926505386) -- (-1.2129502079717582,-0.358939307301468);
				\draw [->,line width=0.8pt,dotted] (-2.300281404711223,-0.37971633653852777) -- (-2.672074705887499,-0.3868206671342527);
				\draw [line width=0.8pt,color=uuuuuu] (-1.0757989208734393,0.5910103209596329)-- (-1.1896792058755943,-0.24756268678350726);
				\draw [line width=0.8pt,color=uuuuuu] (-2.680475664085623,-0.2682681931475354)-- (-2.7425921831777074,0.4771300359574781);
				\draw [line width=0.8pt,color=uuuuuu] (1.622873983406778,0.7280369181105291)-- (1.5825135006084454,-0.11727038698971248);
				\draw [line width=0.8pt,color=uuuuuu] (0.09919457679178412,-0.26782857525173487)-- (-0.027651306492101346,0.46931937719128514);
				\draw [line width=0.8pt,color=uuuuuu] (3.8724789389615184,-0.07716438025673583)-- (3.728463532404968,0.7567618582691302);
				\draw [line width=0.8pt,color=uuuuuu] (2.2378888289081287,0.7237489653231244)-- (2.2026651086700286,-0.02340314196801785);
				\draw [line width=0.8pt,dotted] (-2.2228665802085334,0.3341220914319145)-- (-1.4354662715024864,0.3684185702559699);
				\draw [line width=0.8pt,dotted] (-1.7590806308863691,-0.10262414223531019)-- (-1.7789460070852172,0.3534577614428977);
				\draw (-1.878248845706997,0.6922321771699975) node[anchor=north west] {$u$};
				\draw (0.5681007406242223,-0.1451721043049131) node[anchor=north west] {$\widetilde{Q_0^M}$};
				\draw [->,line width=0.8pt,dotted] (1.170557563861104,-0.3031726795820575) -- (1.560754130462466,-0.2957166942329869);
				\draw [->,line width=0.8pt,dotted] (0.4734229337230019,-0.3164937234700467) -- (0.10162963254672658,-0.32359805406577163);
				\begin{scriptsize}
				\draw [fill=ududff] (-2.2228665802085334,0.3341220914319145) circle (1.0pt);
				\draw [fill=ududff] (-1.7590806308863691,-0.10262414223531019) circle (1.0pt);
				\draw [fill=ududff] (-1.4354662715024864,0.3684185702559699) circle (1.0pt);
				\draw [fill=ududff] (0.502001128040547,0.43186208268857457) circle (1.0pt);
				\draw [fill=ududff] (1.2467713914526635,0.4657152764800343) circle (1.0pt);
				\draw [fill=ududff] (0.8659229612987402,-0.008229436600402183) circle (1.0pt);
				\draw [fill=ududff] (2.8632613949948706,0.5164950671672239) circle (1.0pt);
				\draw [fill=ududff] (2.547664692831491,0.031961649130872814) circle (1.0pt);
				\draw [fill=ududff] (3.3758849473926182,0.11478367458698542) circle (1.0pt);
				\draw [fill=uuuuuu] (-2.72154816717285,0.2246018438991946) circle (1.0pt);
				\draw [fill=uuuuuu] (-2.701491470953744,-0.01607851073008079) circle (1.0pt);
				\draw [fill=uuuuuu] (-1.7789460070852172,0.3534577614428977) circle (1.0pt);
				\end{scriptsize}
				\end{tikzpicture}
			\caption{}
			\label{fig:recurring finite corridors}
		\end{figure}

		If (2) in Proposition \ref{prop:infinite neighbouring vertices} holds, then we can choose $M \geq 0$ and vertices $v_1, v_2$ and $v_3$ in $Q_0^M$, appearing successively in that order, such that $v_1, v_3 \in \CL_L(\epsilon/2)$ and $v_2 \in \CR_L(\epsilon / 2)$ and the line segment $uv_2$ defined above has length less than $\epsilon$. Again, let $\wt{Q_0^M}$ be one of the uniformly recurrent copies of $Q_0^M$ in the unfolding $Q^\infty$. Then, due to the length of $uv_2$, the image of $v_2$ in $\wt{Q_0^M}$ will lie at most $\epsilon$-away from one side of $\Gamma_L$ and the image of $v_1$ or $v_3$ in $\wt{Q_0^M}$ will lie at most $\epsilon$-away on the other side of $\Gamma_L$.
		We can deduce using an argument similar to case (1) above that each copy $\wt{Q_0^M}$ intersects both $\CL_L(\epsilon)$ and $\CR_L(\epsilon)$ non-trivially, and thus both $\CL_L(\epsilon)$ and $\CR_L(\epsilon)$ are uniformly recurrent.

		The case (3) is symmetric to (2) and we omit the proof.
	\end{proof}

	\paragraph{Consequences of Lemma \ref{lem:splitting vertices} and Corollary \ref{cor:existence of uniform recurrent vertices}}

	We are now able to finish the proof of the main theorem \ref{thm:main theorem} using Lemma \ref{lem:splitting vertices} and Corollary \ref{cor:existence of uniform recurrent vertices}. 
 
	By Corollary \ref{cor:existence of uniform recurrent vertices}, both $\CL_L \cup \CL_R$ and $\CR_L \cup \CR_R$ are uniformly recurrent.
	Let the time intervals of recurrence in $\CL_L \cup \CL_R$ and in $\CR_L \cup \CR_R$ be both bounded from above by $D > 0$, and let $\text{diam}(Q)$ be the diameter of $Q$. Let $z_0 = (x', \theta')$.
	Take $m > 0$ sufficiently large such that, with $q_m = (x_m, \theta_m)$,
	\[(D + 2 \ \text{diam}(Q))\tan(\theta_m - \theta') + \epsilon < \min\Big(\delta, \frac{L}{\sin(\theta_m - \theta')}\Big). \tag{$\clubsuit$} \label{eq:inequality of Delta non-periodic}\]
	The choice of $m$ is possible as $\epsilon > 0$ has been chosen to satisfy $\epsilon < \min(L /2, \delta / 2)$. Informally, the inequality \eqref{eq:inequality of Delta non-periodic} ensures that $S_m$ and $\Gamma_L$ are approximately parallel such that the distance between $S_m \cap Q_n$ and $\Gamma_L \cap Q_n$ does not increase too fast with $n$.

	Recall that $\beta(m)$ is defined as the $\mathcal{B}$-coding associated with $q_m$ and $\tau\inv(q_m)$. 
	Let $N \geq 1$ be the largest integer such that $\beta(m)_{n} = \omega_{n}$ for $n = 0, 1,2, \ldots, N-1$.
	We re-use notations introduced in Lemma \ref{lem:splitting vertices} and consider the two cases in (3) and (4) of Lemma \ref{lem:splitting vertices} separately. We will show that both cases lead to a contradiction.

	In the case where $b = b'$, there is a hole in $Q_N$ lying between $S_m$ and $\Gamma_L$ or lying between $S'_m$ and $\Gamma_R$. As illustrated by Figure \ref{fig:splitting vertex on a hole}, since the minimal diameter of the hole is strictly greater than $\delta$, there exists some point in $S_m \cap Q_N$ whose distance to $\Gamma_L$ is greater than $\delta$, which is in turn greater than $(D + 2\text{diam}(Q))\tan(\theta_m - \theta') + \epsilon$ by \eqref{eq:inequality of Delta non-periodic}. Thus, the perpendicular distance from some point in $S_m \cap Q_0^{N-1}$ to $\Gamma_L$ must be greater than $D\tan(\theta_m - \theta') + \epsilon$.
	
	Recall that both $\CL_L(\epsilon) \cup \CL_R(\epsilon)$ and $\CR_L(\epsilon) \cup \CR_R(\epsilon)$ are uniformly recurrent by Corollary \ref{cor:existence of uniform recurrent vertices}.
	Suppose $Q_N \cap S_m$ lies on the left side of $\Gamma_L$. Observe that there is an open left $\epsilon$-strip of $\Gamma_L$ of length $\geq D$ lying in $\CT_1$ and an open left $\epsilon$-strip of $\Gamma_R$ of length $\geq D$ lying in $\CT_2$ indicated by the shaded regions in Figure \ref{fig:splitting vertex on a hole}.
	Since $\CL_L(\epsilon) \cup \CL_R(\epsilon)$ is uniformly recurrent, there will be a vertex lying in the closure of the two strips. 
	In particular, this vertex belongs to one of the following four sets 
	\[\CT_1, \quad \CT_2, \quad \ol{\CT_1} \cap \Gamma_L, \quad \ol{\CT_2} \cap \Gamma_R.\]
	However, by (1) of Lemma \ref{lem:splitting vertices}, this vertex cannot lie in $\CT_1$ or $\CT_2$. On the other hand, if the vertex lies in $\ol{\CT_1} \cap \Gamma_L$ or $\ol{\CT_2} \cap \Gamma_R$, then by the definitions of $\CL_L(\epsilon)$ and $\CL_R(\epsilon)$ this vertex blocks $\Gamma_L$ or $\Gamma_R$ from the left. This contradicts (2) of Lemma \ref{lem:splitting vertices}.  
	If $Q_N \cap S_m$ lies on the right side of $\Gamma_L$ instead, then we can deduce a contradiction by applying a similar argument using the uniform recurrence of $\CR_L(\epsilon) \cup \CR_R(\epsilon)$ instead of $\CL_L(\epsilon) \cup \CL_R(\epsilon)$.
	Hence, the case $b = b'$ leads to a contradiction.

	Next, suppose $b \neq b'$. Then $S_m$ and $S'_m$ hit a different edge from $\Gamma_R$ and $\Gamma_L$ as illustrated by Figure \ref{fig:splitting vertex on an edge}. Since the base points of $f^N(q_m)$ and $f^N(\tau\inv(q_m))$ either both lie on the left side of $\Gamma_L$ or both lie on the right side of $\Gamma_R$, the perpendicular distance from some point on $S_m$ to $\Gamma_L \cap Q_N$ must be greater than $\frac{L}{\sin(\theta_m - \theta')}$, which is in turn greater than $(D + 2\text{diam}(Q))\tan(\theta_m - \theta') + \epsilon$ by \eqref{eq:inequality of Delta non-periodic}. Now, we can use exactly the argument in the previous paragraph to deduce a contradiction in the case $b \neq b'$. The proof by contradiction is complete.

	\begin{figure}
		\centering
		\begin{subfigure}[b]{\textwidth}
			\centering
			\begin{tikzpicture}[line cap=round,line join=round,>=triangle 45,x=.9cm,y=.7cm]
				\clip(-1.7513258192576415,-2.5105288841624827) rectangle (12.750619409624429,1.5246417809820763);
				\fill[line width=2.pt,color=zzttqq,fill=zzttqq,pattern=north east lines,pattern color=zzttqq] (1.610399169587111,-0.5756080874501613) -- (1.6069421513243423,-0.11519710307660608) -- (7.246073901106057,-0.07285540814141614) -- (7.25504663530899,-0.5332249774936436) -- cycle;
				\fill[line width=2.pt,color=zzttqq,fill=zzttqq,pattern=north east lines,pattern color=zzttqq] (1.712035625927518,-1.875894070867818) -- (1.708578607664749,-1.4154830864942631) -- (7.34771035744646,-1.373141391559073) -- (7.356683091649393,-1.8335109609113005) -- cycle;
				\draw [line width=0.8pt] (-1.6,0.2)-- (-1.62,-2.14);
				\draw [line width=0.8pt] (0.0972104853124314,0.46673311726718586)-- (-0.40635016085941006,-2.140991657551279);
				\draw [line width=0.8pt] (1.2302219391990747,0.4487488084753344)-- (1.7157982765790647,-2.1589759663431303);
				\draw [line width=0.8pt] (5.3,0.94)-- (5.68,-2.14);
				\draw [line width=0.8pt] (7.68,1.2)-- (7.38,-2.12);
				\draw [line width=0.8pt] (9.5,1.32)-- (9.88,-2.2);
				\draw [line width=0.8pt,dotted] (-1.6094874800213106,-0.5997847629195523)-- (11.68,-0.5);
				\draw [line width=0.8pt,dotted] (-1.6078806606286629,-0.37965050612679807)-- (11.76,0.72);
				\draw [line width=0.8pt,dotted] (-1.6176980552227287,-1.7246335655138736)-- (12.167520961478033,-0.5906525363870871);
				\draw [line width=0.8pt,dotted] (-1.6188327509876004,-1.8800868853012824)-- (11.855274718082452,-1.778915894217165);
				\draw (-1.169838343494448,0.661220983636734) node[anchor=north west] {{\scriptsize $Q_0$}};
				\draw (0.3984157578062862,0.661220983636734) node[anchor=north west] {{\scriptsize $Q_1$}};
				\draw (2.8829531542490225,0.925533472620002) node[anchor=north west] {{\scriptsize $\cdots$}};
				\draw (6.089944687246029,1.2603292919988083) node[anchor=north west] {{\scriptsize $Q_{N-1}$}};
				\draw (8.32779042730438,1.4541584505865381) node[anchor=north west] {{\scriptsize $Q_N$}};
				\draw (11.940061110075733,1.3308126223943464) node[anchor=north west] {{\scriptsize $S_m$}};
				\draw (11.552402792900272,0.34404599685681225) node[anchor=north west] {{\scriptsize $\Gamma_L$}};
				\draw (11.992923607872388,-0.448891470092992) node[anchor=north west] {{\scriptsize $S'_m$}};
				\draw (11.605265290696925,-1.6999705846137942) node[anchor=north west] {{\scriptsize $\Gamma_R$}};
				\draw [line width=0.8pt] (8.69371008781744,0.2329371029731166)-- (8.226118059229302,-0.12674907286391301);
				\draw [line width=0.8pt] (8.226118059229302,-0.12674907286391301)-- (8.8196002493604,-0.39651370474168524);
				\draw [line width=0.8pt] (8.8196002493604,-0.39651370474168524)-- (9.269207969156685,-0.09078045528021006);
				\draw [line width=0.8pt] (9.269207969156685,-0.09078045528021006)-- (9.215255042781132,0.322858646932374);
				\draw [line width=0.8pt] (9.215255042781132,0.322858646932374)-- (8.69371008781744,0.2329371029731166);
				\draw [line width=0.8pt,dotted] (9.107190984058445,-0.5193180615940957)-- (9.073166182145515,0.4989793129748561);
				\draw (9.12072789425419,0.06211267527465962) node[anchor=north west] {{\scriptsize $> \delta$}};
				\draw [->,line width=0.8pt, -{Latex[width=1.5mm]}] (5.062003666351299,-0.7171465951330066) -- (7.246642955361751,-0.702960625723848);
				\draw [->,line width=0.8pt, -{Latex[width=1.5mm]}] (3.7989951148502596,-0.7253479493635329) -- (1.602881358291464,-0.7396084283022262);
				\draw [->,line width=0.8pt, -{Latex[width=1.5mm]}] (3.7990565078081513,-2.0123097044889735) -- (1.7105140106969041,-2.020245059125668);
				\draw [->,line width=0.8pt, -{Latex[width=1.5mm]}] (5.109326846440652,-2.0073313715982746) -- (7.355525351560548,-1.9987970076831552);
				\draw (4.010686440577641,-0.43127063749410743) node[anchor=north west] {{\scriptsize $\geq D$}};
				\draw (4.010686440577641,-1.7528330824104479) node[anchor=north west] {{\scriptsize $\geq D$}};
				\end{tikzpicture}
			\caption{}
			\label{fig:splitting vertex on a hole}
		\end{subfigure}

		\begin{subfigure}[b]{\textwidth}
			\centering
			\begin{tikzpicture}[line cap=round,line join=round,>=triangle 45,x=.9cm,y=0.7cm]
				\clip(-2.038620786915046,-2.796908384845598) rectangle (13.285176750659076,2.211648008758338);
				\fill[line width=2.pt,color=zzttqq,fill=zzttqq,pattern=north east lines,pattern color=zzttqq] (1.712035625927518,-1.875894070867818) -- (1.708578607664749,-1.4154830864942631) -- (7.34771035744646,-1.373141391559073) -- (7.356683091649393,-1.8335109609113005) -- cycle;
				\fill[line width=2.pt,color=zzttqq,fill=zzttqq,pattern=north east lines,pattern color=zzttqq] (1.585118873015018,-0.5478885076410602) -- (1.5816618547522494,-0.08747752326750513) -- (7.2207936045339585,-0.04513582833231492) -- (7.2297663387368925,-0.5055053976845422) -- cycle;
				\draw [line width=0.8pt] (-1.6,0)-- (-1.62,-2.24);
				\draw [line width=0.8pt] (0.0972104853124314,0.46673311726718586)-- (-0.40635016085941006,-2.140991657551279);
				\draw [line width=0.8pt] (1.2302219391990747,0.4487488084753344)-- (1.7157982765790647,-2.1589759663431303);
				\draw [line width=0.8pt] (5.3,0.94)-- (5.68,-2.14);
				\draw [line width=0.8pt,dotted] (-1.6094874800213106,-0.5997847629195523)-- (11.68,-0.5);
				\draw [line width=0.8pt,dotted] (-1.6188327509876004,-1.8800868853012824)-- (11.855274718082452,-1.778915894217165);
				\draw (-1.2833622831176181,-1.902523314559181) node[anchor=north west] {{\scriptsize $Q_0$}};
				\draw (0.3861565147503797,-1.8230224194226103) node[anchor=north west] {{\scriptsize $Q_1$}};
				\draw (2.0755505364025204,-2.0018994334798936) node[anchor=north west] {{\scriptsize $\cdots$}};
				\draw (6.229472307288372,-1.9223985383433233) node[anchor=north west] {{\scriptsize $Q_{N-1}$}};
				\draw (8.495247818680655,-1.8627728669908956) node[anchor=north west] {{\scriptsize $Q_N$}};
				\draw (11.814410190632508,2.1122718898376247) node[anchor=north west] {{\scriptsize $S_m$}};
				\draw (11.138652581971652,-0.3323806356119152) node[anchor=north west] {{\scriptsize $\Gamma_L$}};
				\draw (12.390791680372649,0.8005071200842131) node[anchor=north west] {{\scriptsize $S'_m$}};
				\draw (11.615657952791079,-1.6640206291494695) node[anchor=north west] {{\scriptsize $\Gamma_R$}};
				\draw [->,line width=0.8pt, -{Latex[width=1.5mm]}] (3.7990565078081513,-2.0123097044889735) -- (1.7105140106969041,-2.020245059125668);
				\draw [->,line width=0.8pt, -{Latex[width=1.5mm]}] (5.109326846440652,-2.0073313715982746) -- (7.355525351560548,-1.9987970076831552);
				\draw (4.003447243464375,-1.7236463005018974) node[anchor=north west] {{\scriptsize $\geq D$}};
				\draw [line width=0.8pt,dotted] (-1.62,-2.04)-- (12.088199690187151,0.5711649120505341);
				\draw [line width=0.8pt,dotted] (-1.6099712945348557,-0.6660673512752476)-- (11.791290592026236,1.8866314426034057);
				\draw [line width=0.8pt] (7.471017736078846,1.2194824900336674)-- (7.633563801513711,-2.2404266170798666);
				\draw [line width=0.8pt] (9.421570521297228,1.6374580868661748)-- (10.977368576173795,-0.10410689993593966);
				\draw [line width=0.8pt] (10.977368576173795,-0.10410689993593966)-- (9.909208717601825,-2.2404266170798666);
				\draw (7.839365433803941,1.15826114819878) node[anchor=north west] {{\scriptsize $> \frac{L}{\sin(\theta_m - \theta')}$}};
				\draw [line width=0.8pt,dotted] (7.7722010136412045,-0.5293418986982975)-- (7.759826534188959,1.1187102180860058);
				\end{tikzpicture}
			\caption{}
			\label{fig:splitting vertex on an edge}
		\end{subfigure}
		
		\caption{}\label{fig:splitting vertices}
		\end{figure}

	\appendix

	\section*{Appendix. Proof of Lemma~\ref{lemma:extend}} 
	\label{sec:proof_of_the_lemma}

		The family of disjoint open subsets $\{E_a\}_{a \in \mathcal{A}}$ has been defined in Section~\ref{sec:partition and encoding}. We fix a metric $d$ on $E_a$ by defining $d(\rho, \rho') = |x - x'| + |\theta - \theta'|$ for $\rho = (x, \theta), \rho' = (x', \theta') \in E_a$. We need to show the uniform continuity of the billiard map $f$ on each $U_{a,b}^{i,j}$. In fact, we will prove that $f|_{U_{a,b}^{i,j}}$ is $M$-Lipschitz for some $M > 0$ depending only on the geometry of $Q$.

		For $a,b \in \mathcal{A}$ and $i,j \in I_{a,b}$, 
		fix an arbitrary $(x, \theta) \in U_{a,b}^{i,j}$ and suppose $f(x, \theta) = (y,  \phi) \in E_b$. We first calculate the coordinates of $f(x + \epsilon_1, \theta + \epsilon_2)$.

		Consider a phase point $(x - \epsilon_1, \theta)$ in $U_{a,b}^{i,j}$ parallel to $(x, \theta)$. Note that $f(x, \theta)$ and $f(x - \epsilon_1, \theta)$ land on the same edge $e_b$ and remain parallel as shown in Figure~\ref{fig:extend}. Let $f(x- \epsilon_1, \theta) = (y+\delta_1,  \phi)$ for some $\delta_1$ depending on $\epsilon_1$. By sine rule we deduce that $\delta_1 = (\sin \theta / \sin \phi) \epsilon_1$. Therefore
		\begin{equation}\label{eq:e1}
			f(x + \epsilon_1, \theta) = \Big(y - \frac{\sin\theta}{\sin \phi} \epsilon_1,  \phi\Big).
		\end{equation}
		Next, consider $(x,  \theta + \epsilon_2) \in U_{a,b}^{i,j}$, a phase point having the same base point as $(x, \theta)$ but a possibly different direction as illustrated in Figure~\ref{fig:extend_angle}. Let $d$ be the spatial distance on $Q$ between the base point $x$ and the base point $y$, we have $f(x,  \theta + \epsilon_2) = (y+\delta_2,  \phi - \epsilon_2)$ for some $\delta_2$ depending on $\epsilon_2$. By sine rule, we have $\delta_2 = d \sin(\epsilon_2) / \sin(\phi - \epsilon_2)$. Therefore
		\begin{equation}\label{eq:e2}
			f(x, \theta + \epsilon_2) = \Big(y+\frac{d \sin(\epsilon_2)}{\sin(\phi - \epsilon_2)}, \phi- \epsilon_2\Big).
		\end{equation}

		\begin{figure}
			\centering
			\begin{subfigure}{.48\textwidth}
				\centering
				\begin{tikzpicture}[line cap=round,line join=round,>=triangle 45,x=1.7cm,y=1.5cm]
					\clip(1.0320698198916618,-2.591771564547342) rectangle (4.100415570199323,1.9927723752286934);
					\draw [shift={(2.8800832876004776,-2.165341382315912)},line width=0.8pt,color=qqwuqq,fill=qqwuqq,fill opacity=0.10000000149011612] (0,0) -- (-0.32554069865930585:0.16722186148372803) arc (-0.32554069865930585:80.93786888323237:0.16722186148372803) -- cycle;
					\draw [shift={(1.6638999594398116,0.7511609459716206)},line width=0.8pt,color=qqwuqq,fill=qqwuqq,fill opacity=0.10000000149011612] (0,0) -- (0.02605765098786885:0.16722186148372803) arc (0.02605765098786885:80.93786888323237:0.16722186148372803) -- cycle;
					\draw [shift={(1.7475485079175759,1.2756156624487696)},line width=0.8pt,color=qqwuqq,fill=qqwuqq,fill opacity=0.10000000149011612] (0,0) -- (-99.06213111676763:0.16722186148372803) arc (-99.06213111676763:-18.146618793447264:0.16722186148372803) -- cycle;
					\draw [line width=1.2pt,color=wrwrwr,dotted,domain=0.5320698198916618:4.100415570199323] plot(\x,{(--5.427371755740889-0.04*\x)/7.04});
					\draw [line width=1.2pt,color=ududff] (0.7354758890917297,1.6073236738821206)-- (4.25033839421836,0.45532328344588);
					\draw [line width=1.2pt,color=ududff] (0.6168908035445706,-2.1524823341110486)-- (3.81430199149529,-2.170649443133496);
					\draw [line width=1.2pt,color=wrwrwr,dotted] (1.2002511540820606,-2.1557968815572846)-- (1.7475485079175759,1.2756156624487696);
					\draw [line width=1.2pt,color=wrwrwr,dotted] (1.7475485079175759,1.2756156624487696)-- (0.844058095047499,-0.48004137664483304);
					\draw [line width=1.2pt,color=wrwrwr,dotted] (2.8800832876004776,-2.165341382315912)-- (3.345376375050532,0.7519256677185893);
					\draw [line width=1.2pt,color=wrwrwr,dotted] (3.345376375050532,0.7519256677185893)-- (2.6336105859113372,-0.6311732562850869);
					\draw [->,line width=1.2pt,color=wrwrwr] (1.2002511540820606,-2.1557968815572846) -- (1.3052833994748942,-1.4972719282749252);
					\draw [->,line width=1.2pt,color=wrwrwr] (2.8800832876004776,-2.165341382315912) -- (2.9721059254254434,-1.5883832962733215);
					\draw [->,line width=1.2pt,color=wrwrwr] (1.7475485079175759,1.2756156624487696) -- (1.2786781722128702,0.3645096965457081);
					\draw [->,line width=1.2pt,color=wrwrwr] (3.345376375050532,0.7519256677185893) -- (2.916003212239388,-0.08242963362058553);
					\begin{scriptsize}
						\draw[color=black] (2.9327538817238765,-2.26050903744627) node {$x$};
						\draw[color=black] (3.4009750938783148,0.8916230515219989) node {$y$};
						\draw[color=black] (1.3984933026106716,-2.244523025538245) node {$x-\epsilon_1$};
						\draw[color=black] (2.0004920039520924,1.4309135548070209) node {$y+\delta_1$};
						\draw[color=black] (3.0748924639850452,-1.968203896739983) node {$\theta$};
						\draw[color=black] (1.8625339682280169,0.8598175861510671) node {$\theta$};
						\draw[color=black] (1.8708950613022033,1.0560667744894547) node {$\phi$};
					\end{scriptsize}
				\end{tikzpicture}
				\caption{}
				\label{fig:extend}
			\end{subfigure}%
			\begin{subfigure}{.48\textwidth}
				\centering
				\begin{tikzpicture}[line cap=round,line join=round,>=triangle 45,x=3.4cm,y=2.265cm]
					\clip(0.767812320914,-1.853635933589052) rectangle (2.2581213725361957,0.6625072998568945);
					\draw [shift={(1.6885694007594698,0.2450964717214209)},line width=0.8pt,color=ududff,fill=ududff,fill opacity=0.10000000149011612] (0,0) -- (-95.3251818800813:0.09561207863844716) arc (-95.3251818800813:-12.916585564613358:0.09561207863844716) -- cycle;
					\draw [shift={(1.5134265791083257,-1.633906615968953)},line width=0.8pt,color=ududff,fill=ududff,fill opacity=0.10000000149011612] (0,0) -- (-0.38570012997658126:0.09561207863844716) arc (-0.38570012997658126:84.6748181199187:0.09561207863844716) -- cycle;
					\draw [shift={(1.5134265791083257,-1.633906615968953)},line width=0.8pt,color=qqwuqq,fill=qqwuqq,fill opacity=0.10000000149011612] (0,0) -- (84.6748181199187:0.09561207863844716) arc (84.6748181199187:103.91027372005495:0.09561207863844716) -- cycle;
					\draw [shift={(1.0094922709346117,0.4008327960155045)},line width=0.8pt,color=qqwuqq,fill=qqwuqq,fill opacity=0.10000000149011612] (0,0) -- (-76.08972627994493:0.09561207863844716) arc (-76.08972627994493:-12.916585564613358:0.09561207863844716) -- cycle;
					\draw [line width=1.2pt,color=ududff] (0.5003605773597061,-1.6270868188692815)-- (2.063200202008491,-1.6376076034193845);
					\draw [line width=1.2pt,color=ududff] (0.2666186990954486,0.5711998933779006)-- (2.2053387842696597,0.12658303523394382);
					\draw [line width=1.2pt,color=wrwrwr,dotted,domain=0.54167812320914:2.2581213725361957] plot(\x,{(--0.6370580271592912-0.010520784550102968*\x)/1.562839624648785});
					\draw [line width=1.2pt,color=wrwrwr,dotted,domain=0.54167812320914:2.2581213725361957] plot(\x,{(--0.4008115527311339-0.010520784550102968*\x)/1.562839624648785});
					\draw [line width=1.2pt,color=wrwrwr,dotted] (1.5134265791083257,-1.633906615968953)-- (1.0094922709346117,0.4008327960155045);
					\draw [line width=1.2pt,color=wrwrwr,dotted] (1.5134265791083257,-1.633906615968953)-- (1.6885694007594698,0.2450964717214209);
					\draw [->,line width=1.2pt,color=wrwrwr] (1.5134265791083257,-1.633906615968953) -- (1.570557442888552,-1.020983583021103);
					\draw [->,line width=1.2pt,color=wrwrwr] (1.5134265791083257,-1.633906615968953) -- (1.3654800110060206,-1.0365416259237976);
					\begin{scriptsize}
						\draw[color=black] (1.7239863025252742,0.334275730406164) node {$y$};
						\draw[color=black] (1.1574847365924747,0.48920656605748526) node {$y+\delta_2$};
						\draw[color=black] (1.6905220750018175,-0.6600898874336858) node {$d$};
						\draw[color=black] (1.762231133980653,0.14954123604002748) node {$\phi$};
						\draw[color=black] (1.6235936199549046,-1.507965842362778) node {$\theta$};
						\draw[color=black] (1.5029584865335644,-1.3530811162772424) node {$\epsilon_2$};
						\draw[color=black] (1.1722192309586113,0.2066660323721252) node {$\phi-\epsilon_2$};
					\end{scriptsize}
				\end{tikzpicture}
				\caption{}
				\label{fig:extend_angle}
			\end{subfigure}
			\caption{Proof of Lemma \ref{lemma:extend}. The edge below is $e_a$ and the edge on top is $e_b$}
		\end{figure}

		Now consider a phase point $(x + \epsilon_1,  \theta + \epsilon_2) \in U_{a,b}^{i,j}$ whose distance from $(x, \theta)$ is $\epsilon = \epsilon_1 + \epsilon_2$.
		By combining \eqref{eq:e1} with \eqref{eq:e2}, we deduce 
		\begin{align*}
			f(x+\epsilon_1,  \theta+\epsilon_2) = \Big(y - \frac{\epsilon_1 \sin (\theta+\epsilon_2)}{\sin( \phi- \epsilon_2)} + \frac{d\sin\epsilon_2}{\sin(\phi - \epsilon_2)}, \ \ \phi - \epsilon_2\Big).
		\end{align*}
		Therefore the distance between $f(x, \theta)$ and $f(x + \epsilon_1,  \theta + \epsilon_2)$ is explicitly given by 
		\begin{equation}\label{eq:dist}
			d(f(x, \theta), f(x + \epsilon_1,  \theta + \epsilon_2)) = \left|\frac{\epsilon_1 \sin (\theta+\epsilon_2)}{\sin( \phi- \epsilon_2)} - \frac{d\sin\epsilon_2}{\sin(\phi - \epsilon_2)}\right| + |\epsilon_2|.
		\end{equation}

		The value of $d > 0$ is uniformly bounded from above by the diameter of the polygon $Q$. On the other hand, since both $f(x + \epsilon_1,  \theta + \epsilon_2)$ and $ \tau \circ f (x + \epsilon_1,  \theta + \epsilon_2)$ are on the same edge $e_b$ by the definition of $U_{a,b}^{i,j}$, the distance $|L/\sin (\phi - \epsilon_2)|$ between their base points is bounded from above by the length of $e_b$. In particular, the value of $|\sin(\phi- \epsilon_2)|^{-1}$ is bounded from above by the length of $e_b$ divided by $L$, and we have 
		\[ \left|\frac{\epsilon_1 \sin (\theta+\epsilon_2)}{\sin( \phi- \epsilon_2)}\right| + \left|\frac{d\sin\epsilon_2}{\sin(\phi - \epsilon_2)}\right| \leq M' (|\epsilon_1|+ |\epsilon_2|)\] 
		for some $M' > 0$ depending only on $L$, the length of $e_b$ and the diameter of $Q$. Therefore the distance \eqref{eq:dist} is not greater than $M'(|\epsilon_1|+ |\epsilon_2|) + |\epsilon_2|$. This proves that $f$ is $M$-Lipschitz with $M = M'+1$.
	
	\printbibliography

@article{katok1987,
author = {Anatoly Borisovich Katok},
fjournal = "Communications in Mathematical Physics",
journal = "Comm. Math. Phys.",
number = "1",
pages = "151--160",
publisher = "Springer",
title = "The growth rate for the number of singular and periodic orbits for a polygonal billiard",
url = "https://projecteuclid.org:443/euclid.cmp/1104159472",
volume = "111",
year = "1987"
}

@book{furstenberg2014recurrence,
  title={Recurrence in Ergodic Theory and Combinatorial Number Theory},
  author={Furstenberg, H.},
  isbn={9781400855162},
  series={Porter Lectures},
  url={https://books.google.fr/books?id=B8X\_AwAAQBAJ},
  year={2014},
  publisher={Princeton University Press}
}

@Article{Galperin1995,
author="Galperin, G.
and Kr{\"u}ger, T.
and Troubetzkoy, S.",
title="Local instability of orbits in polygonal and polyhedral billiards",
fjournal = "Communications in Mathematical Physics",
journal = "Comm. Math. Phys.",
year="1995",
month="May",
day="01",
volume="169",
number="3",
pages="463--473",
issn="1432-0916",
doi="10.1007/BF02099308",
url="https://doi.org/10.1007/BF02099308"
}

@article{BT11,
Author = {Jozef Bobok and Serge Troubetzkoy},
year = {2011},
month = {11},
pages = {129-144},
title = {Does a billiard orbit determine its (polygonal) table?},
volume = {212},
fjournal = "Fundamenta Mathematicae",
journal = "Fund. Math.",
doi = {10.4064/fm212-2-2}
}

@article{BT12,
	Author = {Jozef Bobok and Serge Troubetzkoy},
	Doi = {https://doi.org/10.1016/j.topol.2011.09.007},
	Issn = {0166-8641},
	fjournal = "Topology and its Applications",
	journal = "Topology Appl.",
	Number = {1},
	Pages = {236 - 247},
	Title = {Code and order in polygonal billiards},
	Url = {http://www.sciencedirect.com/science/article/pii/S016686411100397X},
	Volume = {159},
	Year = {2012}
}

@article{BT14,
title = "Homotopical rigidity of polygonal billiards",
fjournal = "Topology and its Applications",
journal = "Topology Appl.",
volume = "173",
pages = "308 - 324",
year = "2014",
issn = "0166-8641",
doi = "https://doi.org/10.1016/j.topol.2014.06.003",
url = "http://www.sciencedirect.com/science/article/pii/S0166864114002569",
Author = {Jozef Bobok and Serge Troubetzkoy}
}

@misc{duchin2018you,
      title={You can hear the shape of a billiard table: Symbolic dynamics and rigidity for flat surfaces}, 
      author={Moon Duchin and Viveka Erlandsson and Christopher J. Leininger and Chandrika Sadanand},
      year={2019},
      eprint={1804.05690},
      archivePrefix={arXiv},
      primaryClass={math.GT}
}

@misc{calderon2018hear,
      title={How to hear the shape of a billiard table}, 
      author={Aaron Calderon and Solly Coles and Diana Davis and Justin Lanier and Andre Oliveira},
      year={2018},
      eprint={1806.09644},
      archivePrefix={arXiv},
      primaryClass={math.DS}
}

@article {ZemlyakovKatok,
    AUTHOR = {Zemljakov, A. N. and Katok, A. B.},
     TITLE = {Topological transitivity of billiards in polygons},
   JOURNAL = {Mat. Zametki},
  FJOURNAL = {Akademiya Nauk SSSR. Matematicheskie Zametki},
    VOLUME = {18},
      YEAR = {1975},
    NUMBER = {2},
     PAGES = {291--300},
      ISSN = {0025-567X},
   MRCLASS = {28A65 (58F15)},
  MRNUMBER = {399423},
MRREVIEWER = {D. Newton},
}
	
\end{document}